\documentclass[12pt, reqno]{amsart}

\usepackage{mathtools}
\mathtoolsset{showonlyrefs}
\numberwithin{equation}{section}

\usepackage{graphicx}
\usepackage{hyperref}

\usepackage{mathtools}
\mathtoolsset{showonlyrefs}

\usepackage{amssymb}
\usepackage{amsthm}
\usepackage{amscd}
\usepackage{amsmath}
\usepackage{epic,eepic}
\usepackage{mathrsfs}
\usepackage{latexsym}
\usepackage{pifont}
\usepackage[all]{xy}
\usepackage {cancel}

\theoremstyle{plain}
\newtheorem{lemma}{Lemma}[section]
\newtheorem{prop}[lemma]{Proposition}
\newtheorem{thm}[lemma]{Theorem}
\newtheorem{cor}[lemma]{Corollary}
\newtheorem{intthm}{Theorem}

\theoremstyle{definition}

\newtheorem{rem}[lemma]{Remark}
\newtheorem{defi}[lemma]{Definition}
\newtheorem{exa}[lemma]{Example}

\newtheorem{problem}{Problem}

\newcommand{\bde}{\begin{defi}}
\newcommand{\ede}{\end{defi}\vspace{1mm}}
\newcommand{\ble}{\begin{lemma}}
\newcommand{\ele}{\end{lemma}}
\newcommand{\bpr}{\begin{prop}}
\newcommand{\epr}{\end{prop}}
\newcommand{\bt}{\begin{thm}}
\newcommand{\et}{\end{thm}}
\newcommand{\bco}{\begin{cor}}
\newcommand{\eco}{\end{cor}}
\newcommand{\bre}{\begin{rem}}
\newcommand{\ere}{\end{rem}}
\newcommand{\bex}{\begin{exa}}
\newcommand{\eex}{\end{exa}}
\newcommand{\bpf}{\begin{proof}}
\newcommand{\epf}{\end{proof}}

\newcommand{\mcA}{\mathcal{A}}
\newcommand{\mcB}{\mathcal{B}}
\newcommand{\mcC}{\mathcal{C}}
\newcommand{\mcD}{\mathcal{D}}
\newcommand{\mcE}{\mathcal{E}}
\newcommand{\mcF}{\mathcal{F}}
\newcommand{\mcG}{\mathcal{G}}
\newcommand{\mcH}{\mathcal{H}}

\newcommand{\mcK}{\mathcal{K}}

\newcommand{\mcM}{\mathcal{M}}

\newcommand{\mcO}{\mathcal{O}}

\newcommand{\mcQ}{\mathcal{Q}}
\newcommand{\mcR}{\mathcal{R}}
\newcommand{\mcS}{\mathcal{S}}
\newcommand{\mcT}{\mathcal{T}}

\newcommand{\mbC}{\mathbb{C}}

\newcommand{\mbE}{\mathbb{E}}

\newcommand{\mbG}{\mathbb{G}}

\newcommand{\mbN}{\mathbb{N}}

\newcommand{\mbP}{\mathbb{P}}

\newcommand{\mbR}{\mathbb{R}}

\newcommand{\mbZ}{\mathbb{Z}}

\newcommand{\mfS}{\mathfrak{S}}

\newcommand{\mfl}{\mathfrak{l}}
\newcommand{\mfm}{\mathfrak{m}}

\newcommand{\mfo}{\mathfrak{o}}

\newcommand{\mfs}{\mathfrak{s}}

\newcommand{\msC}{\mathscr{C}}
\newcommand{\msD}{\mathscr{D}}
\newcommand{\msE}{\mathscr{E}}

\newcommand{\msO}{\mathscr{O}}
\newcommand{\msP}{\mathscr{P}}

\newcommand{\msX}{\mathscr{X}}

\newcommand{\mr}{\mathrm}
\newcommand{\N}{N}
\newcommand{\M}{m}

\title[Genus formulas for dormant modular curves]{Genus formulas for dormant modular curves  and asymptotic behavior of their function fields}
\author{Kohei Aoyama}
\author{Youhei Morita}
\author{Yasuhiro Wakabayashi}

\address{\emph{Kohei Aoyama}
 \newline
 \textnormal{Department of Mathematics, Graduate School of Science, The University of Osaka, Toyonaka, Osaka 560-0043, JAPAN.
 }
 \newline
 \textnormal{\texttt{u045178a@ecs.osaka-u.ac.jp}}}

\address{\emph{Youhei Morita}
 \newline
 \textnormal{Graduate School of Information Science and Technology, The University of Osaka, Suita, Osaka 565-0871, JAPAN.}
 \newline
 \textnormal{\texttt{morita.youhei@ist.osaka-u.ac.jp}}}

 \address{\emph{Yasuhiro Wakabayashi}
 \newline
 \textnormal{Graduate School of Information Science and Technology, The University of Osaka, Suita, Osaka 565-0871, JAPAN.}
 \newline
  \textnormal{\texttt{wakabayashi@ist.osaka-u.ac.jp}}}

\begin{document}
\maketitle


\begin{abstract}
Towers of algebraic function fields over finite fields play a fundamental role in arithmetic geometry and coding theory. Classical examples arising from modular and Drinfeld modular curves exhibit asymptotically good behavior. In this paper, we introduce an analogous construction derived from the moduli spaces of higher-level dormant $\mathrm{PGL}_2$-opers of prescribed radii on $4$-pointed stable curves of genus $0$. These spaces, which we refer to as dormant modular curves, form  projective systems under level reduction. Building on previous results in the moduli theory of dormant opers, we establish an explicit  formula for  computing the  genera  of these curves. This formula allows us to study the asymptotic behavior of the corresponding towers of function fields and to compare them with the classical modular and Drinfeld modular cases.
\end{abstract}
\tableofcontents

\section{Introduction} \label{S2}

\subsection{Motivation and background} \label{SS123}

Towers of algebraic function fields over finite fields have been studied for a long time in arithmetic geometry and coding theory. 
Since the pioneering works of Tsfasman, Vl\u{a}dut, and Zink, it has become clear that asymptotic invariants, such as the ratio between the number of rational points and the genus, play a decisive role in both the theoretical understanding of curves over finite fields and their applications to algebraic-geometric codes.
One of the most striking developments in this area is the discovery that certain towers arising from modular and Drinfeld modular curves attain the so-called Drinfeld-Vl\u{a}dut bound, thereby realizing optimal asymptotic behavior for families of function fields. These examples reveal profound connections between moduli-theoretic constructions and arithmetic optimization phenomena.
In particular, Elkies' conjecture predicts that every optimal recursive tower is  modular in an appropriate  sense (cf. ~\cite{Elk}).

In the present paper, we propose and analyze a new class of such towers, constructed not from classical modular curves but from {\it moduli spaces of dormant $\mr{PGL}_2$-opers on pointed stable curves}.
Here,  a {\it $G$-oper} for a reductive group $G$ is a specific type of flat $G$-bundle on an algebraic curve,  originally introduced in the study of the geometric Langlands correspondence for
constructing Hecke eigensheaves on the moduli space of bundles via quantization of Hitchin's
integrable system (cf. ~\cite{BeDr1}, ~\cite{BeDr2}).
 For instance, 
   $\mr{GL}_n$-opers  (for $n \geq 1$) may be described as   
flat vector bundles of rank $n$ equipped with complete flags, corresponding locally  to  
  scalar differential equations of the form  
   \begin{align} \label{EQ330}
  Dy = 0, \ \ \  D = \frac{d^n}{dx^n} + a_1 \frac{d^{n-1}}{dx^{n-1}} + \cdots + a_{n-1} \frac{d}{dx} + a_n,
   \end{align}
where $x$ is 
  a local coordinate on the underlying curve  and $a_1, \cdots, a_n$ are coefficient functions.
Also, on a complex algebraic curve, 
 a $\mr{PGL}_2$-oper corresponds to  a
{\it projective structure}, which is by definition 
an atlas of coordinate charts on the associated Riemann surface into the Riemann sphere $\mbC P^1$ such that the transition
maps are M\"{o}bius transformations.

 $G$-opers in prime characteristic $p > 0$ have been studied in the context of   a characteristic-$p$ analogue  of the geometric Langlands correspondence (cf. ~\cite{BeTr}),  as well as various other topics, including 
      $p$-adic Teichm\"{u}ller theory (cf., e.g.,  ~\cite{Moc1},  \cite{Moc2},  ~\cite{JRXY}, ~\cite{JoPa},  and ~\cite{LasPa0}).
 A key common feature in  these developments is the notion of  $p$-curvature of a connection, which serves as an  
 invariant measuring  the obstruction to the compatibility of $p$-power structures that arise in certain spaces of infinitesimal symmetries.
 This invariant also plays a central role in 
the Grothendieck-Katz conjecture, which provides  a conjectural  criterion for  
 the algebraicity of solutions to linear differential equations 
  (cf. ~\cite{NKa3}, ~\cite{And}).
  
A $G$-oper is said to be  {\it dormant} if its $p$-curvature  vanishes.
For example, a $\mr{GL}_n$-oper is dormant if and only if the corresponding differential equation admits  a full set of solutions.
The moduli  theory  of dormant $G$-opers for general $G$  has been  developed in, e.g.,  ~\cite{Wak2}, ~\cite{Wak5}, \  ~\cite{Wak24}, ~\cite{Wak22},   ~\cite{Wak7},  ~\cite{Wak10}, ~\cite{Wak20}, and
  ~\cite{Wak6}.
In these works, we introduced a generalization of  dormant $\mr{PGL}_n$-opers
 using the sheaf of differential operators of finite level,  developed  by  P. Berthelot and C. Montagnon (cf. ~\cite{Ber1}, ~\cite{Ber2}, ~\cite{Mon}).
 These generalized objects are referred to as 
{\it dormant $\mr{PGL}_n^{(\N)}$-opers}
  for $\N \in \mbZ_{> 0}$ (cf.  ~\cite{Wak6}, ~\cite{Wak20}).
Within the framework of  Teichm\"{u}ller theory in characteristic $p$, 
 dormant $\mr{PGL}_2^{(\N)}$-opers may be regarded as analogues of ``nice" projective structures on hyperbolic Riemann surfaces, such as 
 those  with real monodromy, including the uniformizing projective structures (cf. ~\cite{Hos}, ~\cite{Wak25}).

\subsection{Main research focus} \label{SS12k}

Now, let $p$ be an odd prime, $k$ a finite  field of characteristic $p$,  $\N$ a positive integer.
To each pair $(g, r)$  of nonnegative integers satisfying $2g-2 +r >0$, and to each $r$-tuple $\rho := (\rho_i)_{i=1}^r$ of elements in $(\mbZ/p^\N \mbZ)^\times / \{ \pm 1 \}$,
one can associate the moduli stack
\begin{align}
\mcO p_{\N, \rho, g, r}^{^\mr{Zzz...}}
\end{align}
(cf. \eqref{EQQ100}), 
which classifies 
pairs $(\msX, \msE^\spadesuit)$ consisting of an $r$-pointed stable curve $\msX$ of genus $g$ over $k$ and a dormant $\mr{PGL}_2^{(\N)}$-oper on $\msX$ of radii $\rho$ (cf. Definition \ref{Def3} for the definition of a dormant $\mr{PGL}_2^{(\N)}$-oper).
According to ~\cite[Chapter II, Theorem 2.8]{Moc2}, ~\cite[Theorems B and C]{Wak20}, and  ~\cite[Theorem A]{Wak80},
this stack can be represented by  a (possibly empty) geometrically connected, proper, and smooth Deligne-Mumford $k$-stack of dimension $3g-3+r$.
In the special case $(g, r) = (0, 4)$,
it  defines a projective smooth curve, which we refer  to as the {\bf dormant modular curve of type $(\N, \rho, 0, 4)$}, or simply, {\bf of type $(0, 4)$}.

If we  take a compatible system of radii $(\rho_\N)_{\N \in \mbZ_{> 0}}$ with $\rho_\N \in ((\mbZ/p^\N \mbZ)^\times /\{ \pm 1 \})^4$ (i.e., $\rho \in (\mbZ_p^\times /\{ \pm 1 \})^4$),
the corresponding  curves  form  a projective system
\begin{align} \label{EQ1000}
\cdots \rightarrow \mcO p^{^\mr{Zzz...}}_{\N, \rho_\N, 0, 4} \rightarrow \cdots \rightarrow \mcO p_{2, \rho_2, 0, 4}^{^\mr{Zzz...}} \rightarrow 
\mcO p_{1, \rho_1, 0, 4}^{^\mr{Zzz...}},
\end{align}
where the transition morphisms arise from reducing the level of dormant opers.
Each curve  in this system defines  a function field over $k$, and our goal is to investigate the asymptotic growth of the associated tower of these functions fields, particularly the limiting ratios between rational points and genus.

As mentioned above, classical modular curves 
 and their variants provide a paradigmatic example of towers whose  function fields exhibit optimal asymptotic behavior.
The effectiveness of these towers lies in their moduli interpretation: each level parameterizes elliptic curves with additional level structure, and the transition morphisms correspond to natural forgetful morphisms.
In our setting,
the dormant modular curves serve as hyperbolic analogues, replacing classical level structures (or more precisely, Igusa-structures via the Abel-Jacobi map) with data coming from higher-level dormant $\mr{PGL}_2^{(\N)}$-opers (cf. ~\cite[Introduction]{Wak80} for a detailed account of this analogy).

\subsection{First result: Genus formula for dormant modular curves} \label{SS12f}

A salient feature of  dormant modular curves $\mcO p^{^\mr{Zzz...}}_{\N, \rho_\N, 0, 4}$ is that, over the boundary of the moduli stack $\overline{\mcM}_{0, 4}$ of $4$-pointed stable genus-$0$ curves, the objects correspond to {\it balanced $(p, \N)$-edge numberings} on certain  trivalent graphs (cf.  ~\cite{LiOs}, ~\cite[Section 10]{Wak20}),  described entirely in terms of edge-numberings  satisfying explicit triangle inequalities.
This combinatorial model not only renders the geometric  structure of moduli spaces accessible but also enables the computation of numerical invariants such as the genus and the number of rational points, both of which are crucial for asymptotic analysis.
The first main result of this paper gives explicit  
 genus formulas for the dormant modular curves of type $(0, 4)$, described as follows.

\begin{intthm}[cf. Theorem \ref{Thm55} for the precise statement] \label{ThA}
 Let $\rho := (\rho_i)_{i=1}^4$ be an element of $((\mbZ/p^\N \mbZ)^\times /\{ \pm 1 \})^4$.
Then, the  genus $g_{\N, \rho, 0, 4}$ of the dormant modular curve $\mcO p_{\N, \rho, 0, 4}^{^\mr{Zzz...}}$ is given by the following formula:
\begin{align} \label{EQ901}
g_{\N, \rho, 0, 4} &= 1 +  \frac{-3p^\N +1 +   \sum_{i=1}^4 \rho_i^\circledast (p^\N -\rho_i^\circledast)}{6 p^\N} \cdot \sharp (C_{\N, \rho, 0, 4})-  \sum_{\lambda \in C_{\N,\rho, 0, 4}}\frac{\lambda^\circledast (p^\N -\lambda^\circledast)}{2p^\N},
\end{align}
where $C_{\N, \rho, 0, 4}$ denotes a finite set parametrizing the boundary points of $\mcO p_{\N, \rho, 0, 4}^{^\mr{Zzz...}}$  (cf. \eqref{EQQ103}), and $\rho_i^\circledast$ is a canonical representative of the radius $\rho_i$ in $\{1, \cdots, \frac{p^\N -1}{2} \}$ (cf. the beginning of Section \ref{S3}).
Moreover, 
when  the radii $\rho$ satisfy   the inequalities $2 \cdot \delta^{-1} (\rho_i) +1 \leq \frac{p^\N -3}{4}$  for  $i=1, 2, 3, 4$ (cf. \eqref{Eq114} for the definition of $\delta$),
this expression simplifies to 
\begin{align} \label{EQ9011}
g_{\N, \rho, 0, 4} &= 1 + \frac{-3 + \sum_{i=1}^4 \rho_i^\circledast}{6}  \cdot \sharp (C_{\N, \rho, 0, 4}) - \frac{1}{2} \cdot \sum_{\lambda \in C_{\N, \rho, 0, 4}} \lambda^{\circledast}.
\end{align}
 \end{intthm}

The proof of the above theorem  relies  on detailed analysis of local 
$\mcD$-module structures of finite level  around  smooth and nodal points, Frobenius-pullback techniques, and computations of Chern characters of  various sheaves  associated to the universal family of dormant opers over $\mcO p^{^\mr{Zzz...}}_{\N, \rho, g, r}$.

\subsection{Second  result: Asymptotic behavior of dormant modular  towers} \label{SS12fd}

Having established the genus formula, we next study the asymptotic behavior of the resulting tower of function fields.
Let $\mcK := (K_\N)_{\N \in \mbZ_{>0}}$ be a sequence of functions fields forming nested field extensions $K_1 \subsetneq K_2 \subsetneq \cdots \subsetneq K_\N \subsetneq \cdots$.
Denote by $P (K_\N)$ the set of rational places in $K_\N$ and by $g_{K_\N}$ the genus of $K_\N$, i.e., the genus of the corresponding projective smooth curve.
 By analogy with classical arguments  of Ihara and Tsfasman-Vl\u{a}dut,  the tower 
$\mcK$ is said 
to be {\bf asymptotically 
$\alpha$-good} (for a nonnegative  number $\alpha$)
 if 
\begin{align}
g (K_\N)^\alpha  = O (\sharp (P (K_\N))) \ (\N \to \infty).
\end{align}
This means that the number of rational points grows at least on the order of 
$g (K_\N)^\alpha$, providing a quantitative measure of asymptotic richness comparable to that of the best classical towers.
In particular, $\mcK$ is asymptotically 
$1$-good if and only if it is asymptotically 
good, in the classical sense (cf. ~\cite[Definition 7.2.5]{Sti}).

Given a compatible system of radii $\rho := (\rho_\N)_\N$ with $\rho_\N \in ((\mbZ/p^\N \mbZ)^\times /\{ \pm 1 \})^4$ (i.e., $\rho \in (\mbZ_p^\times /\{ \pm 1 \})^4$),
we set 
\begin{align}
K_{\N, \rho_\N, 0, 4} := \text{the function field of $\mcO p_{\N, \rho_\N, 0, 4}^{^\mr{Zzz...}}$}
\end{align}
for every  $\N \in \mbZ_{> 0}$.
The sequence  $\mcK_{\N, \rho, 0, 4} := (K_{\N, \rho_\N, 0, 4})_{\N \in \mbZ_{>0}}$ forms a tower of function fields satisfying the following assertion,
 which is the second main consequence  of this paper.

\begin{intthm}[cf. Theorem  \ref{Th45}] \label{ThB}
Suppose that $\rho$ lies in $\varprojlim_{\N} \delta (D_{s, \N})$ for some integer $s$ with $1 \leq s \leq \frac{p-3}{2}$ 
(cf. \eqref{Eq4444} for the definition of $ \delta (D_{s, \N})$).
Then, 
the tower $\mcK_{\N, \rho, 0, 4}$ 
 is asymptotically $\frac{\log s}{2 \log p}$-good, that is,  
\begin{align}
g_{\N, \rho_\N, 0, 4}^{\frac{\log s}{2 \log p}}
  = O (\sharp (P (K_{\N, \rho_\N, 0,  4}))) \ (\N \to \infty).
\end{align}
 \end{intthm}

The supremum of the $\alpha$'s 
such that $\mcK_{\N, \rho, 0, 4}$ is asymptotically $\alpha$-good will allow us to compare dormant modular curves with modular curves whose associated function field towers are asymptotically optimal. It has not yet been thoroughly investigated which patterns of radii can lead to asymptotically good towers, and this remains a topic for future research.

\subsection{Relations to previous works and perspectives} \label{SSdd0}

This work situates the theory of dormant opers within the broader context of asymptotic geometry of function fields.
From a geometric viewpoint,  dormant modular curves constitute hyperbolic  counterparts of modular curves  in positive characteristic, with the level corresponding to the iteration of Frobenius pull-backs in the underlying oper structures.

The results presented here also build on and refine a sequence of earlier papers by the author, notably   ~\cite{Wak20}, ~\cite{Wak80}, and ~\cite{Wak6}, which established  foundational constructions of dormant opers of higher level.
The explicit genus formulas and  the asymptotic analysis developed in this work open several directions for further investigation:
for example, possible extensions to higher-rank groups  $\mr{PGL}_n$,  connections with 
$p$-adic Teichm\"{u}ller uniformizations (cf. ~\cite{Moc1}, ~\cite{Moc2}), 
and the study of zeta functions associated with these towers.
Ultimately, these findings  indicate that the theory of dormant opers offers a rich setting for exploring arithmetic moduli spaces whose asymptotic invariants parallel those of classical modular objects, thereby deepening the interplay between arithmetic geometry and algebraic coding theory.

Finally, recall from ~\cite{Wak5} (or ~\cite{Wak20}) that the degree of $\mcO p^{^\mr{Zzz...}}_{\N, \rho, g, r}$ over the stack $\overline{\mcM}_{g, r}$  satisfies  a factorization rule governed
by various gluing procedures of underlying stable curves, and this structure endows them with
the properties of a $2$-dimensional topological quantum field theory ($2$d TQFT). 
One of the  major goals in the study of dormant opers 
 is to understand this $2$d TQFT, as it may provide a bridge between the theory of dormant
opers and other enumerative geometries, such as the Gromov-Witten theory of Grassmannians
and the conformal field theory 
associated to the affine Lie algebras
(cf. ~\cite{Jos}, ~\cite{LiOs}, ~\cite{Wak5}).
Investigating the global geometric invariants of $\mcO p^{^\mr{Zzz...}}_{\N, \rho, g, r}$, such as  the genus computed in this paper, is expected to provide further insight into this direction.

\vspace{10mm}
\section{$\mcD$-modules of higher level} \label{S1}

In this section, we recall some basics of logarithmic $\mcD$-modules of finite level in positive characteristic,
 introduced  in  ~\cite{Mon}, and examine their  local description around   a non-smooth point in a family of pointed stable curves (cf. Proposition \ref{Pro2}).
See  ~\cite{Ber1}, ~\cite{Ber2} for the non-logarithmic setting, and  ~\cite{Wak20} for a discussion in the case where the underlying space is a pointed stable curve.

\subsection{Pointed stable curves and their logarithmic structures} \label{SS1}

Throughout this paper, we fix an odd  prime $p$, a field $k$ of characteristic $p$, and 
a pair of nonnegative integers $(g, r)$ satisfying  $2g-2+r > 0$.

We denote by $\overline{\mcM}_{g, r}$
the moduli stack classifying $r$-pointed stable curves of genus $g$ over $k$, and by $\mcM_{g, r}$ its  dense open substack classifying  smooth curves, i.e., the complement of the boundary locus 
 $\partial \overline{\mcM}_{g, r}$
 classifying singular curves.
Also, denote by 
\begin{align}
\msC_{g, r} := (f_{\mr{univ}} : \mcC_{g, r} \rightarrow \overline{\mcM}_{g, r}, \{ \sigma_{\mr{univ}, i} : \overline{\mcM}_{g, r} \rightarrow \mcC_{g, r} \}_{i=1}^r)
\end{align}
the universal family of $r$-pointed stable curves over $\overline{\mcM}_{g, r}$,
   which consists of a prestable curve 
  $f_{\mr{univ}} : \mcC_{g, r} \rightarrow \overline{\mcM}_{g, r}$ over 
  $\overline{\mcM}_{g, r}$ and a collection of mutually disjoint $r$ marked points
  $\{ \sigma_{\mr{univ}, i} : \overline{\mcM}_{g, r} \rightarrow \mcC_{g, r} \}_{i=1}^r$ of that curve.
Recall from ~\cite[Theorem 4.5]{KaFu} that  $\overline{\mcM}_{g, r}$  and $\mcC_{g, r}$ admit natural log structures; we denote the resulting fs log stack by $\overline{\mcM}_{g, r}^\mr{log}$  and $\mcC_{g, r}^\mr{log}$, respectively.
In particular, the log structure of $\overline{\mcM}_{g, r}^\mr{log}$ arises from
the normal crossing divisor defined as the boundary $\partial \overline{\mcM}_{g, r} \subseteq \overline{\mcM}_{g, r}$.
 For the basic properties on log schemes (or log stacks),
we refer the reader to  ~\cite{KaKa}, ~\cite{ILL}, and ~\cite{KaFu}.

Next, let $\msX := (f : X \rightarrow S, \{ \sigma_i \}_{i=1}^r)$ be an $r$-pointed stable curve of genus $g$ over a $k$-scheme $S$, where $\sigma_i$ denotes the $i$-th marked point $S \rightarrow X$.
 Both $S$ and $X$ are equipped with log structures pulled-back from $\overline{\mcM}_{g, r}^\mr{log}$ and $\mcC_{g, r}^\mr{log}$, respectively,  via 
 the classifying morphism $S \rightarrow \overline{\mcM}_{g, r}$ of $\msX$.
For simplicity, we set $\Omega := \Omega_{X^\mr{log}/S^\mr{log}}$ and $\mcT := \mcT_{X^\mr{log}/S^\mr{log}}$.

We write $F_S$ (resp., $F_X$) for the absolute Frobenius endomorphism of $S$ (resp., $X$).
For a positive integer $\N$, the base-change  $f^{(\N)} : X^{(\N)} \left(:=S \times_{F_S^{\N}, S} X \right) \rightarrow S$ of $f$ along the $\N$-th iterate $F_S^{\N} := F_S \circ \cdots \circ F_S$ of $F_S$ is referred  to as the {\bf $\N$-th Frobenius twist} of $X$ over $S$.
The morphism $F_{X/S}^{(\N)}  \left(:= (f, F_X^{\N}) \right) : X \rightarrow X^{(\N)}$ is called the {\bf $\N$-th relative Frobenius morphism} of $X$ over $S$.

\subsection{$\mcD$-modules of finite level} \label{SS11}

For each $\ell, \M \in \mbZ_{\geq 0}$,
we denote by
$\mcD^{(m)}_{X^\mr{log}/S^\mr{log}, \leq \ell}$ 
the sheaf of logarithmic differential operators on $X^\mr{log}/S^\mr{log}$ of level $\M$ and order $\leq \ell$, 
introduced in ~\cite[D\'{e}finition 2.3.1]{Mon}.
Also,  we set
\begin{align}
\mcD_{X^\mr{log}/S^\mr{log}}^{(m)}
 := \bigcup_{\ell \in \mbZ_{\geq 0}} \mcD^{(m)}_{X^\mr{log}/S^\mr{log}, \leq \ell}.
\end{align}
When there is no fear of confusion, we shall use the notation
 $\mcD_{\leq \ell}^{(\M)}$ (resp., $\mcD_{\leq \infty}^{(\M)}$ or  simply  $\mcD_{}^{(\M)}$)
 instead of $\mcD^{(m)}_{X^\mr{log}/S^\mr{log}, \leq \ell}$ (resp., $\mcD_{X^\mr{log}/S^\mr{log}}^{(m)}$). 

Let us take a (locally defined) logarithmic coordinate 
$x$ in $X$ relative to $S$.
To such a coordinate  $x$, one can associate 
 a  local basis $\{ \partial^{\langle  j \rangle}\}_{j \leq \ell}$ of $\mcD_{\leq \ell}^{(\M)}$ (cf. ~\cite[Section 1.2.3]{Mon}).
For any nonnegative  integers $j'$ and $j''$,
the following equality  holds:
\begin{align} \label{e364}
\partial^{\langle j' \rangle} \cdot \partial^{\langle j'' \rangle} = \sum_{j = \mr{max} \{ j', j'' \} }^{j' + j''} \frac{j!}{(j' + j'' - j)! \cdot (j - j')! \cdot  (j-j'')!} \cdot \frac{q_{j'}! \cdot q_{j''}!}{q_j !} \cdot \partial^{\langle j \rangle}
\end{align}
(cf. ~\cite[Lemme 2.3.4]{Mon}), 
where, 
for  each  $j \in \mbZ_{\geq 0}$,  let $(q_j, r_j)$ be the pair of nonnegative integers uniquely determined by the condition that $j = p^\M \cdot q_j + r_j$ and $0 \leq r_j < p^\M$.
In particular, we have  $\partial^{\langle j' \rangle} \cdot \partial^{\langle j'' \rangle} = \partial^{\langle j'' \rangle} \cdot \partial^{\langle j' \rangle}$.

Given each  $\ell \in \mbZ_{\geq 0} \sqcup \{ \infty \}$,
we shall write  ${^L}\mcD_{\leq \ell}^{(m)}$ (resp., ${^R}\mcD_{\leq \ell}^{(m)}$) for the sheaf $\mcD_{\leq \ell}^{(m)}$ endowed with a structure of $\mcO_X$-module arising from left (resp., right) multiplication by sections of $\mcD_{\leq 0}^{(m)} \left(= \mcO_X \right)$.
A {\bf (left) $\mcD^{(\M)}$-module structure} on an $\mcO_X$-module $\mcF$ is a left $\mcD^{(\M)}$-action $\nabla : {^L}\mcD^{(\M)} \rightarrow \mcE nd_{f^{-1}(\mcO_S)}(\mcF)$ on $\mcF$ extending its $\mcO_X$-module structure.
An $\mcO_X$-module equipped with a $\mcD^{(\M)}$-module structure is called a {\bf (left) $\mcD^{(\M)}$-module}.
One can define the notion of a morphism between $\mcD^{(\M)}$-modules (but we omit the detail of the precise definition).

Given a $\mcD^{(\M)}$-module $(\mcF, \nabla)$, we shall write $\mcS ol (\nabla)$ for the subsheaf of $\mcF$ on which $\mcD_+^{(\M)}$ acts as zero, where $\mcD_+^{(\M)}$ denotes the kernel of the canonical projection $\mcD^{(\M)} \twoheadrightarrow \mcO_X$.
The sheaf $\mcS ol (\nabla)$ may be regarded as an $\mcO_{X^{(\M +1)}}$-module via the underlying homeomorphism of $F_{X/S}^{(\M +1)}$.
Also, we set
\begin{align}
{^p}\psi_{(\mcF, \nabla)} := \nabla \circ {^p}\psi  : \mcT^{\otimes p^{\M+1}} \rightarrow \mcE nd_{f^{-1}(\mcO_S)} (\mcF),
\end{align} 
where  ${^p}\psi$ denotes   the $p^{\M +1}$-curvature map $\mcT^{\otimes p^{\M+1}} \rightarrow \mcD^{(\M)}$ defined in ~\cite[Definition 3.10]{Ohk}; we refer to  ${^p}\psi_{(\mcF, \nabla)}$ as 
 the {\bf $p^{\M+1}$-curvature} of $(\mcF, \nabla)$.

\subsection{Local $\mcD$-module structure around  a smooth point} \label{SS90}

We  set $U_\oslash := \mr{Spec} (k[\![t]\!])$, where $t$ denotes  a formal parameter, and 
set $\mcO_\oslash := \mcO_{U_\oslash}$ for simplicity.
We equip $U_\oslash$ with the log structure associated to the monoid morphism $\mbN \rightarrow \mcO_{U_\oslash}$ given by $n \mapsto t^n$; if $U_\oslash^\mr{log}$ denotes the resulting log scheme, then
one can define the sheaf of logarithmic differential operators on it, as in the global setting.
To be precise, 
we set
\begin{align}
 \mcD^{(\M)}_{\oslash}
   := \varprojlim_{n \geq 1} \mcD_{U_{\oslash, n}^\mr{log}/k}^{(\M)},
\end{align}
where $U_{\oslash, n}^\mr{log}$ ($n \geq  0$) denotes the strict closed subscheme of $U_\oslash^{\mr{log}}$ defined by the ideal sheaf $(t^n) \subseteq \mcO_\oslash$.
This sheaf, equipped with the $\mcO_\oslash$-module structure arising from left multiplication, 
 can be decomposed as the direct sum $\bigoplus_{j \in \mbZ_{\geq 0}} \mcO_\oslash \cdot \partial^{\langle j \rangle}$, where $\partial^{\langle j \rangle}$'s denote the sections associated to the logarithmic coordinate $t$ (cf.  ~\cite[Section 2.2]{Wak20}).

For each $d \in \mbZ/p^{\M +1}\mbZ$, 
we denote by $\widetilde{d}$ the integer defined as the unique lifting of $d$ via the natural surjection $\mbZ \twoheadrightarrow \mbZ/p^{\M +1} \mbZ$ satisfying $0 \leq \widetilde{d} < p^{\M +1}$.
There exists a unique (left) $\mcD_\oslash^{(\M)}$-module structure
\begin{align}
\nabla_{\oslash, d}^{(\M)}  : \mcD_\oslash^{(\M)} \rightarrow \mcE nd_{k} (\mcO_\oslash)
\end{align}
on $\mcO_\oslash$ determined by the condition that $\nabla_{\oslash, d}^{(\M)} (\partial^{\langle j \rangle}) (t^n) = q_j ! \cdot \binom{n - \widetilde{d}}{j} \cdot t^n$ for every $j, n \in \mbZ_{\geq 0}$.
The resulting $\mcD_\oslash^{(\M)}$-module
\begin{align}
\msO_{\oslash, d}^{(\M)} := (\mcO_\oslash, \nabla_{\oslash, d}^{(\M)})
\end{align}
is isomorphic to the unique extension of $\msO_{\oslash, 0}^{(\M)}$ to $t^{-\widetilde{d}} \cdot \mcO_\oslash \left(\supseteq \mcO_\oslash \right)$.
In particular, $\nabla_{\oslash, d}^{(\M)}$ has vanishing $p^{\M +1}$-curvature.

\subsection{Local $\mcD$-module structure around   a non-smooth point I} \label{SS8}

Next, we  set  
\begin{align}
\widetilde{T} := \mr{Spec}(k[\![t]\!]),
\hspace{10mm} \widetilde{U}_\otimes := \mr{Spec}(k[\![x, y, t]\!]/(xy -t)),
\end{align}
 where $x$, $y$, and $t$ are formal parameters.
The $k$-algebra homomorphism $k [\![t]\!] \rightarrow k[\![x, y, t]\!]/(xy -t)$ given by $t \mapsto t$ determines a $k$-morphism $\widetilde{U}_\otimes \rightarrow \widetilde{T}$.

We equip  $\widetilde{T}$ (resp., $\widetilde{U}_\otimes$)
 with the log structure associated to the monoid morphism $\mbN \rightarrow \mcO_T$  (resp., $\mbN \oplus \mbN \rightarrow \mcO_{\widetilde{U}_\otimes}$)  given by $n \mapsto  t^n$ (resp., $(n, m) \mapsto x^n y^m$); we denote by $\widetilde{T}^\mr{log}$ (resp.,  $\widetilde{U}_\oslash^\mr{log}$)  the resulting log scheme.
The diagonal embedding  $\mbN \hookrightarrow \mbN \oplus \mbN$ enables us to extend the natural morphism $\widetilde{U}_\otimes \rightarrow \widetilde{T}$  to a (log smooth) morphism of log schemes $\widetilde{U}_\otimes^\mr{log} \rightarrow \widetilde{T}^\mr{log}$.

Denote by $T$ (resp., $U_\otimes$) the closed subscheme of $\widetilde{T}$ (resp., $\widetilde{U}_\otimes$) defined by $t =0$, which is equipped with the log structure pulled-back from $\widetilde{T}^\mr{log}$ (resp., $\widetilde{U}_\otimes^\mr{log}$); the resulting log scheme will be denoted by $T^\mr{log}$ (resp., $U_\otimes^\mr{log}$).
For simplicity, we shall write $\widetilde{\mcO}_\otimes := \mcO_{\widetilde{U}_\otimes}$.

Let ``$\breve{\ \ }$"
denote either the absence or presence of ``$\widetilde{\ \ }$". 
Then, we obtain the sheaf of noncommutative rings on $\breve{U}_\otimes$ defined as 
\begin{align} \label{dE91}
 \breve{\mcD}^{(\M)}_{\otimes}
   := \varprojlim_{n \geq 1} \mcD_{\breve{U}_{\otimes, n}^\mr{log}/\breve{T}^\mr{log}}^{(\M)},
 \end{align}
where  $\breve{U}_{\otimes, n}^\mr{log}$ ($n \geq 1$) denotes the strict closed subscheme of $\breve{U}^\mr{log}_{\otimes}$ defined by 
the ideal sheaf $\mfm^n$, where $\mfm$ denotes the  maximal ideal of $\breve{\mcO}_\otimes$.
This sheaf has two 
$\breve{\mcO}_\otimes$
-module structures ${^L}\breve{\mcD}_{\otimes}^{(\M)}$ and ${^R}\breve{\mcD}_{\otimes}^{(\M)}$ arising from  those of ${^L}\mcD_{\breve{U}_{\otimes, n}^\mr{log}/T^\mr{log}}^{(\M)}$'s and ${^R}\mcD_{\breve{U}_{\otimes, n}^\mr{log}/T^\mr{log}}^{(\M)}$'s (defined as in Section \ref{SS11}), respectively.

For $z \in \{x, y \}$,
let $\partial^{\langle j \rangle}_z$'s 
denote the sections of $\mcT_{\breve{U}_\otimes^\mr{log}/\breve{T}^\mr{log}} \left(\subseteq \breve{\mcD}^{(\M)}_\otimes \right)$ associated to the logarithmic coordinate $z$.
 Then,
the $\breve{\mcO}_\otimes$-module  ${^L}\breve{\mcD}_{\otimes}^{(\M)}$ can be decomposed as the direct sum 
$\bigoplus_{j \in \mbZ_{\geq 0}} \breve{\mcO}_{\otimes} \cdot \partial^{\langle j \rangle}_z$.
By an argument similar to the proof of ~\cite[Lemma 4.1]{Wak20}, we see that
for each positive  integer $j \leq p^\M$, the following equality holds:
\begin{align} \label{YY8}
\partial_{y}^{\langle j \rangle} = (-1)^j \cdot \sum_{j'=1}^j \binom{j-1}{j'-1} \cdot \partial_x^{\langle j'\rangle}.
\end{align}
In particular, for each nonnegative  integer $a\leq \M$,
 the following equalities hold:
\begin{align} \label{e444}
\partial_{y}^{\langle p^a \rangle} =   \sum_{j'=1}^{p^a} (-1)^{j'} \cdot \partial_x^{\langle j'\rangle} = - \partial_x^{\langle p^a \rangle} -1 + \prod_{b = 0}^{a-1} (1 -(\partial_x^{\langle p^b \rangle})^{p-1}),
\end{align}
where $\prod_{b =0}^{-1} (-) := 1$.

\subsection{Local $\mcD$-module structure around   a non-smooth point II} \label{SS44}

Note that the  inclusion $k[\![t, x, y]\!]/(xy -t) \hookrightarrow  k[\![t, x]\!][\frac{1}{x}]$ (identified with the inclusion $k[\![x, \frac{t}{x}]\!] \hookrightarrow k[\![t, x, \frac{1}{x}]\!]$) induces an open immersion
$o : W \left(:= \mr{Spec}( k[\![t, x]\!][\frac{1}{x}])\right)  \hookrightarrow \widetilde{U}_\otimes$.

For an element $d \in \mbZ/p^{\M +1}\mbZ$,
 let us consider the $\mcD_{W/\widetilde{T}}^{(\M)}$-module structure $\nabla_{W, d}^{(\M)} : \mcD_{W/\widetilde{T}}^{(\M)} \rightarrow  \mcE nd_{\mcO_{\widetilde{T}}} (\mcO_W)$ on $\mcO_W$ determined uniquely by  the condition that 
$\nabla_{W, d}^{(\M)} (\partial^{\langle j \rangle}_x) (x^n) = q_j ! \cdot \binom{n -\widetilde{d}}{j}\cdot x^n$;
this preserves the subsheaf $\widetilde{\mcO}_\otimes$ via the natural inclusion   $\widetilde{\mcO}_\otimes \subseteq o_* (\mcO_{W})$.
Hence, $\nabla_{W, d}^{(\M)}$ restricts to  
 a $\widetilde{\mcD}_{\otimes}^{(\M)}$-module  structure $\widetilde{\nabla}_{\otimes, d}^{(\M)}$ on $\widetilde{\mcO}_\otimes$.
Moreover, it induces, via reduction modulo $t$,  a $\mcD_{\otimes}^{(\M)}$-module  structure $\nabla_{\otimes, d}^{(\M)}$ on $\mcO_\otimes$, which has vanishing $p^{\M +1}$-curvature.
Thus, with  the notation ``$\breve{\ \  }$" as above, we obtain a $\breve{\mcD}_{\otimes}^{(\M)}$-module
\begin{align} \label{Eq90}
\breve{\msO}_{\otimes, d}^{(\M)} := (\breve{\mcO}_\otimes, \breve{\nabla}_{\otimes, d}^{(\M)}).
\end{align}
It is immediately verified that  $\msO_{\otimes, d}^{(\M)}$ coincides with ``$\msO_{\otimes, d}^{(\M)}$" introduced in ~\cite[Section 4.3]{Wak20}.

The $(\M +1)$-st  Frobenius twist $\widetilde{U}_\otimes^{(\M +1)}$ of $\widetilde{U}_\otimes$ over $\widetilde{T}$ can be identified with the affine scheme $\mr{Spec}(k [\![t, x^{p^{\M +1}}, y^{p^{\M+1}}]\!]/(x^{p^{\M +1}}y^{p^{\M+1}}-t^{p^{\M+1}}))$ over $k$, and the $(\M +1)$-st relative Frobenius morphism $\widetilde{U}_\otimes \rightarrow \widetilde{U}_\otimes^{(\M +1)}$
corresponds to  the  injective  $k$-algebra homomorphism
\begin{align}
k [\![t, x^{p^{\M +1}}, y^{p^{\M+1}}]\!]/(x^{p^{\M +1}}y^{p^{\M+1}}-t^{p^{\M+1}}) \rightarrow k[\![t, x, y ]\!]/(xy -t).
\end{align}
For simplicity, we write $\widetilde{\mcO}_\otimes^{(\M +1)} := \mcO_{\widetilde{U}^{(\M +1)}_\otimes}$. 
The sheaf $\mcS ol (\nabla)$ associated to a $\widetilde{\mcD}_\otimes^{(\M)}$-module $(\mcF, \nabla)$ (as in Section \ref{SS11}) can be regarded as an $\widetilde{\mcO}_\otimes^{(\M)}$-module via the $(\M +1)$-st relative Frobenius morphism $\widetilde{U}_\otimes \rightarrow \widetilde{U}_\otimes^{(\M +1)}$.

Note that we have an identification
\begin{align} \label{Eq88}
H^0 (\widetilde{U}_\otimes, \widetilde{\mcO}_\otimes) = \bigoplus_{s=0}^\infty \left( k \cdot t^s \oplus \left(\bigoplus_{a = 1}^\infty  k \cdot t^s x^a  \right) \oplus \left(\bigoplus_{b = 1}^\infty k \cdot t^s y^b \right) \right),
 \end{align}
 which restricts to  an identification
 \begin{align} \label{Eq97}
 H^0 (\widetilde{U}_\otimes^{(\M)}, \widetilde{\mcO}^{(\M)}_\otimes) = \bigoplus_{s=0}^\infty \left( k \cdot t^s \oplus \left(\bigoplus_{a = 1}^\infty  k \cdot t^{s} x^{p^{\M +1} \cdot  a}  \right) \oplus \left(\bigoplus_{b = 1}^\infty k \cdot t^s y^{p^{\M +1} \cdot b} \right) \right).
 \end{align}

The following two propositions generalize  ~\cite[Propositions 4.6 and 4.7]{Wak20}.

\bpr \label{Pro1}
Let us keep the above notation, and let $d \in \mbZ/p^{\M+1}\mbZ$.
Also, we write $c := -d$.
Then, the following assertions hold:
\begin{itemize}
\item[(i)]
Under the identification \eqref{Eq88},
the following equality holds:
\begin{align}
H^0(\widetilde{U}_\otimes, \mcS ol (\widetilde{\nabla}^{(\M)}_{\otimes, d})) = 
\begin{cases}  \bigoplus\limits_{s=0}^\infty \left(k \cdot t^s \oplus\left(\bigoplus\limits_{a = 1}^\infty  k \cdot t^{s} x^{p^{\M +1} \cdot  a}  \right) \oplus \left(\bigoplus\limits_{b = 1}^\infty k \cdot t^s y^{p^{\M +1} \cdot b} \right) \right) 
  & \text{if $d = 0$}; \\
\bigoplus\limits_{s=0}^\infty \left( \left( \bigoplus\limits_{a =0}^\infty k \cdot t^s x^{\widetilde{d} + p^{\M +1}\cdot a}
\right) \oplus \left( \bigoplus\limits_{b= 0}^\infty
 k \cdot t^s y^{\widetilde{c} + p^{\M +1}\cdot b}\right)\right) & \text{if $d \neq 0$}.
 \end{cases}
\end{align}
\item[(ii)]
For simplicity, we write $F := F^{(\M +1)}_{\widetilde{U}_\otimes/\widetilde{T}}$.
Let us consider the morphism
\begin{align}
\kappa_d : \left( \widetilde{\mcO}_\otimes \otimes_{F^{-1}(\widetilde{\mcO}_\otimes^{(\M)})} F^{-1}(\mcS ol (\widetilde{\nabla}_{\otimes, d}^{(\M)}))= \right) F^*(\mcS ol (\widetilde{\nabla}_{\otimes, d}^{(\M)})) \rightarrow \widetilde{\mcO}_\otimes,
\end{align}
obtained by extending  the inclusion $\mcS ol (\widetilde{\nabla}^{(\M)}_{\otimes, d}) \hookrightarrow \widetilde{\mcO}_\otimes$.
Then,
in the case  $d = 0$, 
 the equalities  
 \begin{align} \label{EQ400}
 H^0 (\widetilde{U}_\otimes, \mr{Ker}(\kappa_d)) = H^0 (\widetilde{U}_\otimes, \mr{Coker}(\kappa_d)) = 0
 \end{align}
  hold.
On the other hand, if  $d \neq 0$, then we have
\begin{align}
& \ \ \ \ H^0 (\widetilde{U}_\otimes, \mr{Ker}(\kappa_d)) \\
& =\bigoplus\limits_{s=0}^{\mr{min}  \left\{ \widetilde{d}, \widetilde{c}\right\} -1} 
 \left(\bigoplus\limits_{a=0}^{\widetilde{c}-s-1} k \cdot t^s  (x^{\widetilde{d}+a} \otimes y^{\widetilde{c}} - t^a  y^{\widetilde{c}-a} \otimes x^{\widetilde{d}}) \oplus \bigoplus\limits_{b=1}^{\widetilde{d}-s-1} k \cdot t^s (t^b x^{\widetilde{d}-b} \otimes y^{\widetilde{c}} -  y^{\widetilde{c}+b} \otimes x^{\widetilde{d}})\right),
\end{align}
and
\begin{align}
H^0 (\widetilde{U}_\otimes, \mr{Coker}(\kappa_d))
\cong
 \bigoplus\limits_{s= 0}^{\mr{min}\{\widetilde{d}, \widetilde{c} \} -1} \left( k \cdot t^s \oplus\left( \bigoplus\limits_{a= 1}^{\widetilde{d}-s -1} k \cdot t^s x^a \right) \oplus \left( \bigoplus\limits_{b= 1}^{\widetilde{c} - s -1} k \cdot t^s y^b\right) \right).
\end{align}
In particular, we obtain the following equalities:
\begin{align}
\mr{dim}_k(H^0 (\widetilde{U}_\otimes, \mr{Ker}(\kappa_d)))  =\mr{dim}_k H^0(\widetilde{U}_\otimes, \mr{Coker}(\kappa_d))) &=  \begin{cases} 0& \text{if $d = 0$};  \\
\widetilde{d} \cdot \widetilde{c} \left(= \widetilde{d} (p^{\M +1} - \widetilde{d}) \right) &  \text{if $d\neq 0$}. \end{cases}\notag
\end{align}
\end{itemize}
\epr
\begin{proof}
For each $n$, $j \in \mbZ_{\geq 0}$, we have 
\begin{align}
\widetilde{\nabla}_{\otimes, d}^{(m)} (\partial_x^{\langle j \rangle}) (y^n) 
& = (-1)^j \cdot \sum_{j' =1}^j \binom{j-1}{j' -1} \cdot \widetilde{\nabla}_{\otimes, d}^{(m)} (\partial_y^{\langle j' \rangle}) (y^n)  \\
& = (-1)^j \cdot \sum_{j' =1}^j \binom{j-1}{j' -1} \cdot q_{j'}! \cdot \binom{n- \widetilde{c}}{j'} \cdot y^n, \notag
\end{align}
where the first equality follows from \eqref{YY8}.
It follows that $\widetilde{\nabla}_{\otimes, d}^{(m)} (\partial_x^{\langle j \rangle}) (y^n) = 0$ for every $j$ and $n$ if and only if $c = 0$ (or equivalently, $d = 0$).
Hence, assertion (i) follows from the definition of $\widetilde{\nabla}_{\otimes, d}^{(\M)}$ together with  a straightforward calculation. 

Moreover, the various computations in assertion (ii) can be immediately carried out    by applying  assertion (i), so the details are left to reader.
\end{proof}

\bpr \label{Prop233}
Let $d$ and $c$ be elements of $\mbZ/p^{\M +1}\mbZ$.
Then, the following assertions hold.
\begin{itemize}
\item[(i)]
The canonical isomorphism $\widetilde{\mcO}_\otimes \otimes_{\widetilde{\mcO}_\otimes} \widetilde{\mcO}_\otimes \xrightarrow{\sim} \widetilde{\mcO}_\otimes$ defines an isomorphism of $\widetilde{\mcD}_\otimes^{(\M)}$-modules
\begin{align}
\widetilde{\msO}_{\otimes, d}^{(\M)} \otimes \widetilde{\msO}^{(\M)}_{\otimes, c}
\xrightarrow{\sim} \widetilde{\msO}^{(\M)}_{\otimes, d + c}.
 \end{align}
In particular,  the dual $(\widetilde{\msO}^{(\M)}_{\otimes, d})^\vee$ of $\widetilde{\msO}^{(\M)}_{\otimes, d}$ is isomorphic to $\widetilde{\msO}_{\otimes, -d}^{(\M)}$.
\item[(ii)]
The following equality holds:
\begin{align} \label{Eq95}
\mr{Hom} (\widetilde{\msO}^{(\M)}_{\otimes, d}, \widetilde{\msO}^{(\M)}_{\otimes, c}) = \left\{ \mr{mult}_s \, \Big| \, s \in H^0 (\widetilde{U}_\otimes, \mcS ol (\widetilde{\nabla}_{\otimes, c-d})) \right\},
\end{align}
where $\mr{mult}_s$ denotes the endomorphism of $\widetilde{\mcO}_\otimes$ given by multiplication by $s$.
In particular, there exists
a surjective morphism $\widetilde{\msO}^{(\M)}_{\otimes, d} \twoheadrightarrow \widetilde{\msO}^{(\M)}_{\otimes, c}$ if and only if the equality $d = c$ holds.
\end{itemize}
\epr
\begin{proof}
The assertions follow from various definitions involved.
\end{proof}

\subsection{Local $\mcD$-module structure around   a non-smooth  point III} \label{SS45}
The following assertion provides a local description of $\mcD^{(\M)}$-modules with vanishing $p^{\M +1}$-curvature around a non-smooth point of a family of pointed stable curves.

\bpr \label{Pro2}
Let $(\mcF, \nabla)$ be a  $\widetilde{\mcD}_\otimes^{(\M)}$-module with vanishing $p^{\M +1}$-curvature such that $\mcF$ is a free $\widetilde{\mcO}_\otimes$-module of rank $n>0$.
Then, there exists an isomorphism of $\widetilde{\mcD}_\otimes^{(\M)}$-modules
\begin{align}
\bigoplus_{i=1}^n \widetilde{\msO}_{\otimes, d_i}^{(\M)} \xrightarrow{\sim} (\mcF, \nabla)
\end{align}
for some $d_1, \cdots, d_n \in \mbZ/p^{\M +1} \mbZ$.
Moreover, the resulting multiset
\begin{align}
e (\nabla) := [d_1, \cdots, d_n]
\end{align}
depends only on the isomorphism class of $(\mcF, \nabla)$.
\epr
\begin{proof}
Denote by $(\mcF_0, \nabla_0)$ the reduction modulo $t$  of $(\mcF, \nabla)$, which specifies a $\mcD_\otimes^{(\M)}$-module with vanishing $p^{\M +1}$-curvature.
According to ~\cite[Proposition-Definition 4.8]{Wak20}, there exists an isomorphism of $\mcD_\otimes^{(\M)}$-modules
\begin{align}
\xi : \bigoplus_{i=1}^n \msO_{\otimes, d_i}^{(\M)} \xrightarrow{\sim} (\mcF_0, \nabla_0)
\end{align}
for some $d_1, \cdots, d_n \in \mbZ/p^{\M +1} \mbZ$.

Let us choose $j \in \{ 1, \cdots, n \}$.
The tensor product of $\xi$ and the identity morphism of $\msO^{(\M)}_{\otimes, -d_j}$ gives an isomorphism
\begin{align}
\xi_j : \bigoplus_{i=1}^n \msO_{\otimes, d_i-d_j}^{(\M)} \xrightarrow{\sim} (\mcF_0, \nabla_0) \otimes \msO_{\otimes, -d_j}^{(\M)} \left(= (\mcF_0, \nabla_0 \otimes \nabla_{\otimes, -d_j}) \right).
\end{align}
If $e_j$ denotes the image of $1 \in \mcO_\otimes$ via the inclusion into the $j$-th factor $\mcO_\otimes \hookrightarrow \mcO_\otimes^{\oplus n}$,
then it is a horizontal section in the domain of $\xi_j$.
In particular, we  have $\xi_j (e_j) \in \mcS ol (\nabla_0 \otimes \nabla_{\otimes, -d_j})$.
Since the natural morphism $\mcS ol (\nabla \otimes \widetilde{\nabla}_{\otimes, -d_j}) \rightarrow \mcS ol (\nabla_0 \otimes \nabla_{\otimes, -d_j})$ is surjective by ~\cite[Proposition 2.15, (ii)]{Wak20},
we can find a section $v_j \in \mcS ol (\nabla \otimes \widetilde{\nabla}_{\otimes, -d_j})$ mapped to $\xi_j (e_j)$ via this surjection.
The section $v_j$ determines a morphism of $\widetilde{\mcD}_{\otimes}^{(\M)}$-modules $ \widetilde{\msO}_{\otimes, 0}^{(\M)} \rightarrow (\mcF, \nabla \otimes \widetilde{\nabla}_{\otimes, -d_j}^{(\M)})$.
The tensor product of this morphism and the identity morphism of $\widetilde{\msO}^{(\M)}_{\otimes, d_j}$ specifies a morphism $\zeta_j : \widetilde{\msO}_{\otimes, d_j}^{(\M)} \rightarrow (\mcF, \nabla)$.
Thus, we obtain a morphism of $\widetilde{\mcD}_\otimes^{(\M)}$-modules
\begin{align}
\zeta := \bigoplus_{j=1}^n \zeta_j :  \bigoplus_{j=1}^n \widetilde{\mcO}^{(\M)}_{\otimes, d_j} \rightarrow  (\mcF, \nabla).
\end{align}

The reduction modulo  $t$ of this morphism coincides with $\xi$, so it turns out to be an isomorphism by Nakayama's lemma.
This proves the first assertion.

Moreover, the second assertion follows immediately from Proposition \ref{Prop233}, (ii),  and thus we have finished the proof.
\end{proof}

\vspace{10mm}
\section{Moduli space of dormant $\mr{PGL}_2^{(\N)}$-opers} \label{S5}

This section discusses  dormant $\mr{PGL}_2^{(\N)}$-opers (or more generally,  dormant $\mr{PGL}_n^{(\N)}$-opers, for a general $n$) and their moduli stack.
For a detailed treatment of this subject, we refer the reader to   ~\cite{Wak20}.
In the special cases $(g, r) = (0, 4)$ and $(1, 1)$,
we explicitly describe the boundary locus of the moduli stack in terms of integers satisfying certain triangle inequalities (cf. Theorem \ref{Th3}).

\subsection{Dormant $\mr{PGL}_n^{(\N)}$-opers} \label{SS30}

Let $S$ be   a scheme  over $k$,  and
$\msX := (f: X \rightarrow S, \{ \sigma_i \}_{i=1}^r)$ an $r$-pointed stable curve  of genus $g$ over $S$, as in Sections \ref{SS1}-\ref{SS11}.
Also, 
let $\N$ and $n$ be positive integers with $1 <n < p$.

Consider a pair $\msE^\spadesuit := (\mcE_B, \phi)$ consisting of 
a $B$-bundle  $\mcE_B$ on $X$ and an $(\N -1)$-PD stratification on $\mcE := \mcE_B \times^B \mr{PGL}_n$.
 (The precise definition of a  higher-level PD stratification  on a $G$-bundle will not be recalled here, as 
the discussion below will only concern its translation into  a $\mcD^{(\N-1)}$-action on a vector bundle.
For full details, see   ~\cite[Definition 2.3]{Wak20}.)
Denote by $\nabla_\phi$  the $S^\mr{log}$-connection on $\mcE$  corresponding to the $0$-PD stratification induced by  $\phi$ (cf. ~\cite[Definition 1.17]{Wak5} for the definition of an $S^\mr{log}$-connection).

\bde[cf. ~\cite{Wak20}, Definitions 5.2, 5.3] \label{Def3}
\begin{itemize}
\item[(i)]
We say that $\msE^\spadesuit$ is a {\bf $\mr{PGL}_n^{(\N)}$-oper}  (or a {\bf $\mr{PGL}_n$-oper of level $\N$})
on $\msX$ if 
the pair $(\mcE_B, \nabla_\phi)$ forms a $\mr{PGL}_n$-oper (in other words,  an $\mfs \mfl_n$-oper) on $\msX$, in the sense of ~\cite[Definition 2.1]{Wak5}.
One can define the notion of an isomorphism between two   $\mr{PGL}_n^{(\N)}$-opers in a natural manner.
\item[(ii)]
 A $\mr{PGL}_n^{(\N)}$-oper $\msE^\spadesuit := (\mcE_B, \phi)$ is said to be {\bf dormant} if
  the $p^\N$-curvature of $\phi$ (in the sense of ~\cite[Definition 2.10, (ii)]{Wak20}) vanishes identically.
\end{itemize}
\ede

Let us recall a description of 
 $\mr{PGL}_n^{(\N)}$-opers in terms of $\mcD^{(\N -1)}$-modules.
An {\bf $n^{(N)}$-theta characteristic} of $\msX$ is defined as a pair
$\vartheta := (\varTheta, \nabla_\vartheta)$
consisting of a line bundle $\varTheta$ on $X$ and a $\mcD^{(N -1)}$-module structure $\nabla_\vartheta$ on the line bundle $\mcT^{\otimes \frac{n (n-1)}{2}} \otimes \varTheta^{\otimes n}$ (cf. ~\cite[Definition 5.12]{Wak20}).
We say that  an $n^{(N)}$-theta characteristic $\vartheta := (\varTheta, \nabla_\vartheta)$ is   {\bf dormant} if 
$\nabla_\vartheta$ has vanishing $p^\N$-curvature.

We fix  a dormant $n^{(N)}$-theta characteristic $\vartheta := (\varTheta, \nabla_\vartheta)$ of $\msX$.
(According to ~\cite[Proposition 5.14]{Wak20}, such an object always exists.)
We set
\begin{align}
\mcF_\varTheta := \mcD^{(\N -1)}_{\leq n-1} \otimes \varTheta, \hspace{5mm} \mcF_\varTheta^j := \mcD^{(\N -1)}_{\leq n  -j-1} \otimes \varTheta
 \ (j=0, \cdots, n).
\end{align}
Note   that these vector bundles do not depend on the level $\N$ because of the assumption $n < p$.
Since $\mcF_\varTheta^{j}/\mcF_\varTheta^{j+1}$ ($j=0, \cdots, n-1$)
can be identified with $\mcT^{\otimes (n-j-1)} \otimes \varTheta$, we obtain a composite isomorphism
\begin{align}
\mr{det}(\mcF_\varTheta) \xrightarrow{\sim} \bigotimes_{j=0}^{n-1} \mcF_\varTheta^j/\mcF_\varTheta^{j+1} \xrightarrow{\sim}
\bigotimes_{j=0}^{n-1} (\mcT^{\otimes (n-j-1)} \otimes \varTheta)
\xrightarrow{\sim} \mcT^{\otimes \frac{n (n-1)}{2}} \otimes \varTheta^{\otimes n}.
\end{align}

\bde[cf. ~\cite{Wak20}, Definition 5.15] \label{Def8}
A {\bf  $(\mr{GL}_n^{(N)}, \vartheta)$-oper} on $\msX$ is a $\mcD^{(N-1)}$-module structure $\nabla^\diamondsuit$ on $\mcF_\varTheta$   such that, for every $j=0, \cdots, n-1$,
the $\mcO_X$-linear morphism $\mcD^{(\N -1)} \otimes \mcF_\varTheta \rightarrow \mcF_\varTheta$ induced by $\nabla^\diamondsuit$ restricts to an isomorphism
\begin{align} \label{EQ200}
\mcD_{\leq n-j-1}^{(\N -1)} \otimes \mcF_\varTheta^{n-1} \xrightarrow{\sim} \mcF_\varTheta^j.
\end{align}
We say that  a  $(\mr{GL}_n^{(\N)}, \vartheta)$-oper $\nabla^\diamondsuit$ is  {\bf dormant} if its $p^\N$-curvature ${^p}\psi_{(\mcF_\varTheta, \nabla^\diamondsuit)}$ vanishes identically.
  \ede

 Let $\nabla^\diamondsuit$ be a  dormant  $(\mr{GL}_n^{(\N)}, \vartheta)$-oper on $\msX$.
   This  induces an $(\N -1)$-PD stratification $\phi_{\nabla^\diamondsuit}$ on the $\mr{PGL}_n$-bundle $\mcE$ obtained by projectivizing $\mcF_\varTheta$.
 Also, the filtration $\{ \mcF_\varTheta^j \}_{j=0}^n$
  defines  a $B$-reduction $\mcE_B$ of $\mcE$.
   Since 
 the morphisms \eqref{EQ200} associated to  the  various $j$'s  are isomorphisms, 
  the  pair
 \begin{align}
 \nabla^{\diamondsuit \Rightarrow \spadesuit} := (\mcE_B, \phi_{\nabla^\diamondsuit})
 \end{align}
turns out to  form  a  dormant $\mr{PGL}_n^{(\N)}$-oper.
 The resulting assignment $\nabla^\diamondsuit \mapsto \nabla^{\diamondsuit \Rightarrow \spadesuit}$ determines a well-defined bijection of sets 
\begin{align} \label{Eq107w}
\left(\begin{matrix}\text{the set of isomorphism classes of} \\ \text{dormant $(\mr{GL}_n^{(\N)}, \vartheta)$-opers on $\msX$} \end{matrix} \right)
\xrightarrow{\sim}
\left(\begin{matrix}\text{the set of isomorphism classes of} \\ \text{dormant $\mr{PGL}_n^{(\N)}$-opers on $\msX$} \end{matrix} \right)
\end{align}
(cf. ~\cite[Theorem 5.18]{Wak20}).

\subsection{Radii of dormant $\mr{PGL}_2^{(\N)}$-opers} \label{SS43}

Let us fix $i \in \{ 1, \cdots, r \}$, and take an open subscheme $U$ of $X$ meeting $\mr{Im}(\sigma_i)$, and take a section $t \in \mcO_X$ on $U$ defining the closed subscheme $\mr{Im} (\sigma_i) \cap U$ of $U$.
The $t$-adic completion  $\widehat{U}_t$ of $U$ can  be identified with $U_\oslash$.
Under this identification $\widehat{U}_t = U_\oslash$,
the restriction $\mcD^{(\N -1)} |_{\widehat{U}_t}$ of $\mcD^{(\N -1)}$ to $\widehat{U}_t$ can be identified with  $\mcD_\oslash^{(\N -1)}$.

Let $(\mcF, \nabla)$ be a $\mcD^{(\N -1)}$-module such that $\mcF$ is a rank $n$ vector bundle and $\nabla$ has vanishing $p^\N$-curvature.
Suppose that $S$ is connected.
According to ~\cite[Proposition-Definitions 4.8 and 4.13]{Wak20},
there exists  a well-defined multiset 
\begin{align}
e_i (\nabla) := [d_1, \cdots, d_n]
\end{align}
of elements in $\mbZ/p^\N \mbZ$ satisfying the following condition:
for any local section $t$ as above, the $\mcD_\oslash^{(\N -1)}$-module $(\mcF, \nabla) |_{\widehat{U}_t}$ obtained by restricting $(\mcF, \nabla)$ to $\widehat{U}_t \left(= U_\oslash \right)$ is isomorphic to $\bigoplus_{i=1}^n \msO_{\oslash, d_i}^{(\N -1)}$.
We refer to $e_i (\nabla)$ as the {\bf exponent} of $\nabla$ at $\sigma_i$ (cf. ~\cite[Definition 6.2]{Wak20}).

Next, we specialize our discussion  to the case $n =2$.
Denote by 
$(\mbZ/p^\N \mbZ)^\times /\{ \pm 1 \}$ the set of equivalence classes of elements $a \in \mbZ/p^\N \mbZ$, in which $a$ and $-a$ are identified.

Let 
$\msE^\spadesuit$ be a dormant $\mr{PGL}_2^{(\N)}$-oper on $\msX$.
This corresponds to 
a dormant $(\mr{GL}_2^{(\N)}, \vartheta)$-oper  $\nabla^\diamondsuit$
via \eqref{Eq107w}.
If 
 $e_i (\nabla^\diamondsuit) := [d_{i, 1}, d_{i, 2}]$ denotes  the exponent of  $\nabla^\diamondsuit$ at $\sigma_i$,
 then 
 the difference $d_{i, 1} -d_{i, 2}$ lies in  $(\mbZ/p^\N \mbZ)^\times$ (cf. ~\cite[Proposition 6.14]{Wak20}), and hence 
 we obtain an element
\begin{align}
\rho_i (\msE^\spadesuit) := \left(\frac{d_{i, 1}-d_{i, 2}}{2} \right) \in (\mbZ/p^\N \mbZ)^\times /\{ \pm 1 \}.
\end{align}
This element  does not depend on the choice of $\vartheta$.
We shall refer to $\rho_i (\msE^\spadesuit)$ as the {\bf radius} of $\msE^\spadesuit$ at $\sigma_i$ (cf. ~\cite[Definition 6.11]{Wak20}).
Also, for $\rho := (\rho_i)_{i=1}^r \in ((\mbZ/p^\N \mbZ)^\times /\{ \pm \})^r$, we say that $\msE^\spadesuit$ is {\bf of radii $\rho$}
if $\rho_i = \rho_i (\msE^\spadesuit)$ for every $i=1, \cdots, r$.
When $r =0$,
any dormant $\mr{PGL}_2^{(\N)}$-oper is said to be {\bf of radii $\emptyset$}.

\subsection{Dormant modular curves} \label{SS27}

For an element $\rho \in ((\mbZ/p^\N \mbZ)^\times /\{ \pm 1 \})^{r}$, we shall denote by
\begin{align} \label{EQQ100}
\mcO p_{\N, \rho, g, r}^{^\mr{Zzz...}}
\end{align}
the moduli stack classifying pairs $(\msX, \msE^\spadesuit)$ consisting of an $r$-pointed stable curve of genus $g$ over $k$ and a dormant $\mr{PGL}_2^{(\N)}$-oper $\msE^\spadesuit$ on $\msX$ of radii $\rho$.
It follows from  ~\cite[Theorems B and C]{Wak20} and  ~\cite[Theorem A]{Wak80} that this stack can  be represented by a (possibly empty) geometrically connected, smooth, and proper Deligne-Mumford stack  over $k$ of dimension $3g-3 +r$.
Moreover, 
the projection
\begin{align} \label{Eq34}
\Pi_{N, \rho, g, r} : \mcO p^{^\mr{Zzz...}}_{\N, \rho, g, r} \rightarrow \overline{\mcM}_{g, r}
\end{align}
 obtained by forgetting the data of dormant $\mr{PGL}_2$-opers is finite, faithfully flat,  and \'{e}tale over the points of $\overline{\mcM}_{g, r}$ classifying totally degenerate curves (cf.  ~\cite[Definition 6.24]{Wak20} for the definition of a totally degenerate curve).
In particular, when $(g, r)$ is either $(0, 4)$ or $(1, 1)$,
the Deligne-Mumford stack $\mcO p^{^\mr{Zzz...}}_{\N, \rho, g, r}$  is of dimension $1$.

\bde \label{Def891}
We shall refer to $\mcO p_{\N, \rho, 0, 4}^{^\mr{Zzz...}}$  (resp., $\mcO p_{\N, \rho, 1, 1}^{^\mr{Zzz...}}$) as the {\bf dormant modular curve of type $(\N, \rho, 0, 4)$} (resp., {\bf  of type $(\N, \rho, 1, 1)$}).
\ede

Next, 
let $\rho := (\rho_i)_{i=1}^r$ be an element of $(\mbZ_p^{\times}/\{ \pm 1 \})^r$, where $\mbZ_p := \varprojlim_{\N} \mbZ /p^\N \mbZ$.
For each $\N \in \mbZ_{> 0}$, we denote by $\rho_{i, \N}$ the image of $\rho_i$ via  the natural projection $\mbZ_p^{\times}/\{ \pm 1 \} \rightarrow  (\mbZ/p^\N \mbZ)^{\times}/\{ \pm 1 \}$ and write $\rho_\N := (\rho_{i, \N})_{i=1}^r$.
Then, 
reducing the level of dormant $\mr{PGL}_2^{(\N)}$-opers (for various $\N$'s) yields
 a projective system
\begin{align} \label{Eq12}
\cdots \rightarrow \mcO p^{^\mr{Zzz...}}_{N, \rho_\N, g, r} \rightarrow \cdots \rightarrow  \mcO p^{^\mr{Zzz...}}_{2, \rho_2, g, r} \rightarrow  \mcO p^{^\mr{Zzz...}}_{1, \rho_1, g, r}
\end{align}
each of whose morphisms is 
finite, faithfully flat, and generically \'{e}tale.

\subsection{Combinatorial description of the divisor at infinity} \label{SS39}

We set 
\begin{align} \label{Eq11722}
B_\N := \left\{a\in \mbZ\, \Biggl| \,  0\leq a \leq \frac{p^\N -3}{2}, a \not\equiv \frac{p-1}{2} \ (\mr{mod}\, p) \right\}.
\end{align}
The 
assignment $s \mapsto \overline{\left( \frac{2s +1}{2}\right)}$ determines a bijection
\begin{align}  \label{Eq114}
\delta_\N \left(\text{or simply} \ \delta \right) : B_\N \xrightarrow{\sim} (\mbZ/p^\N \mbZ)^\times /\{ \pm 1 \}.
\end{align}

Let us take a quadruple $\lambda := (\lambda_1, \lambda_2, \lambda_3, \lambda_4)$ of elements of $B_\N$.
Denote by 
$B_{\N, \lambda, 0, 4}$
 the subset of $B_\N$ consisting of elements $\eta$ satisfying the following two conditions:
\begin{itemize}
\item
$|\lambda_1 - \lambda_2|\leq \eta \leq \mr{min} \left\{ \lambda_1 + \lambda_2,  p^\N -2 -\lambda_1 - \lambda_2\right\}$;
\item
$ |\lambda_3 - \lambda_4 |\leq \eta \leq \mr{min} \left\{ \lambda_3 + \lambda_4, p^\N -2 -\lambda_3 - \lambda_4\right\}$;
\item
$| [\lambda_1]_{\N'} - [\lambda_2]_{\N'}| \leq [\eta]_{\N'} \leq \mr{min} \left\{ [\lambda_1]_{\N'} + [\lambda_2]_{\N'},  p^{\N'} -2 -[\lambda_1]_{\N'} - [\lambda_2]_{\N'} \right\}$ for any positive integer  $\N' < \N$;
\item
$|[\lambda_3]_{\N'} - [\lambda_4]_{\N'} |\leq [\eta]_{\N'} \leq \mr{min} \left\{  [\lambda_3]_{\N'} + [\lambda_4]_{\N'},  p^{\N'} -2 -[\lambda_3]_{\N'}  -[\lambda_4]_{\N'}\right\}$ for any positive integer  $\N' < \N$.
\end{itemize}
Here, 
for each positive integer $M$ and a nonnegative integer $a$, we set
\begin{align}
[a]_M := \frac{p^M -1}{2} - \left| a - p^M \cdot \left\lfloor \frac{a}{p^M}\right\rfloor -\frac{p^M -1}{2}\right|.
\end{align}
Given a quadruple $\rho := (\rho_i)_{i=1}^4 \in ((\mbZ/p^\N \mbZ)^\times /\{ \pm 1 \})^4$, 
we set
\begin{align}
C_{\N, \rho, 0, 4}^0 &:= \delta (B_{\N, (\lambda_1, \lambda_4, \lambda_2, \lambda_3), 0, 4}), \\
C_{\N, \rho, 0, 4}^1 &:= \delta (B_{\N, (\lambda_1, \lambda_3, \lambda_2, \lambda_4), 0, 4}), \notag \\
\hspace{5mm}
C_{\N, \rho, 0, 4}^\infty &:= \delta (B_{\N, (\lambda_1, \lambda_2, \lambda_3, \lambda_4), 0, 4}), \notag
\end{align}
 where $\lambda_i := \delta^{-1}(\rho_i)$.
 The disjoint union of $C_{\N, \rho, 0, 4}^q$'s ($q \in \{0,1, \infty \}$) is denoted by $C_{\N, \rho, 0, 4}$, i.e., we have 
  \begin{align} \label{EQQ103}
C_{\N, \rho, 0, 4} := \left\{ (q, \lambda) \, \big| \, q \in \{0,1, \infty \}, \lambda \in C_{\N, \rho, 0, 4}^q\right\}.
\end{align}

Next, for  $\lambda \in B_\N$, 
 we shall set $B_{\N, \lambda, 1, 1}$ to be the subset of $B_{\N}$ consisting of elements $\eta$ satisfying the following two conditions:
\begin{itemize}
\item
$0\leq \eta \leq \mr{min} \left\{ 2 \lambda, p^\N -2 \lambda -2 \right\}$;
\item
$0\leq [\eta]_{\N'} \leq \mr{min} \left\{ 2 [\lambda]_{\N'}, p^{\N'} -2 [\lambda]_{\N'} -2 \right\}$ for any positive integer  $\N' < \N$.
\end{itemize}
For an element $\rho$ of $(\mbZ/p^\N \mbZ)^\times /\{ \pm 1\}$,  we obtain a set 
\begin{align}
C_{\N, \rho, 1, 1} := \delta (B_{\N, \rho, 1, 1}) \left(\subseteq (\mbZ/p^\N \mbZ)^\times/\{ \pm 1\} \right),
\end{align}
where $\lambda := \delta^{-1} (\rho)$.

Then, we can prove the following assertion.

\bt \label{Th3}
Suppose that $k$ is algebraically closed, and that  $(g, r) = (1, 1)$ or $(0, 4)$.
Let $\rho$ be an element of $((\mbZ/p^\N \mbZ)^\times /\{ \pm 1\})^r$.
Then, the following assertions hold.
\begin{itemize}
\item[(i)]
We set $\partial \mcO p^{^\mr{Zzz...}}_{\N, \rho, g, r} := \Pi_{\N, \rho, g, r}^{-1} (\partial \overline{\mcM}_{g, r})$, which forms a $0$-dimensional reduced stack  over $k$.
Then, there exists a canonical  bijection of sets
\begin{align} \label{Eq104}
\partial \mcO p^{^\mr{Zzz...}}_{\N, \rho, g, r} (k) \xrightarrow{\sim} C_{\N, \rho, g, r}.
\end{align}
\item[(ii)]
The stack $\mcO p_{\N, \rho, g, r}^{^\mr{Zzz...}}$ is nonempty if and only if
$C_{\N, \rho, g, r} \neq \emptyset$.
Moreover, if $\mcO p_{\N, \rho, g, r}^{^\mr{Zzz...}} \neq \emptyset$ and $(g, r) = (0, 4)$ (resp., and $(g, r) = (1, 1)$), then  the degree of the projection 
$\Pi_{\N, \rho, g, r}$
 coincides with $ \sharp (C^q_{\N, \rho, g, r})$ for any $q \in \{0,1,\infty \}$ (resp., $\sharp (C_{\N, \rho, g, r})$).
 In particular, when $(g, r) = (0, 4)$, we have 
\begin{align}
\frac{1}{3} \cdot \sharp (C_{\N, \rho, 0, 4}) = \sharp (C^0_{\N, \rho, 0, 4})  = \sharp (C^1_{\N, \rho, 0, 4}) = \sharp (C^\infty_{\N, \rho, 0, 4}).
\end{align}
 \end{itemize}
\et
\begin{proof}
For each pointed curve $\msX$ classified by $\partial \overline{\mcM}_{g, r} (k)$,
we denote by $\mr{Op}_{\N, \rho, \msX}^{^\mr{Zzz...}}$ the set of isomorphism classes of dormant $\mr{PGL}_2^{(\N)}$-opers on $\msX$ of radii $\rho$.

First, we shall prove assertion (i).
Note that 
 $\partial \mcO p^{^\mr{Zzz...}}_{\N, \rho, g, r} (k)$ is in bijection with the set of  dormant $\mr{PGL}_2^{(\N)}$-opers of radii $\rho$ on $r$-pointed totally degenerate curves over $k$ of genus $g$.
Let $\mbG := (G, \{ \lambda_j \}_{j=1}^r)$ denote  
 the trivalent clutching data of type $(g, r)$ (in the sense of ~\cite[Definition 6.6]{Wak20}),
   where $G := (V, E, \zeta)$ denotes the underlying semi-graph,  defined as follows:
\begin{itemize}
 \item
 Case of $(g, r) = (0, 4)$:
 \begin{itemize}
 \item
 The vertex set  of $G$ is  given by $V := \{v_1, v_2 \}$, and the edges are $E:= \{ e_1, e_2, e_3, e_4, e_5\}$ such that $e_{\ell} := \{ e_{\ell, 1}, e_{\ell, 2} \}$ ($\ell = 1, \cdots, 5$), $\zeta (e_{\ell, 1}) = \circledcirc$ for $\ell =1, 2,3,4$ (where $\circledcirc$ denotes an abstract symbol), $\zeta (e_{\ell, 2}) = v_1$ (resp., $\zeta (e_{\ell, 2}) = v_2$) for $\ell = 1, 2$ (resp., $\ell = 3, 4$)  and $\zeta (e_{5, j}) = v_j$ for $j =1, 2$.
 \item
 The bijections $\lambda_j : \zeta^{-1}(\{ v_j \}) \xrightarrow{\sim} \{1, 2, 3\}$ ($j =1, 2$) are given by $\lambda_1 (e_{1, 2}) = \lambda_2 (e_{3, 2})   =1$,  $\lambda_1 (e_{2, 2}) = \lambda_2 (e_{4, 2})  =2$, and
 $\lambda_1 (e_{5, 1}) = \lambda_2 (e_{5, 2}) =3$.
 \end{itemize}
\item
Case of $(g, r) = (1, 1)$:
 \begin{itemize}
 \item
 The vertex set of $G$ is  given by $V :=\{ v_1 \}$,  and the edges are $E := \{e_1, e_2 \}$ such that $e_\ell := \{ e_{\ell, 1}, e_{\ell, 2}\}$ ($\ell =1, 2$),  $\zeta (e_{1, 1}) = v_1$, $\zeta (e_{1, 2}) = \circledcirc$ (where $\circledcirc$ denotes an abstract symbol) and $\zeta (e_{2, j}) =  v_1$ for $j=1, 2$.
 \item
 The bijection $\lambda_1 : \zeta^{-1}( \{ v_1\}) \xrightarrow{\sim} \{ 1, 2, 3 \}$ is given by $\lambda_1 (e_{1, 1}) = 1$, $\lambda_1 (e_{2, 1}) = 2$ and $\lambda_1 (e_{2, 2}) = 3$.
 \end{itemize}
\end{itemize}
As introduced in  ~\cite[Definition 10.17]{Wak20}, a balanced $(p, \N)$-edge numberings of radii $\rho$ is defined as  a certain collection of nonnegative integers $(a_e)_{e \in E}$ indexed by edges of $V$ (i.e., elements of $E$).
In particular, we obtain the set $\mr{Ed}_{p, \N, \mbG, \rho}$ of balanced $(p, \N)$-edge numberings on $\mbG$ of radii $\rho$.
If $(g, r) = (0, 4)$ (resp., $(g, r) = (1, 1)$),
then elements $(a_e)_e$ of $\mr{Ed}_{p, \N, \mbG, \rho}$ can  be uniquely determined by  the integer $a_{e_1}$; moreover, it follows from the various definitions involved  that  the assignment $(a_{e})_e \mapsto a_{e_1}$ defines a bijection 
\begin{align} \label{Eq91}
\mr{Ed}_{p, \N, \mbG, \rho} \xrightarrow{\sim} B_{\N, (\lambda_1, \lambda_2, \lambda_3, \lambda_4), 0, 4} \  \left(\text{resp.,} \  \mr{Ed}_{p, \N, \mbG, \rho} \xrightarrow{\sim} B_{\N, \lambda, 1, 1}\right),
\end{align}
where $\rho := (\rho_i)_{i=1}^4$ and $\lambda_i := \delta^{-1} (\rho_i)$ (resp., $\lambda := \delta^{-1} (\rho)$).

Now, let us first consider the case  $(g, r) = (0, 4)$.
If $\msX_{0, 4}$ denotes the totally degenerate curve over $k$ induced by $\mbG$ (cf. ~\cite[Definition 6.25]{Wak20}), then
any object in  $\partial\overline{\mcM}_{0, 4} (k)$  is  given by 
the pointed stable curve 
$\msX^\varsigma_{0, 4}$ (cf. ~\cite[Remark 6.16]{Wak20})
obtained from  $\msX_{0, 4}$  by changing the order of its marked points via the action of $\varsigma \in  \mfS := \{ (1), (2 \ 3), (2 \  4 \  3)\} \subseteq \mfS_4$
 (:= the symmetric group of $4$ letters).
By composing \eqref{Eq91} with the bijection resulting from ~\cite[Proposition 10.18]{Wak20} for the clutching data $\mbG$, 
we obtain a bijection 
$\mr{Op}^{^\mr{Zzz...}}_{\N, \rho, \msX_{0, 4}} \xrightarrow{\sim} B_{\N, (\lambda_1, \lambda_2, \lambda_3, \lambda_4), 0, 4}$.
After changing the order of the marked points, this bijection becomes a bijection
$\mr{Op}^{^\mr{Zzz...}}_{\N, \rho, \msX^\varsigma_{0, 4}} \xrightarrow{\sim} B_{\lambda_{\varsigma (1)}, \lambda_{\varsigma (2)}, \lambda_{\varsigma (3)}, \lambda_{\varsigma (4)}}$.
Thus, we obtain  a composite bijection
\begin{align}
\partial \mcO p^{^\mr{Zzz...}}_{\N, \rho, 0, 4} (k)\left(= \coprod_{\varsigma \in \mfS} \mr{Op}_{\N, \rho, \msX_{0, 4}^\varsigma}^{^\mr{Zzz...}} \right) &\xrightarrow{\sim} \coprod_{\varsigma \in\mfS}B_{\N, (\lambda_{\varsigma (1)}, \lambda_{\varsigma (2)}, \lambda_{\varsigma (3)}, \lambda_{\varsigma (4)}), 0, 4} \\
& \xrightarrow{\sim}C_{\N, \rho, 0, 4}^0 \sqcup C_{\N, \rho, 0, 4}^1 \sqcup C_{\N, \rho, 0, 4}^\infty \\
& \xrightarrow{\sim} C_{\N, \rho, 0, 4}.
\end{align}
This completes the proof of assertion (i) in the case $(g, r) = (0, 4)$.

Next, suppose that   $(g, r) = (1, 1)$.
Since  the unique (up to isomorphism) totally degenerate curve $\msX_{1, 1}$ in $\overline{\mcM}_{1, 1}$  corresponds to $\mbG$,
it follows from   ~\cite[Proposition 10.18]{Wak20} that 
taking radii at both  the marked point and the  nodal point on $\msX_{1, 1}$ yields a bijection of sets $\left(\partial \mcO p^{^\mr{Zzz...}}_{\N, \rho, 1, 1} (k) =  \right)\mr{Op}_{\N, \rho, \msX_{1, 1}}^{^\mr{Zzz...}} \xrightarrow{\sim} \mr{Ed}_{p, \N, \mbG, \rho}$.
Thus, we obtain a composite bijection
\begin{align}
\partial \mcO p^{^\mr{Zzz...}}_{\N, \rho, 1, 1} (k) \xrightarrow{\sim}
\mr{Ed}_{p, \N, \mbG, \rho} \xrightarrow{\eqref{Eq91}} B_{\N, \lambda, 1, 1} \xrightarrow{\delta} C_{\N, \rho, 1, 1}, 
\end{align}
completing the proof of assertion (i) in the case $(g, r) = (1, 1)$.

Next, let us consider assertion (ii).
The ``if" part of the first assertion is clear.
To prove the inverse implication,
we suppose that there exists a point $q$ in $\mcO p^{^\mr{Zzz...}}_{\N, \rho, g, r}$.
Since $\Pi_{\N, \rho, g, r}$ is finite and faithfully flat, the irreducivility of $\overline{\mcM}_{g, r}$ implies the surjectivity of $\Pi_{\N, \rho, g, r}$.
In particular, the inverse image of $\Pi_{\N, \rho, g, r}$ over any point of $\partial \overline{\mcM}_{g, r}$ is nonempty.
This prove the ``only if" part of the first assertion.
Moreover, since $\Pi_{\N, \rho, g, r}$ is \'{e}tale over the point classifying totally degenerate curves, the degree $\mr{deg}(\Pi_{\N, \rho, g, r})$ coincides with $\sharp (\mr{Op}^{^\mr{Zzz...}}_{\N, \rho, \msX})$ for $\msX \in \partial \overline{\mcM}_{g, r}$.
It follows that  the construction of \eqref{Eq104} implies  the second assertion.

Finally,  the third assertion follows from  the non-resp'd portion of the second assertion.
\end{proof}

\bco \label{Cor34}
Suppose that $(g, r) = (0, 4)$ or $(1, 1)$ and $\mcO p_{\N, \rho, g, r}^{^\mr{Zzz...}} \neq \emptyset$.
Then,
for any $\rho \in ((\mbZ/p^\N \mbZ)^\times /\{ \pm 1 \})^r$,
the degree of the projection $\Pi_{\N, \rho, g, r}$ satisfies the inequality
\begin{align} \label{EQ908}
\mr{deg} (\Pi_{\N, \rho, g, r}) \leq \frac{(p-1)p^{\N -1}}{2} \left( = \sharp ((\mbZ/p^\N \mbZ)^\times /\{ \pm 1 \})\right).
\end{align}
\eco
\begin{proof}
The assertion follows  from  Theorem \ref{Th3}, (ii).
\end{proof}

Moreover, one can prove the following assertion.

\bpr \label{Prop2}
Suppose that $(g, r) = (0, 4)$ or $(1, 1)$, and 
let $\rho := (\rho_i)_{i=1}^r$ be an element of 
$((\mbZ/p^\N \mbZ)^\times / \{ \pm 1\})^r$.
Also, let $\N'$ be
 a positive integer  with $\N \leq \N'$,
 and write $\rho_{\N'} := (\delta_{\N'} (\delta_\N^{-1} (\rho_i)))_{i=1}^r \in ((\mbZ/p^{\N'} \mbZ)^\times / \{ \pm 1\})^r$.
 (In particular, the image of $\rho_{\N'}$ under  the natural quotient $((\mbZ/p^{\N'} \mbZ)^\times / \{ \pm 1\})^r \twoheadrightarrow  ((\mbZ/p^{\N} \mbZ)^\times / \{ \pm 1\})^r$ coincides with $\rho$.)
Then,  the projection
\begin{align} \label{Eq201}
\mcO p^{^\mr{Zzz...}}_{\N', \rho_{\N'}, g, r} \rightarrow \mcO p^{^\mr{Zzz...}}_{\N, \rho, g, r}
\end{align}
obtained by reducing the level of dormant $\mr{PGL}_2^{(\N')}$-opers is an isomorphism.
\epr
\begin{proof}
Note that 
the map $C_{\N', \rho_{\N'}, g, r} \rightarrow C_{\N, \rho, g, r}$ obtained by restricting the natural quotient $((\mbZ/p^{\N'} \mbZ)^\times / \{ \pm 1\})^r \twoheadrightarrow  ((\mbZ/p^{\N} \mbZ)^\times / \{ \pm 1\})^r$   is bijective.
Hence, by  Theorem \ref{Th3}, (ii),  the equality $\mr{deg}(\Pi_{\N', \rho_{\N'}, g, r}) = \mr{deg}(\Pi_{\N, \rho, g, r})$ holds.
This implies that \eqref{Eq201} is of degree $1$.
Since both  $\mcO p^{^\mr{Zzz...}}_{\N, \rho, g, r}$ and $\mcO p^{^\mr{Zzz...}}_{\N', \rho_{\N'}, g, r}$ are proper smooth Deligne-Mumford $k$-stacks of dimension $1$,
this projection turns out to be an isomorphism, as desired.
\end{proof}

\subsection{Lower bounds for the  degrees of projections} \label{SS320}
For simplicity, we here  focus on the case $(g, r) = (0, 4)$.
Let us fix an integer $s$ satisfying $0 \leq s \leq \frac{p-3}{2}$.
We shall define $D_s$
to be the subset of $B_1^4 \left(= B_1 \times B_1 \times B_1 \times B_1 \right)$ consisting of elements $(\lambda_1, \lambda_2, \lambda_3, \lambda_4)$  satisfying the inequality 
\begin{align}
s \leq \mr{min} \left( \left\{ \lambda_i + \lambda_j\right\}_{1 \leq i < j \leq 4} \cup \left\{p -2 - \lambda_i -\lambda_j \right\}_{1 \leq i < j \leq 4} \right) - \mr{max} \left\{ |\lambda_i -\lambda_j| \right\}_{1 \leq i < j \leq 4}.
\end{align}
Note that this set turns out to be nonempty.
In fact, under the assumption  $p \equiv 1$ (resp., $p \equiv 3$) mod $4$,
the quadruple $(\frac{p-1}{4}, \frac{p-1}{4}, \frac{p-1}{4}, \frac{p-1}{4})$
(resp., $(\frac{p-3}{4}, \frac{p-3}{4}, \frac{p-3}{4}, \frac{p-3}{4})$) belongs to $D_s$.

For each $a \in \mbZ_p \left(\supseteq \mbZ \right)$,  let $(\pi_0 (a), \pi_1 (a), \cdots)$ be the  sequence of elements in $\{0, \cdots, p-1 \}$  uniquely determined by the equality
$a  = \sum_{j=0}^\infty \pi_j (a) \cdot  p^j$.
Given  a positive integer $\N$, we obtain the set
\begin{align} \label{Eq2222}
D_{s, \N} := \left\{ (\lambda_i)_{i=1}^4 \in B_\N^4 \, | \, (\pi_j (\lambda_1), \pi_j (\lambda_2), \pi_j (\lambda_3), \pi_j (\lambda_4)) \in D_s \ \text{for every $j =0, 1, \cdots, \N -1$}\right\}.
\end{align}
Since  $D_s \neq \emptyset$,  the set $D_{s, \N}$ is verified to be nonempty.
Also, for $(\lambda_i)_{i=1}^4 \in D_{s, \N}$, the associated sets $B_{\N, (\lambda_1, \lambda_2, \lambda_3, \lambda_4), 0, 4}$,
$B_{\N, (\lambda_1, \lambda_3, \lambda_2, \lambda_4), 0, 4}$, and
$B_{\N, (\lambda_1, \lambda_4, \lambda_2, \lambda_3), 0, 4}$ are nonempty, and their cardinalities satisfies 
\begin{align} \label{EQ903}
s^\N \leq \sharp (B_{\N, (\lambda_1, \lambda_2, \lambda_3, \lambda_4), 0, 4}), \sharp (B_{\N, (\lambda_1, \lambda_3, \lambda_2, \lambda_4), 0, 4}), \sharp (B_{\N, (\lambda_1, \lambda_4, \lambda_2, \lambda_3), 0, 4}).
\end{align}

Now, denote by 
\begin{align} \label{Eq4444}
\delta (D_{s, \N})
\end{align}
 the image of  $D_{s, \N}$ under the bijection  $B_\N^4 \xrightarrow{\sim} ((\mbZ/p^\N \mbZ)^\times /\{ \pm 1 \})^4$ induced by $\delta$.
Then, the following assertion holds.

\ble \label{Lem22}
Let $\rho$ be an element of $\delta (D_{s, \N})$.
Then, the following inequalities hold:
\begin{align}
3 \cdot s^\N \leq \sharp (C_{\N, \rho, 0, 4}), \hspace{5mm} s^\N \leq \mr{deg} (\Pi_{\N, \rho, 0, 4}).
\end{align}
\ele
\begin{proof}
The assertion follows from  Theorem \ref{Th3} together with the above argument.
\end{proof}

\vspace{10mm}
\section{Explicit  genus formula for   dormant modular curves} \label{S3}

This section is devoted to establishing explicit formulas for computing the genera of dormant modular curves of type $(0, 4)$  (cf. Theorem \ref{Thm55}).
The proof relies on detailed analysis of local $\mcD$-module structures of finite level around  smooth and nodal points, Frobenius-pullback techniques, and computations of Chern characters of various sheaves associated to the universal dormant oper.

Throughout this section, 
 we assume that $k$ is algebraically closed.
 For each $\lambda \in (\mbZ/p^\N\mbZ)^\times/\{ \pm1 \}$, we shall write $\lambda_\N^\circledast$,  or simply $\lambda^\circledast$,  for the unique element of $\left\{ 1, \cdots, \frac{p^\N - 1}{2} \right\}$  
 such that the element  $\frac{1}{2} \cdot  \lambda^\circledast$ in $(\mbZ/p^\N \mbZ)^\times/\{ \pm 1 \}$  coincides with $\lambda$.

\subsection{Adjoint bundle associated to the universal dormant $\mr{PGL}_2$-oper} \label{SS49}

Let $(g, r)$ be either $(0, 4)$ or $(1, 1)$.
Let $\rho := (\rho_i)_{i=1}^r$ be an element of  $((\mbZ/p^\N \mbZ)^\times / \{ \pm 1 \})^r$ with $\mcO p^{^\mr{Zzz...}}_{\N, \rho, g, r} \neq \emptyset$.
For simplicity, we write $S = \mcO p^{^\mr{Zzz...}}_{\N, \rho, g, r}$, $\Pi := \Pi_{\N, \rho, g, r}$, and 
 write   $\msX := (f :X \rightarrow S, \{ \sigma_i \}_{i=1}^r)$  for 
the universal family of pointed stable curves over $S$.
Also, write $\Omega := \Omega_{X^\mr{log}/S^\mr{log}}$ and $\mcT := \mcT_{X^\mr{log}/S^\mr{log}}$.
The sum of the divisors $[\sigma_i]$ determined by (the images of) $\sigma_i$'s ($i=1, \cdots, r$) defines a relative effective divisor $D := \sum_{i=1}^r [\sigma_i]$ on $X$.
The relative dualizing sheaf $\omega$ on $X/S$ satisfies $\Omega (-D) \cong \omega$.

For each $\lambda \in C_{\N, \rho, g, r}$, we denote by $q_\lambda$ the 
unique non-smooth point of $X$ (relative to $S$) lying over the $k$-rational point of $S$  corresponding to $\lambda$ via  \eqref{Eq104}.
The base-change of $q_\lambda$ along $F_S^\N : S \rightarrow S$ determines a $k$-rational point  $q_\lambda^{(\N)}$  of $X^{(\N)}$.
 The smooth locus $X^{\mr{sm}}$  and  the non-smooth locus $Z$ of $X$ relative to $S$ satisfies 
 $X^\mr{sm} = X \setminus Z$ and  $Z = \coprod_{\lambda \in C_{\N, \rho, g, r}} \{ q_\lambda \}$.

Let us fix a dormant $2^{(\N)}$-theta characteristic $\vartheta := (\varTheta, \nabla_\vartheta)$ of $X^\mr{log}/S^\mr{log}$.
(According to ~\cite[Proposition 5.14]{Wak20}, there always exists such a $2^{(\N)}$-theta characteristic.)
The universal   dormant $\mr{PGL}_2^{(\N)}$-oper  on $\msX$ 
 corresponds to a
dormant $(\mr{GL}_2^{(\N)}, \vartheta)$-oper  $\nabla^\diamondsuit$ via \eqref{Eq107w}.

Denote by $\mcA d (\mcF_\varTheta)$ the sheaf of 
 $\mcO_X$-linear endomorphisms of $\mcF_\varTheta$ with vanishing trace.
This sheaf  forms a rank $3$ vector bundle, and admits a decreasing filtration
\begin{align}
0 =\mcA d (\mcF_\varTheta)^2 \subseteq  \mcA d (\mcF_\varTheta)^1 \subseteq \mcA d (\mcF_\varTheta)^0 \subseteq \mcA d (\mcF_\varTheta)^{-1} = \mcA d (\mcF_\varTheta)
\end{align}
determined by the condition that
$\mcA d (\mcF_\varTheta)^1$ (resp., $\mcA d (\mcF_\varTheta)^0$) consisting of local sections $h$ with $h (\mcF^1_\varTheta) = 0$ (resp., $h (\mcF^1_\varTheta) \subseteq \mcF^1_\varTheta$).
Since 
$\varTheta \cong (\mcF_\varTheta /\varTheta) \otimes \Omega$, we have 
 $\mcA d (\mcF_\varTheta)^j/ \mcA d (\mcF_\varTheta)^{j+1} \cong \Omega^{\otimes j}$ ($j=-1, 0, 1$).
Also,
$\mcA d (\mcF_\varTheta)$ admits a $\mcD^{(\N -1)}$-module structure $\nabla^\mr{ad}$ induced naturally from $\nabla^\diamondsuit$.
As discussed in the proof of ~\cite[Proposition 8.1]{Wak20},
the resulting collection 
\begin{align}
\mcA d^\heartsuit := (\mcA d (\mcF_\varTheta), \nabla^\mr{ad}, \{ \mcA d (\mcF_\varTheta)^{j-1}\}_{j=0}^3)
\end{align}
forms a dormant $\mr{GL}_3^{(\N)}$-oper on $\msX$ such that $\nabla^\mr{ad}$ has  exponent  at $\sigma_i$ coincides with $[0,  \rho_i^\circledast, p^\N -\rho_i^\circledast]$, regarded as
a multiset of elements in $\mbZ/p^\N \mbZ$ 
  via the quotient $\mbZ \twoheadrightarrow \mbZ/p^\N \mbZ$.

Recall  that $\mcS ol (\nabla^\mr{ad})$ carries an $\mcO_{X^{(\N)}}$-module structure via the underlying homeomorphism of $F^{(\N)}_{X/S}$.
According to  ~\cite[Corollary 8.13, Proposition 8.16]{Wak20}, 
the equality 
$\mbR^\ell f_* (\mcS ol (\nabla^{\mr{ad}})) = 0$ holds for $\ell \neq 1$, and  
there exists a canonical isomorphism 
\begin{align} \label{EQ11}
\mcT_{S^\mr{log}/k} \xrightarrow{\sim} \mbR^1 f_*^{(\N)} (\mcS ol (\nabla^{\mr{ad}}))
\end{align}
between vector bundles on $S$.
In the subsequent discussion, we compute  the Chern character $\mr{ch} (\mbR^1 f_*^{(\N)} (\mcS ol (\nabla^{\mr{ad}})))$ of the vector bundle   $\mbR^1 f_*^{(\N)} (\mcS ol (\nabla^{\mr{ad}}))$.
(When $(g, r) = (1, 1)$, the discussion must be conducted within the framework of the intersection theory for Deligne-Mumford stacks. For a basic  introduction to this theory, we refer the reader to, e.g., ~\cite{Mum}, ~\cite{Vis}.)

\subsection{Chern characters of higher direct images} \label{Ss1}

Let $\eta$ denote  the  $\mcO_{X^{(\N)}}$-linear  morphism $\mcS ol (\nabla^\mr{ad}) \rightarrow F^{(\N)}_{X/S*}(F^{(\N)*}_{X/S} (\mcS ol(\nabla^\mr{ad})))$ corresponding to the identity morphism of $F_{X/S}^{(\N)*} (\mcS ol (\nabla^\mr{ad}))$ via the adjunction relation $F_{X/S}^{(\N)*} (-) \dashv F_{X/S*}^{(\N)} (-)$.
This morphism is  injective, as  its domain $\mcS ol (\nabla^\mr{ad})$ is relatively torsion-free and  
the restriction $\eta |_{X^\mr{sm}}$ over the scheme-theoretic dense open subscheme  $X^\mr{sm} \left(\subseteq X \right)$ is verified to be  injective.
Hence, we obtain a short exact sequence of $\mcO_{X^{(\N)}}$-modules
\begin{align} \label{EQ30}
0 \rightarrow \mcS ol (\nabla^\mr{ad}) \xrightarrow{\eta} F^{(\N)}_{X/S*}(F_{X/S}^{(\N)*} (\mcS ol (\nabla^\mr{ad}))) \rightarrow \mr{Coker}(\eta) \rightarrow 0. 
\end{align}
Since $\mr{ch} (f_* (\mcS ol (\nabla^\mr{ad}))) = 0$ (cf. ~\cite[Proposition 8.2, (i)]{Wak20}), 
this sequence induces an equality of Chern characters
\begin{align} \label{EQee2}
\mr{ch} (\mbR^1 f^{(\N)}_* (\mcS ol (\nabla^\mr{ad}))) & \left(= - \mr{ch} (f_*^{(\N)}([\mcS ol (\nabla^\mr{ad})])) \right) \\
& \   =  -\mr{ch}( f_* ([F_{X/S}^{(\N)*}(\mcS ol (\nabla^\mr{ad}))])) +\mr{ch} ( f^{(\N)}_*([\mr{Coker}(\eta)])), \notag
\end{align}
where, for each coherent sheaf $\mcF$, we denote by  $[\mcF]$ the class in the Grothendiedk group of coherent sheaves  represented by $\mcF$.

In the first step, we examine  the term  $\mr{ch}( f_* ([F_{X/S}^{(\N)*}(\mcS ol (\nabla^\mr{ad}))]))$ in the right-hand side of the above equality.
The inclusion $\mcS ol (\nabla^\mr{ad}) \hookrightarrow F_{X/S*}^{(\N)}(\mcA d (\mcF_\varTheta))$ is $\mcO_{X^{(\N)}}$-linear, so it corresponds,  via the adjunction relation ``$F^{(\N)*}_{X/S} (-) \dashv F_{X/S*}^{(\N)}(-)$", to an $\mcO_X$-linear morphism
\begin{align} \label{Eq100}
\nu^\sharp : F^{(\N)*}_{X/S} (\mcS ol (\nabla^\mr{ad})) \rightarrow \mcA d (\mcF_\varTheta).
\end{align}
Since $\nu$ becomes an isomorphism when restricted to $X \setminus (Z \sqcup \bigcup_{i=1}^r  \mr{Im}(\sigma_i))$,
this morphism fits into a short exact sequence of $\mcO_X$-modules
\begin{align} \label{Eq115}
0 \rightarrow \mcG_\mr{Ker} \rightarrow 
F^{(\N)*}_{X/S} (\mcS ol (\nabla^\mr{ad})) \xrightarrow{\nu^\sharp} \mcA d (\mcF_\varTheta) \rightarrow  \mcG_{\mr{Coker}} \oplus \bigoplus_{i=1}^r \Lambda_i \rightarrow 0,
\end{align}
where 
both $\mcG_\mr{Ker}$ and $\mcG_{\mr{Coker}}$ are supported on the non-smooth locus  $Z$, and for each $i=1, \cdots, r$, $\Lambda_i$ is supported on $\mr{Im}(\sigma_i)$.
This exact sequence yields  an equality 
\begin{align} \label{EQ1}
& \ \ \ \  \mr{ch} ( f_* ([F^{(\N)*}_{X/S} (\mcS ol (\nabla^\mr{ad}))])) \\
& = \mr{ch} (f_*([\mcG_{\mr{Ker}}])) +  \mr{ch} (f_* ([\mcA d (\mcF_\varTheta)])) - \mr{ch} (f_* ([\mcG_{\mr{Coker}}])) - \sum_{i=1}^r\mr{ch} (f_* ([\Lambda_i])). \notag
\end{align}

The Chern characters   in its right-hand side    can be computed by Lemma \ref{Lem44}, \ref{Lem45}, and \ref{Lem46},  described as follows.

\ble \label{Lem44}
The Chern character $\mr{ch}(f_*([\mcA d (\mcF_\varTheta)]))$ of $f_*([\mcA d (\mcF_\varTheta)])$ is given by  
\begin{align}
\mr{ch}(f_*([\mcA d (\mcF_\varTheta)])) =
\begin{cases}  3 - \Pi^*(c_1 (\mcT_{\overline{\mcM}_{0, 4}^\mr{log}/k}))   & \text{if $(g, r) = (0, 4)$;}
\\
 3 - \Pi^*(c_1 (\mcT_{\overline{\mcM}_{1, 1}^\mr{log}/k}))  +2 \cdot \Pi^*(c_1 (\mbE)) & \text{if $(g, r) = (1, 1)$.}
\end{cases}
\end{align}
Here, $\mbE$ denotes the Hodge bundle on $\overline{\mcM}_{1, 1}$.
\ele
\begin{proof}
Since $\mcA d (\mcF_\varTheta)^j/ \mcA d (\mcF_\varTheta)^{j+1} \cong \Omega^{\otimes j}$ ($j=-1, 0, 1$), we obtain  the following sequence of equalities:
\begin{align}
& \ \ \ \ \mr{ch}(f_*([\mcA d (\mcF_\varTheta)])) \\
& = \sum_{j=-1}^1 \mr{ch}(f_* ([\mcA d (\mcF_\varTheta)^j/\mcA d (\mcF_\varTheta)^{j+1}])) \notag \\
& = \mr{ch} (f_* ([\mcT])) + \mr{ch} (f_* ([\mcO_X])) + \mr{ch} (f_* ([\Omega])) \notag \\
& = - \mr{ch} (\mbR^1 f_* (\mcT))  + \left(\mr{ch} (f_* (\mcO_X)) - \mr{ch} (\mbR^1f_* (\mcO_X)) \right) \\
& \ \ \ \
 + \left( \mr{ch} (f_* (\omega)) -\mr{ch} (\mbR^1f_* (\omega)) + \mr{ch} (f_* (\mcO_D))\right) \\
& = 3 - c_1 (\Pi^*(\mcT_{\overline{\mcM}_{g, r}^\mr{log}/k}))  +2 \cdot c_1 (f_* (\omega)) \notag \\ 
& = 3 - \Pi^*(c_1 (\mcT_{\overline{\mcM}_{g, r}^\mr{log}/k}))  +2 \cdot c_1 (f_* (\omega)), \notag 
\end{align}
where the fourth equality follows from the isomorphism $\Pi^*(\mcT_{\overline{\mcM}_{g, r}^\mr{log}/k}) \xrightarrow{\sim} \mbR^1 f_* (\mcT)$ obtained by  pulling-back the  usual Kodaira-Spencer isomorphism for the universal curve over $\overline{\mcM}_{g, r}$ along $\Pi$.
When  $(g, r) = (0, 4)$, 
the relative degree of $\omega$ is negative, which implies 
 $c_1 (f_* (\omega)) = 0$.
On the other hand, when $(g, r) = (1, 1)$, the equality $c_1 (f_*(\omega)) = \Pi^* (c_1 (\mbE))$ holds because of the definition of $\mbE$.
This completes the proof of the assertion.
\end{proof}

\ble \label{Lem45}
Let us fix $i \in \{1, \cdots, r \}$.
Denote by $\psi_i$ the pull-back of the usual $i$-th psi class on $\overline{\mcM}_{g, r}$ along $\Pi$.
Then, the Chern character $\mr{ch}(f_*([\Lambda_i]))$ of $f_*([\Lambda_i])$ is given by
\begin{align}
\mr{ch}(f_*([\Lambda_i]))  = p^\N  \cdot [S]+  \frac{p^{2\N} -p^\N - 2  p^\N \cdot \rho_i^\circledast +2 \rho_i^{\circledast 2}}{2} \cdot \psi_i.
\end{align}
\ele
\begin{proof}
For simplicity, we set $\mcH := F^{(\N)*}_{X/S} (\mcS ol (\nabla^\mr{ad}))$.
Denote by $\nu^\flat_i$ the natural quotient $\mcA d (\mcF_\varTheta) \twoheadrightarrow \Lambda_i$.
Let us define a decreasing filtration  $\{\mcH^{j -1} \}_{j=0}^3$ on $\mcH$ (resp., $\{ \Lambda^{j-1}_i \}_{j=0}^3$ on $\Lambda_i$) as 
$\mcH^{j-1} := (\nu^\sharp)^{-1} (\mcA d (\mcF_\varTheta)^{j-1})$ (resp., $\Lambda_i^{j-1} :=  \nu_i^\flat (\mcA d (\mcF_\varTheta)^{j-1})$).
 The exact sequence  \eqref{Eq115} induces, 
for each $j= 0, 1, 2$,  a short exact sequence
\begin{align} \label{Eq101}
0 \rightarrow (\mcH^{j-1}/\mcH^j) |_{X^\mr{sm}} \rightarrow (\mcA d (\mcF_\varTheta)^{j-1}/\mcA d (\mcF_\varTheta)^j)|_{X^\mr{sm}}   \rightarrow \bigoplus_{i=1}^r \Lambda_i^{j-1}/\Lambda_i^{j} \rightarrow 0.
\end{align}
As mentioned before,  the exponent  of $\mcA d^\heartsuit$ at $\sigma_i$ is $[0, \rho_i^\circledast, p^\N - \rho_i^\circledast]$.
Since $0 \leq \rho_i^\circledast \leq p^\N - \rho_i^\circledast$, an argument similar to the proof of ~\cite[Proposition 8.8, (i)]{Wak8}
shows that $\Lambda_i^{-1}/\Lambda_i^0$ (resp., $\Lambda_i^{0}/\Lambda_i^1$; resp., $\Lambda_i^{1}/\Lambda_i^2$) restricted to each geometric  fiber of $X/S$  is of length $0$ (resp., $\rho_i^\circledast$; resp., $p^\N -\rho_i^\circledast$).
Under the natural identification $\mcA d (\mcF_\varTheta)^{-1}/\mcA d (\mcF_\varTheta)^{0} \cong \Omega^{\otimes (-1)}$ (resp., $\mcA d (\mcF_\varTheta)^{0}/\mcA d (\mcF_\varTheta)^{1} \cong \mcO_X$; resp., $\mcA d (\mcF_\varTheta)^{-1}/\mcA d (\mcF_\varTheta)^{0} \cong \Omega$), the second arrow  in the short exact sequence \eqref{Eq101} can  be identified with the
identity morphism of $\Omega^{\otimes (-1)}$ (resp., the inclusion $\mcO_X \left(-\sum_{i=1}^r \rho_i^\circledast  \cdot [\sigma_i]\right) \hookrightarrow \mcO_X$; resp., the inclusion $\Omega \left(-\sum_{i=1}^r (p^\N -\rho_i^\circledast )\cdot [\sigma_i] \right)$ $\hookrightarrow \Omega$) restricted over $X^\mr{sm}$.
In particular,  we have 
\begin{align}
\Lambda^{-1}_i/\Lambda^0_i = 0, \hspace{5mm}
\Lambda_i^{0}/\Lambda_i^1 \cong  \mcO_X / \mcO_X \left(- \rho_i^\circledast \cdot [\sigma_i]\right),
\hspace{5mm}
\Lambda^{1}_i/\Lambda_i^2 \cong \Omega / \Omega \left(- \left(p^\N - \rho_i^\circledast\right)\cdot [\sigma_i]\right).
\end{align}
It follows that there exists a short sequence of   $\mcO_X$-modules
\begin{align} \label{EQ212}
0 \rightarrow \Omega / \Omega \left(- (p^\N - \rho_i^\circledast)\cdot [\sigma_i]\right) \rightarrow \Lambda_i \rightarrow \mcO_X / \mcO_X \left(- \rho_i^\circledast \cdot [\sigma_i ]\right)  \rightarrow 0.
\end{align}
Hence, the following sequence of equalities holds:
\begin{align}
& \ \ \ \ \mr{ch} (f_* ([\Lambda_i])) \\
& =   \mr{ch} (f_* ([ \mcO_X / \mcO_X \left(- \rho_i^\circledast \cdot [\sigma_i ]\right)])) +  \mr{ch} (f_* ([\Omega / \Omega \left(- (p^\N - \rho_i^\circledast)\cdot [\sigma_i]\right)])) \notag \\
& = \sum_{\ell = 0}^{\rho_i^\circledast -1} \mr{ch} (f_* (\mcO_X (- \ell \cdot [\sigma_i])/ \mcO_X (- (\ell+1) \cdot [\sigma_i]))) \notag \\
& \ \ \ \ + \sum_{\ell = 0}^{p^\N - \rho_i^\circledast -1} \mr{ch} (f_* (\Omega (- \ell \cdot [\sigma_i])/ \Omega (- (\ell+1) \cdot [\sigma_i]))) \notag \\
& = \sum_{\ell = 0}^{\rho_i^\circledast -1} ([S]+\psi_i)^{\ell}+\sum_{\ell = 0}^{p^\N - \rho_i^\circledast -1}([S]+\psi_i)^{\ell} \notag \\
&= \sum_{\ell = 0}^{\rho_i^\circledast -1} ([S]+ \ell \cdot \psi_i)+\sum_{\ell = 0}^{p^\N - \rho_i^\circledast -1}([S]+ \ell \cdot \psi_i) \notag \\
& = p^\N  \cdot [S]+  \frac{p^{2\N} -p^\N - 2  p^\N \cdot \rho_i^\circledast +2 \rho_i^{\circledast 2}}{2} \cdot \psi_i,
\end{align}
where  the third equality follows from the residue isomorphism $\sigma^*_i (\Omega) \cong \mcO_S$.
This completes the proof of the assertion.
\end{proof}

\ble \label{Lem46}
The Chern characters of $f_* ([\mcG_\mr{Ker}])$ and $f_* ([\mcG_{\mr{Coker}}])$ are given by 
\begin{align}
\mr{ch} (f_* ([\mcG_{\mr{Ker}}])) = \mr{ch} (f_* ([\mcG_{\mr{Coker}}])) = \sum_{\lambda \in C_{\N, \rho, g, r}} 2  \cdot \lambda^\circledast (p^\N -\lambda^\circledast) \cdot  [f(q_\lambda)].
\end{align}
\ele
\begin{proof}
Let us take an element $\lambda \in C_{\N, \rho, g, r}$.
Denote by $\widetilde{U}_\lambda$ (resp.,  $\widetilde{T}_\lambda$) the formal neighborhood of $q_\lambda$ (resp., $f (q_\lambda)$) in $X$ (resp., $S$).
Then, there exist isomorphisms $\widetilde{U}_\otimes \xrightarrow{\sim}\widetilde{U}_\lambda$ and $\widetilde{T} \xrightarrow{\sim} \widetilde{T}_\lambda$ that fit into the following Cartesian diagram:
 \begin{align} \label{Eq107}
\vcenter{\xymatrix@C=46pt@R=36pt{
 \widetilde{U}_\otimes \ar[r]^-{\sim} \ar[d] & \widetilde{U}_\lambda \ar[d] \\
 \widetilde{T} \ar[r]_-{\sim} & \widetilde{T}_\lambda,
 }}
\end{align}
where the right-hand vertical arrow is obtained by restricting $f$, and the left-hand vertical arrow denotes the structure morphism of the $\widetilde{T}$-scheme $\widetilde{U}_\otimes$.
The restriction of $(\mcA d (\mcF_\varTheta), \nabla^\mr{ad})$ to $\widetilde{U}_\lambda$ 
defines 
a $\widetilde{\mcD}_\otimes^{(\N -1)}$-module under the identifications $\widetilde{U}_\otimes =  \widetilde{U}_\lambda$, $\widetilde{T} = \widetilde{T}_\lambda$ given by \eqref{Eq107}, and is isomorphic to
 $\msO^{(\N -1)}_{\otimes, 0} \oplus \msO^{(\N -1)}_{\otimes,  \lambda^\circledast} \oplus \msO^{(\N -1)}_{\otimes, p^\N - \lambda^\circledast}$.
Hence, by Proposition \ref{Pro1}, (ii), we have
\begin{align}
\mr{dim}_k (H^0 (\widetilde{U}_\otimes, \mr{Ker}(\nu^\sharp |_{\widetilde{U}_\lambda}))) 
& =  \sum_{\ell = 0, \lambda, p^\N -\lambda} \mr{dim}_k (H^0 (\widetilde{U}_\otimes, \mr{Ker}(\kappa_\ell)))\\
& = 0 +   \lambda^\circledast  (p^\N -\lambda^\circledast) + (p^\N -\lambda^\circledast)   \lambda^\circledast   \notag \\
&=2 \cdot \lambda^\circledast  (p^\N -\lambda^\circledast).
\end{align}
Similarly,   the equality  $\mr{dim}_k (H^0 (\widetilde{U}_\otimes, \mr{Coker}(\nu^\sharp |_{\widetilde{U}_\lambda}))) = 2 \cdot \lambda^\circledast  (p^\N -\lambda^\circledast)$ holds.
This  completes  the proof of the assertion.
\end{proof}

By  combining  the above lemmas,
we obtain the following assertion.

\bpr \label{Prop44}
Regarding the Chern character $\mr{ch} (f_* ([F^{(\N)*}_{X/S}(\mcS ol (\nabla^\mr{ad}))]))$ of $f_* ([F^{(\N)*}_{X/S}(\mcS ol (\nabla^\mr{ad}))])$,
we have   the following assertions.
\begin{itemize}
\item
If $(g, r) = (0, 4)$, then the following equality holds:
\begin{align}
& \ \ \ \ \mr{ch} (f_* ([F^{(\N)*}_{X/S}(\mcS ol (\nabla^\mr{ad}))])) \\
& =  (3 - 4 \cdot p^\N) \cdot [S]-\Pi^* (c_1 (\mcT_{\overline{\mcM}_{0, 4}^\mr{log}/k}))
-\sum_{i=1}^4 \left( \frac{p^{2\N}-p^\N -2  p^\N \cdot \rho_i^\circledast  + 2 \cdot \rho_i^{\circledast 2}}{2} \cdot  \psi_i\right).  \notag
\end{align}
\item
If $(g, r) = (1, 1)$, then the following equality holds:
\begin{align}
& \ \ \ \ \mr{ch} (f_* ([F^{(\N)*}_{X/S}(\mcS ol (\nabla^\mr{ad}))])) \\
& = (3-p^\N) \cdot [S] -\Pi^* (c_1 (\mcT_{\overline{\mcM}_{1, 1}^\mr{log}/k})) + 2 \cdot \Pi^* (c_1 (\mbE)) 
-  \frac{p^{2\N}-p^\N -2  p^\N \cdot \rho^\circledast  + 2 \cdot \rho^{\circledast 2}}{2} \cdot  \psi_1. \notag
\end{align}
\end{itemize}
\epr
\begin{proof}
First, let us consider the case $(g, r) = (0, 4)$.
The equality \eqref{EQ1} together with Lemmas \ref{Lem44}, \ref{Lem45}, and \ref{Lem46}, we have
\begin{align}
& \ \ \ \ \mr{ch} (f_* ([F^{(\N)*}_{X/S}(\mcS ol (\nabla^\mr{ad}))])) \\
& = 3 \cdot [S] -\Pi^* (c_1 (\mcT_{\overline{\mcM}_{0, 4}^\mr{log}/k}))
-\sum_{i=1}^4 \left(p^\N \cdot [S] + \frac{p^{2\N}-p^\N -2  p^\N \cdot \rho_i^\circledast  + 2 \cdot \rho_i^{\circledast 2}}{2} \cdot  \psi_i\right) \\
& =  (3 - 4 p^\N) \cdot [S]-\Pi^* (c_1 (\mcT_{\overline{\mcM}_{0, 4}^\mr{log}/k}))
-\sum_{i=1}^4 \left( \frac{p^{2\N}-p^\N -2  p^\N\cdot \rho_i^\circledast  + 2 \cdot \rho_i^{\circledast 2}}{2} \cdot  \psi_i\right). 
\end{align}
On the other hand, when $(g, r) = (1, 1)$,  the same lemmas say  that 
\begin{align}
& \ \ \ \ \mr{ch} (f_* ([F^{(\N)*}_{X/S}(\mcS ol (\nabla^\mr{ad}))])) \\
& = 3 \cdot [S] -\Pi^* (c_1 (\mcT_{\overline{\mcM}_{1, 1}^\mr{log}/k})) + 2 \cdot \Pi^* (c_1 (\mbE)) 
- \left(p^\N \cdot [S] + \frac{p^{2\N}-p^\N -2  p^\N \cdot \rho^\circledast  + 2 \cdot \rho^{\circledast 2}}{2} \cdot  \psi\right) \notag \\
& = (3-p^\N) \cdot [S] -\Pi^* (c_1 (\mcT_{\overline{\mcM}_{1, 1}^\mr{log}/k})) + 2 \cdot \Pi^* (c_1 (\mbE)) 
-  \frac{p^{2\N}-p^\N -2 p^\N  \cdot \rho^\circledast  + 2 \cdot \rho^{\circledast 2}}{2} \cdot  \psi_1.
\end{align}
This completes the proof of the assertion.
\end{proof}

\subsection{Cokernel of the adjunction morphism   for  $(g, r) = (0, 4)$} \label{SS9}

The next step is devoted to computing   the term $\mr{ch} (f_*^{(\N)}([\mr{Coker}(\eta)]))$ in \eqref{EQee2}.
To this end, the remainder of this paper will focus exclusively on the case where $(g, r) = (0, 4)$. In the case $(g, r) = (1, 1)$, the discussion regarding calculations of Chern characters becomes more complex, and for this reason, we have not yet been able to obtain results corresponding  to those described below. Completing this work remains a task for the future.

We write $\mbP^1$ for the projective line over $k$, i.e., $\mbP^1 := \mr{Proj} (k[u, v])$.
Also, write  $Y  := \mbP^1 \times_k S$, which is equipped with a structure of  $S$-scheme by using the second projection $\mbP^1 \times_k S \twoheadrightarrow S$. 
 There exists a natural projection 
 $\pi : X \rightarrow Y$ such that $X$ can be obtained as the blow-ups of $Y$ at the points $\overline{q}_\lambda := \pi (q_\lambda)$ ($\lambda \in B_{\N, \rho, 0, 4}$).
 The base-change of $\pi$ via $F_S^\N$ determines 
  a morphism $\pi^{(\N)} : X^{(\N)} \rightarrow Y^{(\N)}$.

Let us consider the composite
\begin{align}
F_{Y/S*}^{(\N)}(\mcO_Y)\rightarrow F_{Y/S*}^{(\N)}(\pi_* (\mcO_X)) \xrightarrow{\sim} \pi^{(\N)}_* (F_{X/S*}^{(\N)}(\mcO_X)),
\end{align}
where the first arrow is the direct image of the natural morphism $\mcO_Y \rightarrow \pi_* (\mcO_X)$ via $F_{Y/S}^{(\N)}$, and the second 
arrow arises from the equality $F_{Y/S}^{(\N)}\circ \pi = \pi^{(\N)} \circ F_{X/S}^{(\N)}$.
This composite corresponds to a morphism 
$\beta : \pi^{(\N)*} (F_{Y/S*}^{(\N)} (\mcO_Y)) \rightarrow F_{X/S*}^{(\N)} (\mcO_X)$ via the adjunction relation ``$\pi^{(\N)*} (-) \dashv \pi^{(\N)}_* (-)$".
In particular, we obtain a composite
\begin{align}
\beta_{\nabla} : \pi^{(\N)*} (F_{Y/S*}^{(\N)} (\mcO_Y)) \otimes \mcS ol (\nabla^\mr{ad})
& \xrightarrow{\beta \otimes \mr{id}_{\mcS ol (\nabla^\mr{ad})}}
F_{X/S*}^{(\N)} (\mcO_X) \otimes \mcS ol (\nabla^\mr{ad}) \\
& \xrightarrow{\sim}
F_{X/S*}^{(\N)} (F_{X/S}^{(\N)*}(\mcS ol (\nabla^\mr{ad}))), \notag
\end{align}
where the last arrow follows from the projection formula.

Let us 
write $\mcB := \mr{Coker} \left(\mcO_{Y^{(\N)}} \rightarrow F_{Y/S*}^{(\N)} (\mcO_Y)  \right)$, which forms  a vector bundle on $Y^{(\N)}$ of rank $p^\N -1$.
 Then, we obtain a morphism of short exact sequences
\begin{align} \label{Ed37}
\vcenter{\xymatrix@C=26pt@R=36pt{
 0 \ar[r] & \mcS ol (\nabla^\mr{ad}) \ar[r] \ar[d]  & \pi^{(\N)*} F_{Y/S*}^{(\N)} (\mcO_Y) \otimes \mcS ol (\nabla^\mr{ad}) \ar[r] \ar[d]^-{\beta_\nabla}  & \pi^{(\N)*}(\mcB) \otimes \mcS ol (\nabla^\mr{ad}) \ar[r] \ar[d] & 0
 \\
0\ar[r]  & \mcS ol (\nabla^\mr{ad}) \ar[r]_-{\eta} & F_{X/S*}^{(\N)} (F_{X/k}^{(\N)*}(\mcS ol (\nabla^\mr{ad})))\ar[r] & \mr{Coker}(\eta)  \ar[r] & 0,
 }}
\end{align}
where the upper horizontal  sequence is obtained as the tensor product of  $\mcS ol (\nabla^\mr{ad})$ and 
the pull-back via $\pi^{(\N)}$ of the  short exact sequence
 $0 \rightarrow \mcO_{Y^{(\N)}} \rightarrow F_{Y/S*}^{(\N)}(\mcO_Y) \rightarrow \mcB \rightarrow 0$.
Applying the snake lemma to this diagram, we obtain a short exact sequence of $\mcO_{X^{(\N)}}$-modules
\begin{align} \label{EQ3}
0 \rightarrow \pi^{(\N)*}(\mcB) \otimes \mcS ol (\nabla^\mr{ad}) \rightarrow \mr{Coker} (\eta)
\rightarrow  \mr{Coker}(\beta_\nabla) \rightarrow 0.
\end{align}
This  sequence induces an equality of Chern characters
\begin{align} \label{EQ61}
\mr{ch} (f_*^{(\N)} ([\mr{Coker} (\eta)])) = \mr{ch} (f_*^{(\N)} ([\mr{Coker} (\beta_\nabla)])) + \mr{ch} (f_*^{(\N)} ([ \pi^{(\N)*}(\mcB) \otimes \mcS ol (\nabla^\mr{ad})])).
\end{align}

Regarding the right-hand side of this equality, one can prove  Lemmas \ref{Lem49} and  \ref{Lem79}, described as follows.

\ble \label{Lem49}
The Chern character  $\mr{ch} (f_*^{(\N)} ([\mr{Coker} (\beta_\nabla)]))$ of $f_*^{(\N)} ([\mr{Coker} (\beta_\nabla)])$ is given by
\begin{align} \label{EQ177}
\mr{ch} (f_*^{(\N)} ([\mr{Coker} (\beta_\nabla)])) = 
 \sum_{\lambda \in C_{\N, \rho, 0, 4}} \left(-p^\N(p^\N -1)+ \lambda^\circledast  \cdot (p^\N -\lambda^\circledast) \right)\cdot [f (q_\lambda)].
\end{align}
\ele
\begin{proof}
Let $S_0 := \mr{Spec} (k[t])$ and  $Y_0 := \mr{Spec} (k[t, x])$.
Using  the morphism $Y_0 \rightarrow S_0$ corresponding to the inclusion $k[t] \rightarrow k[t, x]$,
we consider $Y_0$ as a curve over $S_0$.
Denote by $X_0$ the surface over $k$ defined  by blowing-up $Y_0$ at the point $\overline{q}$ determined by $(t, x) = (0, 0)$.
Also, denote by  $\pi_0 : X_0 \rightarrow Y_0$ and $f_0 : X_0 \rightarrow S_0$ the natural projections, and by $E$ the exceptional divisor of this blow-up.
Then, 
$X_0$ is covered by  two open subschemes  $\mr{Spec} (k[t, x, y]/(xy - t))$ and $\mr{Spec} (k[t, x, z])/(x- zt)$ under  the relation $x = \frac{1}{z}$.
The morphism 
\begin{align}
\beta_0 : \pi_0^{(\N)*} (F_{Y_0/S_0*}^{(\N)}(\mcO_{Y_0})) \rightarrow F_{X_0/S_0*}^{(\N)} (\mcO_{X_0})
\end{align}
  defined  in the same manner as $\beta$  can be obtained by gluing the  $k[t, x^{p^\N}, y^{p^\N}]/(x^{p^\N} y^{p^\N} -t^{p^\N})$-linear  morphism
\begin{align} \label{EQ117}
\beta_{0, y} & : k[t, x, y^{p^\N}/(x^{p^\N}y^{p^\N}-t^{p^\N})] \left(= k[t, x] \otimes_{k[t, x^{p^\N}]} k[t, x^{p^\N}, y^{p^\N}]/(x^{p^\N}y^{p^\N} -t^{p^\N}) \right) \\
&  \ \ \ \rightarrow  k[t, x, y]/(xy - t) \notag
\end{align}
and the  $k[t, x^{p^\N}, z^{p^\N}]/(x^{p^\N} -z^{p^\N}t^{p^\N})$-linear morphism
\begin{align}
\beta_{0, z} &: k[t, x, z^{p^\N}] /(x^{p^\N} - z^{p^\N} t^{p^\N}) \left( = k[t, x] \otimes_{k[t, x^{p^\N}]} k[t, x^{p^\N}, z^{p^\N}]/(x^{p^\N} -z^{p^\N}t^{p^\N})\right)\\
& \ \ \  \rightarrow k[t, x, z]/(x - zt)
\end{align}
obtained in natural manners.
Both $\beta_{0, y}$ and $\beta_{0, z}$ are injective, and we have 
\begin{align} \label{EQ119}
\mr{Coker} (\beta_{0, y}) & = \bigoplus_{i \geq 0, j \geq 1, p^\N-1 \geq i + j} k [y^{p^\N}] \cdot t^i y^j, \\
\mr{Coker} (\beta_{0, z}) & = \bigoplus_{i \geq 0, j \geq 1, p^\N-1 \geq i+j} k[z^{p^\N}] \cdot  x^i z^{j}.  \notag
\end{align}
It follows that
$\beta_0$ is injective and its cokernel is supported on $E^{(\N)} \cong \mbP^{1(\N)}$.

Note that there exists a decreasing filtration $\{ \mr{Coker}(\beta_{0})^\ell \}_{\ell =0}^{p^\N-1}$ on  $\mr{Coker}(\beta_0)$ such that 
$\mr{Coker} (\beta_0)^\ell$ corresponds to
the subspace $\bigoplus_{i \geq \ell, j \geq 1, p^\N-1 \geq i + j} k [y^{p^\N}] \cdot t^i y^j$ via the first decomposition  in \eqref{EQ119} and
corresponds to the subspace $\bigoplus_{i \geq \ell, j \geq 1, p^\N-1 \geq i + j} k [z^{p^\N}] \cdot x^i z^j$ via the second  equality.
The graded piece 
$\mr{Coker}(\beta_0)^\ell /\mr{Coker}(\beta_0)^{\ell +1}$ is considered as an $\mcO_{E^{(\N)}}$-module, and it is isomorphic to
he direct sum $\mcG_{\ell, 1} \oplus \cdots \oplus \mcG_{\ell, p^\N-1 - \ell}$ where $\mcG_{\ell, a}$ corresponds to the factor $k  [y^{p^\N}]\cdot t^\ell y^a$ via
the first decomposition  in \eqref{EQ119}  and the factor $k [z^{p^\N}] \cdot x^\ell z^{p^\N-\ell -a}$ in the second decomposition.
Since 
$x^\ell z^{p^\N-\ell -a} = y^{-p^\N} \cdot t^\ell y^{a}$, 
the $\mcO_{E^{(\N)}}$-module $\mcG_{\ell, a}$ is isomorphic to $\mcO_{E^{(\N)}}(-1)$.

Next, we regard $\widetilde{U}_\otimes$ (resp., $\widetilde{T}$) as the formal neighborhood of the point of $X_0$ (resp., $S_0$) determined by $(t, x, y) = (0, 0, 0)$ (resp., $t=0$).
For $d \in \mbZ/p^\N \mbZ$,
 we have 
\begin{align} \label{EQ200dd}
& \ \ \ \ H^0 (\widetilde{U}_\otimes^{(\N)}, \mr{Coker} (\beta_{0} |_{\widetilde{U}_{\otimes}} \otimes \mr{id}_{ \mcS ol (\widetilde{\nabla}^{(\N -1)}_{\otimes, d})})) \\
& \cong
H^0 (\widetilde{U}_\otimes^{(\N)}, \mr{Coker} (\beta_{0} |_{\widetilde{U}_{\otimes}}) \otimes  \mcS ol (\widetilde{\nabla}^{(\N -1)}_{\otimes, d})) \\
&\cong 
\begin{cases}
\bigoplus_{i \geq 0, j \geq 1, p^\N-1 \geq i + j} k [y^{p^\N}] \cdot t^i y^j \otimes 1 & \text{if $d = 0$};
\\
\bigoplus_{i \geq 0, j \geq 1, p^\N >  i+j, p^\N >  i + \widetilde{d}} k \cdot  t^i y^j \otimes x^{\widetilde{d}} \oplus 
\bigoplus
_{i \geq 0, j \geq 1, p^\N-1 \geq i + j} k [y^{p^\N}]\cdot t^i y^j \otimes y^{p^\N -\widetilde{d}} & \text{if $d \neq  0$}.
\end{cases}
\end{align}
Under the assumption that  $d \neq 0$,
we
write $\mcR_d$ for the quotient of $\mr{Coker} (\beta_{0} |_{\widetilde{U}_{\otimes}} \otimes \mr{id}_{ \mcS ol (\widetilde{\nabla}^{(0)}_{\otimes, d})})$ by the $\widetilde{\mcO}_\otimes$-submodule  corresponding to the second direct summand
$\bigoplus
_{i \geq 0, j \geq 1, p^\N-1 \geq i + j} k [y^{p^\N}]\cdot t^i y^j \otimes y^{p^\N -\widetilde{d}}$
 in the rightmost of  \eqref{EQ200dd}.
 The $\widetilde{\mcO}^{(\N)}_\otimes$-module $\mcR_d$ is supported on the point $(t, x^{p^\N}, y^{p^\N}) = (0, 0, 0)$, and  there exists a natural surjection $\gamma_d : \mr{Coker} (\beta_{0} |_{\widetilde{U}_{\otimes}} \otimes \mr{id}_{ \mcS ol (\widetilde{\nabla}^{(0)}_{\otimes, d})}) \twoheadrightarrow \mcR_d$.
 Also, 
 we have
 \begin{align} \label{EQ201}
 \mr{dim}_k (H^0 (\widetilde{U}_\otimes^{(\N)}, \mcR_d)) &= \mr{dim}_k \left(\bigoplus_{i \geq 0, j \geq 1, p^\N >  i+j, p^\N >  i + \widetilde{d}} k \cdot  t^i y^j \otimes x^{\widetilde{d}} \right) \\
& =  \frac{\widetilde{d} (2p^\N-1- \widetilde{d})}{2}. \notag
 \end{align}

Now, let us take an element $\lambda \in B_{\N, \rho, 0, 4}$.
Denote by $\widetilde{U}_\lambda$ (resp., $\widetilde{V}_\lambda$)  the formal neighborhood of  $q_\lambda$ (resp.,  $\overline{q}_\lambda$) in  $X$ (resp., $Y$).
The natural morphism $\widetilde{V}_\lambda \rightarrow Y$ induces a morphism $\widetilde{V}^{(\N)}_\lambda \rightarrow Y^{(\N)}$.
Fix an isomorphism $\widetilde{V}_0 \xrightarrow{\sim} \widetilde{V}_\lambda$, which extends, via applying the functor $(-)^{(\N)}$,
to an isomorphism
 $Z_0 \xrightarrow{\sim} Z_\lambda$ between 
 $Z_0 := \widetilde{V}_0^{(0)} \times_{Y_0^{(\N)}} X_0^{(\N)}$ and $Z_\lambda := \widetilde{V}^{(\N)}_\lambda \times_{Y^{(\N)}} X^{(\N)}$.
The morphism $\beta$ can be identified with  $\beta_0$ via this isomorphism.
We  consider  $\mr{Coker} (\beta_\nabla) |_{Z_\lambda}$ as an $\mcO_{Z_0}$-module under the fixed isomorphism $Z_0 \xrightarrow{\sim} Z_\lambda$.
Since $(\mcA d (\mcF_\varTheta), \nabla^\mr{ad}) |_{\widetilde{U}_\lambda}$ is isomorphic to $\msO^{(\N -1)}_{\otimes, 0} \oplus \msO^{(\N)}_{\otimes, \lambda^\circledast} \oplus \msO^{(\N)}_{\otimes, p^\N -\lambda^\circledast}$,
 we obtain  a composite
 surjection
 \begin{align} \label{EQ122}
 \mr{Coker} (\beta_\nabla) |_{Z_\lambda} 
& \xrightarrow{\sim} \bigoplus_{\lambda' =0, \lambda, -\lambda} \mr{Coker} (\beta_0 |_{\widetilde{U}_\otimes} \otimes \mr{id}_{\mcS ol (\widetilde{\nabla}^{(\N -1)}_{\otimes, \lambda'})}) \\
 & \xrightarrow{0 \oplus \gamma_{\lambda} \oplus \gamma_{-\lambda}} \mcR_{\lambda} \oplus \mcR_{-\lambda}.
 \end{align}
 Let $\mcQ_\lambda$ denote the kernel of this composite surjection.
 The filtration $\{ \mr{Coker} (\beta_0)^\ell \}_{\ell =0}^{p^\N-1}$  induces  a filtration
 $\{ \mcQ_\lambda^\ell \}_{\ell = 0}^{p^\N-1}$ via the first arrow in \eqref{EQ122}.
According to  Lemma \ref{Lem120} described below  and  the above discussion concerning  $\{ \mr{Coker} (\beta_0)^\ell \}_{\ell =0}^{p^\N-1}$,
 there exists a composite isomorphism 
 \begin{align}
 \mcQ_\lambda^\ell / \mcQ_\lambda^{\ell +1}
 & \xrightarrow{\sim} (\iota^* (\mcS ol (\nabla^\mr{ad}))/\mcH) \otimes \left(\bigoplus_{s =1}^{p^\N-1-\ell} \mcG_{\ell, s} \right) \\
& \xrightarrow{\sim}
 \mcO_{E^{(\N)}} (-1)^{\oplus 3} \otimes \mcO_{E^{(\N)}} (-1)^{\oplus (p^\N-1-\ell)} \notag \\
 & \xrightarrow{\sim} \mcO_{E^{(\N)}} (-2)^{\oplus 3(p^\N-1-\ell)}. \notag
 \end{align}
 Here,   $\iota$, or more precisely $\iota_\lambda$,  denotes
the morphism $\left(\mbP^1 \cong  \right) E^{(\N)} \hookrightarrow X^{(\N)}$ obtained, via base-change along  $F_{S_0}^\N$, from the inclusion $E \hookrightarrow X_0$
under the fixed  identifications $\widetilde{V}_0 = \widetilde{V}_\lambda$ and $Z_0 = Z_\lambda$.

 By taking the direct sum of the short exact sequences $0 \rightarrow \mcQ_\lambda \rightarrow \mr{Coker}(\beta_\nabla) |_{Z_\lambda} \rightarrow \mcR_\lambda \oplus \mcR_{-\lambda} \rightarrow 0$ defined for  the various $\lambda \in B_{\N, \rho, 0, 4}$, we obtain 
 a short exact sequence
\begin{align}
0 \rightarrow \bigoplus_{\lambda \in C_{\N, \rho, 0, 4}}\mcQ_\lambda \rightarrow \mr{Coker}(\beta_\nabla) \rightarrow \bigoplus_{\lambda \in C_{\N, \rho, 0, 4}}  \left(\mcR_{\lambda} \oplus  \mcR_{ - \lambda} \right)\rightarrow 0.
\end{align}
Hence, we have
\begin{align}
& \ \ \ \ \mr{ch} (f_*^{(\N)}[\mr{Coker} (\beta_\nabla)]) \\
& =  \sum_{\lambda \in C_{\N, \rho, 0, 4} }\mr{ch} (f_*^{(\N)}[\mcQ_\lambda]) + \sum_{\lambda \in C_{\N, \rho, 0, 4}} \left(\mr{ch} (f_*^{(\N)} ([\mcR_{\lambda}])) + \mr{ch} (f_*^{(\N)}([\mcR_{ -\lambda}])) \right) \\
& =  \sum_{\lambda \in C_{\N, \rho, 0, 4}} \sum_{\ell =0}^{p^\N-2} \mr{ch} (f_*^{(\N)} ([\mcQ_\lambda^\ell /\mcQ_\lambda^{\ell +1}])) \\
& \ \ \ \  + \sum_{\lambda \in C_{\N, \rho, 0, 4}} \left( \frac{\lambda^\circledast \cdot  (2p^\N-1 -\lambda^\circledast)}{2} \cdot [f(q_\lambda)] +  \frac{(p^\N - \lambda^\circledast) \cdot  (p^\N-1 +\lambda^\circledast)}{2} \cdot [f(q_\lambda)] \right) \notag \\
& =  \sum_{\lambda \in C_{\N, \rho, 0, 4}} \sum_{\ell =0}^{p^\N-2} \mr{ch} (f_*^{(\N)} ([\iota_{\lambda*}(\mcO_{\mbP^1} (-2))^{\oplus 3(p^\N -1-\ell)}]))  \notag \\
& \ \ \ \ + \sum_{\lambda \in C_{\N, \rho, 0, 4}}  \frac{p^{2\N}-p^\N +2p^\N \cdot \lambda^\circledast  -2\lambda^{\circledast 2}}{2} \cdot [f(q_\lambda)]  \notag \\
& =  \sum_{\lambda \in C_{\N, \rho, 0, 4}} \frac{-3p^\N (p^\N -1)}{2} \cdot [f(q_\lambda)]
+\sum_{\lambda \in C_{\N, \rho, 0, 4}}  \frac{p^{2\N}-p^\N +2p^\N \cdot \lambda^\circledast  -2\lambda^{\circledast 2}}{2} \cdot [f(q_\lambda)] \notag \\
& = \sum_{\lambda \in C_{\N, \rho, 0, 4}} \left(-p^\N(p^\N -1)+ \lambda^\circledast  \cdot (p^\N -\lambda^\circledast) \right)\cdot [f(q_\lambda)],
 \end{align}
 where the second equality follows from  \eqref{EQ201}.
This completes the proof of the assertion.
\end{proof}

The following lemma was applied in the above lemma.

\ble \label{Lem120}
Let us keep the notation in the proof of the previous lemma.
Denote by $\mcH$ the torsion sheaf of  $\iota^* (\mcS ol (\nabla^\mr{ad}))$.
Then, 
$\iota^* (\mcS ol (\nabla^\mr{ad}))/\mcH$ is isomorphic to $\mcO_{\mbP^1} (-1)^{\oplus 3}$ (under any identification $E^{(\N)} = \mbP^1$).
\ele
\begin{proof}
We regard $E$ as an irreducible component in the fiber  of $\msX$ over $f (q_\lambda) \in S$.
By equipping $E$ with the marked points   determined by the nodal and marked points of that fiber, 
we obtain a $3$-pointed smooth curve  $\msE := (E, \{ \sigma_{E, 1}, \sigma_{E, 2}, \sigma_{E, 3} \})$  of genus $0$.
Denote by $\vartheta_E := (\varTheta_E, \nabla_{\vartheta_E})$ the $2^{(\N)}$-theta characteristic of $E^\mr{log}$ obtained by restricting $\vartheta$.
The vector bundles $\mcF_{\varTheta_E}$ and $\mcA d (\mcF_{\varTheta_E})$ associated to $\vartheta_E$  can be defined as in the case of $\vartheta$.
In particular, the pull-back $\iota^* (\mcA d(\mcF_\varTheta))$  is naturally isomorphic to  $\mcA d (\mcF_{\varTheta_E})$.
For $i=1, 2,3$, we denote by $\rho_{E, i}$ the radius of the dormant $\mr{PGL}_2^{(\N)}$-oper on  $\msE$ obtained by restricting the universal dormant   $\mr{PGL}_2^{(\N)}$-oper on $\msX$ (cf. ~\cite[Section 6.4]{Wak20}).

Now,
 we set $\mcE := \iota^* (\mcS ol (\nabla^\mr{ad}))/\mcH$ for simplicity.
By the Grothendieck-Birkoff theorem,  there exists an isomorphism $\iota^* (\mcS ol (\nabla^\mr{ad}))/\mcH \cong \mcO_{E^{(\N)}} (\ell_1) \oplus \mcO_{E^{(\N)}} (\ell_2) \oplus \mcO_{E^{(\N)}} (\ell_3)$ for some $\ell_1, \ell_2, \ell_3 \in \mbZ$ with $\ell_1 \geq \ell_2 \geq \ell_3$.
The morphism  $F^{(\N)*}_{E/k}(\mcE) \rightarrow \mcA d (\mcF_{\varTheta_E})$ corresponding, via the adjunction relation $F^{(\N)*}_{E/k} (-) \dashv F_{E/k*}^{(\N)} (-)$, to the inclusion $\mcE \hookrightarrow \mcA d (\mcF_{\varTheta_E})$ is verified to be injective.
This morphism fits into the following short exact sequence
\begin{align}
0 \rightarrow F^{(\N)*}_{E/k} (\mcE) \rightarrow \mcA d (\mcF_{\varTheta_E}) \rightarrow \bigoplus_{i=1}^3 \Lambda_{E, i} \rightarrow 0,
\end{align}
where $\Lambda_{E, i}$ ($i=1, \cdots, r$) denotes an $\mcO_E$-module supported on $\sigma_{E, i}$.
Similarly to the proof of Lemma \ref{Lem45},   
we see that $\Lambda_{E, i}$ is an extension of $\mcO_E/\mcO_E (-\rho^\circledast_{E, i} [\sigma_{E, i}])$ by $\Omega_{E^\mr{log}/k} / \Omega_{E^\mr{log}/k}( -(p^\N -\rho^\circledast_{E, i}) [\sigma_i])$.
Thus,  the following sequence of equalities holds:
\begin{align} \label{Eq4567}
\mr{deg} (\mcE) &= \frac{1}{p^\N} \cdot \mr{deg} (F^{(\N)*}_{E/k} (\mcE)) \\
& = \frac{1}{p^\N} \cdot \left( \mr{deg} (\mcA d (\mcF_{\varTheta_E})) - \sum_{i=1}^3 \mr{length}_{\mcO_E} (\Lambda_{E, i}) \right) \notag \\
& =  \frac{1}{p^\N} \cdot \left( 0  -\sum_{i=1}^3 \left(\rho_{E, i}^\circledast - (p^\N - \rho^\circledast_{E, i}) \right) \right) \notag \\
& = -3.
\end{align}
 
 We suppose  that $\ell_1 \geq 0$.
 Since $\mr{deg} (\mcT_{E^\mr{log}/k}) < 0$,
 the composite
 \begin{align}
 & \ \ \ \ \mcO_{E} (p\ell_1) \\
&\hookrightarrow \bigoplus_{i=1}^3 \mcO_{E} (p \ell_i) \\
 & \xrightarrow{\sim} \bigoplus_{i=1}^3 F^{(\N)*}_{E/k}(\mcO_{E^{(\N)}} (\ell_i)) \\
 & \xrightarrow{\sim} F^{(\N) *}_{E/k} (\mcE) \notag \\
 & \hookrightarrow \mcA d (\mcF_{\varTheta_E}) \notag \\
 & \twoheadrightarrow \mcA d (\mcF_{\varTheta_E})^0/\mcA d (\mcF_{\varTheta_E})^1 \notag \\
 & \xrightarrow{\sim} \mcT_{E^\mr{log}/k} \notag
 \end{align}
 must be the zero map, where the first arrow denotes the inclusion into the first factor.
This contradicts the fact that  the morphism $F^{(\N)*}_{E/k} (\mcE) \rightarrow \mcA d (\mcF_{\varTheta_E})$ is an isomorphism over the generic point.
Hence,   the integers $\ell_1$, $\ell_2$, and $\ell_3$ are all negative.
By \eqref{Eq4567}, we have  $\ell_1 = \ell_2 = \ell_3 =-1$, and this  completes  the proof of the assertion.
 \end{proof}

\ble \label{Lem79}
Let us choose a $k$-rational point $s$ of $\mbP^1 \setminus \{[0], [1], [\infty] \}\left(\cong \mcM_{0, 4}\right)$, and denote by $D_s$ the divisor on $S$ defined as $\Pi^{-1} (s)$.
Then,
the Chern character  $ \mr{ch} (f_*^{(\N)} ([ \pi^{(\N)*}(\mcB) \otimes \mcS ol (\nabla^\mr{ad})]))$  of $f_*^{(\N)} ([ \pi^{(\N)*}(\mcB) \otimes \mcS ol (\nabla^\mr{ad})])$ satisfies the  equality
\begin{align} \label{EQ149}
 & \ \ \ \ \mr{ch} (f_*^{(\N)} ([ \pi^{(\N)*}(\mcB) \otimes \mcS ol (\nabla^\mr{ad})])) \\
& =  -3(p^\N -1) \cdot [S] - (p^\N -1) \cdot  \mr{ch} (\mbR^1 f_{*}^{(\N)} ([\mcS ol (\nabla^\mr{ad})])) + p^\N (p^\N -1) \cdot   [D_s]. \notag
\end{align}
\ele
\begin{proof}
It is verified that $\mcB \cong \mcO_{Y^{(\N)}} (-1)^{\oplus (p^\N -1)} \cong \mcO (-[\sigma^{(\N)}_{Y, s}])^{\oplus (p^\N -1)}$, where $\sigma_{Y, s}$ denotes the section $S \rightarrow Y$ determined by $s$.
In particular, we have
\begin{align} \label{EQ130}
 \pi^{(\N)*} (\mcB) \otimes \mcS ol (\nabla^\mr{ad}) \cong \left( \mcO_{X^{(\N)}} (- [\sigma^{(\N)}_{X, s}]) \otimes \mcS ol (\nabla^\mr{ad})\right)^{\oplus (p^\N -1)}.
\end{align}
The tensor product of  $\mcS ol (\nabla^\mr{ad})$ and the natural short exact sequence 
$0 \rightarrow \mcO_{X^{(\N)}} (-[\sigma_{X, s}^{(\N)}]) \rightarrow \mcO_{X^{(\N)}} \rightarrow \mcO_{X^{(\N)}}/\mcO_{X^{(\N)}} (-[\sigma_{X, s}^{(\N)}])\rightarrow 0$ determines a short exact sequence
\begin{align} \label{EQ133}
0 \rightarrow   \mcO_{X^{(\N)}} (- [\sigma^{(\N)}_{X, s}]) \otimes \mcS ol (\nabla^\mr{ad}) 
\rightarrow \mcS ol (\nabla^\mr{ad})
\rightarrow \sigma_{X, s*}^{(\N)}(\sigma_{X, s}^{(\N)*} (\mcS ol (\nabla^\mr{ad})))
\rightarrow 0.
\end{align}

Since $\mr{Im}(\sigma_i) \cap \mr{Im} (\sigma_{X, s}) = \emptyset$ for $i=1, 2, 3$,
the pull-back of \eqref{Eq115} via $\sigma_{X, s}$ yields a short exact sequence
\begin{align} \label{EQ135}
0 &\rightarrow  \sigma_{X,s}^{(\N)*} (\mcS ol (\nabla^\mr{ad}))  \left(\cong  \sigma_{X, s}^* (F^{(\N)*}_{X/S} (\mcS ol (\nabla^\mr{ad}))) \right) 
\\&
\rightarrow \sigma_{X, s}^* (\mcA d (\mcF_\varTheta)) \rightarrow \sigma_{X, s}^* (\Lambda_4) \rightarrow 0. 
\notag
\end{align}
By  the fact that $\mcA d (\mcF_\varTheta)^{j}/\mcA d (\mcF_\varTheta)^{j+1} \cong \Omega^{\otimes (j-1)}$, the equality  
$\mr{ch} (\sigma_{X, s}^{*} (\mcA d (\mcF_\varTheta))) = 3 \cdot  [S]$ holds.
Also,
it follows from the short exact sequence \eqref{EQ212} for $i=4$ that
$\mr{ch} (\sigma_{X, s}^* (\Lambda_4)) = p^\N \cdot [D_s]$.

By combining \eqref{EQ130}, \eqref{EQ133}, and \eqref{EQ135}, we have
\begin{align}
& \ \ \ \ \mr{ch} (f_*^{(\N)} ([ \pi^{(\N)*}(\mcB) \otimes \mcS ol (\nabla^\mr{ad})])) \\
& = 
(p^\N -1) \cdot \mr{ch} (f_*^{(\N)} (\mcO_{X^{(\N)}} (- [\sigma^{(\N)}_{X, s}]) \otimes \mcS ol (\nabla^\mr{ad})))\notag \\
& = (p^\N -1) \cdot \left(\mr{ch} (f_*^{(\N)} ([\mcS ol (\nabla^\mr{ad})])) - \mr{ch} (\sigma_{X, s}^{(\N)*} (\mcS ol (\nabla^\mr{ad}))) \right) \notag \\
& = (p^\N -1) \cdot \left(\mr{ch} (f_*^{(\N)} ([\mcS ol (\nabla^\mr{ad})])) - 3 \cdot [S] +  p^\N \cdot [D_s] \right) \notag\\
& =  -3(p^\N -1) \cdot [S] - (p^\N -1) \cdot  \mr{ch} (\mbR^1 f_{*}^{(\N)} ([\mcS ol (\nabla^\mr{ad})])) + p^\N (p^\N -1) \cdot   [D_s],
\end{align}
where the last equality follows from $f_*^{(\N)} (\mcS ol (\nabla^\mr{ad})) = 0$ (cf. ~\cite[Proposition 8.2]{Wak20}).
This completes the proof of the assertion.
\end{proof}

By combining Lemmas \ref{Lem49} and  \ref{Lem79}, we obtain the following assertion. 

\bpr \label{Prp145}
The Chern character $\mr{ch} (f_*^{(\N)} ([\mr{Coker}(\eta)]))$ of $f_*^{(\N)} ([\mr{Coker}(\eta)])$ is given by 
\begin{align} \label{EQ178}
& \ \ \ \ 
\mr{ch} (f_*^{(\N)} ([\mr{Coker}(\eta)]))  \\
& = \sum_{\lambda \in C_{\N, \rho, 0, 4}}  \left(-p^\N(p^\N -1)+ \lambda^\circledast  \cdot (p^\N -\lambda^\circledast) \right) \cdot [f(q_\lambda)] 
  -3(p^\N -1) \cdot [S]  \\
& \ \ \ \
- (p^\N -1) \cdot  \mr{ch} (\mbR^1 f_{*}^{(\N)} ([\mcS ol (\nabla^\mr{ad})])) + p^\N (p^\N -1) \cdot   [D_s].   \notag
\end{align}
\epr
\begin{proof}
Observe the following sequence of equalities:
\begin{align}
& \ \ \ \ \mr{ch} (f_*^{(\N)} ([\mr{Coker}(\eta)])) \\
&\stackrel{\eqref{EQ61}}{=} 
\mr{ch} (f_*^{(\N)} ([\mr{Coker} (\beta_\nabla)])) + \mr{ch} (f_*^{(\N)} ([ \pi^{(\N)*}(\mcB) \otimes \mcS ol (\nabla^\mr{ad})])) \\
& \stackrel{\eqref{EQ177}}{=} \sum_{\lambda \in C_{\N, \rho, 0, 4}}  \left(-p^\N(p^\N -1)+ \lambda^\circledast  \cdot (p^\N -\lambda^\circledast) \right) \cdot [f(q_\lambda)] + 
\mr{ch} (f_*^{(\N)} ([ \pi^{(\N)*}(\mcB) \otimes \mcS ol (\nabla^\mr{ad})])) \notag \\
& \stackrel{\eqref{EQ149}}{=}
\sum_{\lambda \in C_{\N, \rho, 0, 4}} \left(-p^\N(p^\N -1)+ \lambda^\circledast  \cdot (p^\N -\lambda^\circledast) \right) \cdot [f(q_\lambda)] 
-3(p^\N -1) \cdot [S] \notag \\
& \ \ \ \  \ \  - (p^\N -1) \cdot  \mr{ch} (\mbR^1 f_{*}^{(\N)} ([\mcS ol (\nabla^\mr{ad})])) + p^\N (p^\N -1) \cdot   [D_s].  
\end{align}
This completes the proof of the assertion.
\end{proof}

\subsection{Genus formulas for  $(g, r) = (0, 4)$} \label{SS129}

Since $\overline{\mcM}_{0, 4}$ can be represented by a $k$-scheme,
it follows from the various assertions recalled in Section \ref{SS27} that 
$\mcO p_{\N, \rho, 0, 4}^{^\mr{Zzz...}}$ defines a (possibly empty) geometrically connected, smooth, and proper curve  over $k$. 
Under the assumption that $\mcO p_{\N, \rho, 0, 4}^{^\mr{Zzz...}} \neq \emptyset$,
we denote by 
\begin{align} \label{EQ213}
g_{\N, \rho, 0, 4}
\end{align}
the genus of  the dormant modular curve $\mcO p_{\N, \rho, 0, 4}^{^\mr{Zzz...}}$.
By applying Propositions \ref{Prop44} and \ref{Prp145}, one can conclude the following assertion, explicitly computing
the value $g_{\N, \rho, 0, 4}$.

\bt \label{Thm55}
Let $\rho := (\rho_i)_{i=1}^4$ be an element of $((\mbZ/p^\N \mbZ)^\times /\{ \pm 1 \})^4$, and suppose that $\mcO p_{\N, \rho, 0, 4}^{^\mr{Zzz...}} \neq \emptyset$.
Then, the following assertions hold.  
\begin{itemize}
\item[(i)]
The genus $g_{\N, \rho, 0, 4}$ of the dormant modular curve $\mcO p_{\N, \rho, 0, 4}^{^\mr{Zzz...}}$ is given by the following formula:
\begin{align} \label{EQ901}
g_{\N, \rho, 0, 4} &= 1 +  \frac{-3p^\N +1 +   \sum_{i=1}^4 \rho_i^\circledast (p^\N -\rho_i^\circledast)}{6 p^\N} \cdot \sharp (C_{\N, \rho, 0, 4})-  \sum_{\lambda \in C_{\N,\rho, 0, 4}}\frac{\lambda^\circledast (p^\N -\lambda^\circledast)}{2p^\N}
 \\
& \hspace{-3mm} \left(=1 +  \frac{-3p^\N +1 +   \sum_{i=1}^4 \rho_i^\circledast (p^\N -\rho_i^\circledast)}{2 p^\N} \cdot \sharp (C^0_{\N, \rho, 0, 4})-  \sum_{\lambda \in C_{\N,\rho, 0, 4}}\frac{\lambda^\circledast (p^\N -\lambda^\circledast)}{2p^\N}
\right).
\end{align}
\item[(ii)]
Suppose further that the inequality $2 \cdot \delta^{-1} (\rho_i) +1 \leq \frac{p^\N -3}{4}$ holds for every $i=1, \cdots, 4$ (which implies $\rho_i^\circledast = 2 \cdot \delta^{-1}(\rho_i) +1$ for every $i$ and $\lambda^\circledast = 2 \cdot \delta^{-1} (\lambda) +1$ for every $\lambda \in C_{\N, \rho, 0, 4}$).
Then, the 
formula  \eqref{EQ901} reads 
\begin{align} \label{EQ900}
g_{\N, \rho, 0, 4} &= 1 + \frac{-3 + \sum_{i=1}^4 \rho_i^\circledast}{6}  \cdot \sharp (C_{\N, \rho, 0, 4}) - \frac{1}{2} \cdot \sum_{\lambda \in C_{\N, \rho, 0, 4}} \lambda^{\circledast}  \\
& \hspace{-3mm} \left(=
1 +  \left(-1 + \sum_{i=1}^4 \delta^{-1} (\rho_i)\right) \cdot  \sharp (C_{\N, \rho, 0, 4}^0) -\sum_{\lambda \in C_{\N, \rho, 0, 4}} \delta^{-1}(\lambda)
\right).
\end{align}
\end{itemize}
\et
\begin{proof}
First, we shall consider assertion (i).
Under the identification $\psi_i = [D_s]$ and $ \Pi^* (c_1 (\mcT_{\overline{\mcM}_{0, 4}})) = -[D_s]$, the Chern character
$\mr{ch} (\mbR^1 f^{(\N)}_* (\mcS ol (\nabla^\mr{ad})))$
 of the vector bundle $\mbR^1 f^{(\N)}_* (\mcS ol (\nabla^\mr{ad}))$ satisfies 
\begin{align}
& \ \ \ \ \ \ \mr{ch} (\mbR^1 f^{(\N)}_* (\mcS ol (\nabla^\mr{ad}))) \\
& \stackrel{\eqref{EQee2}}{=}  -\mr{ch}( f_* ([F_{X/S}^{(\N)*}(\mcS ol (\nabla^\mr{ad}))])) +\mr{ch} ( f^{(\N)}_*([\mr{Coker}(\eta)])) \notag \\
& \stackrel{\eqref{EQ178}}{=} - (3-4 p^\N) \cdot [S] - D_s + \sum_{i=1}^4 \left(\frac{p^{2\N}-p^\N -2 p^\N \cdot \rho_i^\circledast + 2 \cdot \rho_i^{\circledast 2}}{2}  \cdot [D_s]\right) \\
& \ \ \ \  \ \ +  \sum_{\lambda \in C_{\N, \rho, 0, 4}}  \left(-p^\N(p^\N -1)+ \lambda^\circledast  \cdot (p^\N -\lambda^\circledast) \right) \cdot [f(q_\lambda)] 
-3(p^\N -1) \cdot [S]  \notag \\
& \ \ \ \ \ \ - (p^\N -1) \cdot  \mr{ch} (\mbR^1 f_{*}^{(\N)} ([\mcS ol (\nabla^\mr{ad})])) + p^\N (p^\N -1) \cdot   [D_s]  \notag \\
&  = p^\N \cdot [S] + \left(3p^\N (p^\N -1) -1 - \sum_{i=1}^4 \rho_i^\circledast (p^\N -\rho_i^\circledast) \right) \cdot [D_s]
 \notag \\
& \ \ \ \  +  \sum_{\lambda \in C_{\N, \rho, 0, 4}}  \left(-p^\N(p^\N -1)+ \lambda^\circledast  \cdot (p^\N -\lambda^\circledast) \right)\cdot [f(q_\lambda)] 
 - (p^\N -1) \cdot  \mr{ch} (\mbR^1 f_{*}^{(\N)} ([\mcS ol (\nabla^\mr{ad})])).  
\end{align}
It follows that
\begin{align}
  \mr{ch} (\mbR^1 f^{(\N)}_* (\mcS ol (\nabla^\mr{ad}))) 
 & = [S] +  \frac{3p^\N (p^\N -1) -1 - \sum_{i=1}^4 \rho_i^\circledast (p^\N -\rho_i^\circledast) }{p^\N} \cdot [D_s] \\
 &  \ \ \ \ + \sum_{\lambda \in C_{\N,\rho, 0, 4}}\left( -p^\N +1 +  \frac{\lambda^\circledast (p^\N -\lambda^\circledast)}{p^\N}\right) \cdot [f(q_\lambda)]. \notag
\end{align}
On the other hand,  
the relation between the degree of $\mbR^1 f^{(\N)}_* (\mcS ol (\nabla))$ and the value $g_{\N, \rho, 0, 4}$ is given by 
\begin{align}
\mr{deg}(\mbR^1 f^{(\N)}_* (\mcS ol (\nabla))) 
= 
\mr{deg} (\mcT_{S^\mr{log}/k})  =
-2  \cdot g_{\N, \rho, 0, 4} +2 - \sharp (C_{\N, \rho, 0, 4}), \notag
\end{align}
where the second equality follows from Theorem \ref{Th3}, (i).
Hence, we have
\begin{align} \label{EQ888}
g_{\N, \rho, 0, 4} &=  1 - \frac{\sharp (C_{\N, \rho, 0, 4})}{2}  -\frac{\mr{deg}(\mbR^1 f^{(\N)}_* (\mcS ol (\nabla^\mr{ad}))) }{2} \\
& = 1 - \frac{\sharp (C_{\N, \rho, 0, 4})}{2}  - \frac{1}{2} \Bigg( \frac{3p^\N (p^\N -1) -1 - \sum_{i=1}^4 \rho_i^\circledast (p^\N -\rho_i^\circledast) }{p^\N} \cdot \frac{\sharp (C_{\N, \rho, 0, 4})}{3} \\
& \ \ \ \ +  \sum_{\lambda \in C_{\N,\rho, 0, 4}}\left(-p^\N+1+\frac{\lambda^\circledast (p^\N -\lambda^\circledast)}{p^\N}\right)\Bigg) \notag \\
& = 1 +  \frac{-3p^\N +1 +   \sum_{i=1}^4 \rho_i^\circledast (p^\N -\rho_i^\circledast)}{6 p^\N} \cdot \sharp (C_{\N, \rho, 0, 4})-  \sum_{\lambda \in C_{\N,\rho, 0, 4}}\frac{\lambda^\circledast (p^\N -\lambda^\circledast)}{2p^\N}.
\end{align}
This completes the proof of assertion (i).

Next, to consider assertion (ii), we suppose that 
 $2 \cdot \lambda_i +1 \leq \frac{p^\N -3}{4}$ holds for every $i=1, \cdots, 4$,
where
 $\lambda_i := \delta_\N^{-1} (\rho_i)$.
Let $\rho_+ := (\rho_{+, i})_{i=1}^4$, where $\rho_{+, i} :=  \delta_{\N +1} (\lambda_i) \in (\mbZ/p^{\N+1} \mbZ)^\times /\{ \pm 1 \}$.
According to Proposition \ref{Prop2}, the natural projection $\mcO p_{\N +1, \rho', 0, 4}^{^\mr{Zzz...}} \rightarrow S$ is an isomorphism, which implies the equalities  $g_{\N, \rho, 0, 4} = g_{\N +1, \rho', 0, 4}$ and $\sharp (C_{\N, \rho, 0, 4}) = \sharp (C_{\N +1, \rho_+, 0, 4})$.
Moreover, by the assumption  of (ii),  we have   $\rho_i^\circledast = (\rho_i)_{\N +1}^\circledast$ for every $i$.
Hence,  \eqref{EQ888} and these equalities together induce the following sequence of equalities:
\begin{align}
0  & =  g_{\N, \rho, 0, 4} - g_{\N', \rho_{+}, 0, 4} \\
 & = 
1 +  \frac{-3p^\N +1 +   \sum_{i=1}^4 \rho_i^\circledast (p^\N -\rho_i^\circledast)}{6 p^\N} \cdot \sharp (C_{\N, \rho, 0, 4})-  \sum_{\lambda \in C_{\N,\rho, 0, 4}}\frac{\lambda^\circledast (p^\N -\lambda^\circledast)}{2p^\N} \\
& \ \ \ \ -1 -  \frac{-3p^{\N+1} +1 +   \sum_{i=1}^4 \rho_i^\circledast (p^{\N+1} -\rho_i^\circledast)}{6 p^{\N+1}} \cdot \sharp (C_{\N, \rho, 0, 4}) +   \sum_{\lambda \in C_{\N,\rho, 0, 4}}\frac{\lambda^\circledast (p^{\N+1} -\lambda^\circledast)}{2p^{\N+1}} \notag \\
& =  \frac{p-1}{6p^{\N +1}} \cdot \left( \sharp (C_{\N, \rho, 0, 4}) \cdot  \left(1 - \sum_{i=1}^4 \rho_i^{\circledast 2}\right)  + 3 \cdot \sum_{\lambda \in C_{\N, \rho, 0, 4}}\lambda^{\circledast 2} \right).
\end{align}
This implies
\begin{align}\label{EQ889}
 \sharp (C_{\N, \rho, 0, 4}) \cdot  \left(1 - \sum_{i=1}^4 \rho_i^{\circledast 2}\right)  + 3 \cdot \sum_{\lambda \in C_{\N, \rho, 0, 4}}\lambda^{\circledast 2} = 0.
\end{align}
By using this equality, we obtain
\begin{align}
g_{\N, \rho, 0, 4} & \stackrel{\eqref{EQ888}}{=}1 +  \frac{-3p^\N +1 +   \sum_{i=1}^4 \rho_i^\circledast (p^\N -\rho_i^\circledast)}{6 p^\N} \cdot \sharp (C_{\N, \rho, 0, 4})-  \sum_{\lambda \in C_{\N,\rho, 0, 4}}\frac{\lambda^\circledast (p^\N -\lambda^\circledast)}{2p^\N} \notag \\
& \stackrel{\eqref{EQ889}}{=}
1 + \frac{-3 + \sum_{i=1}^4 \rho_i^\circledast}{6}  \cdot \sharp (C_{\N, \rho, 0, 4}) - \frac{1}{2} \cdot \sum_{\lambda \in C_{\N, \rho, 0, 4}} \lambda^{\circledast}.
\end{align}
This proves the second  assertion.
\end{proof}

The following corollary establishes  a very rough estimation from above of the genus $g_{\N, \rho, 0, 4}$ induced from the above theorem.

\bco \label{Cor50}
Let us keep the notation and assumption  in the  previous theorem.
Then, the following inequality holds:
\begin{align}
g_{\N, \rho, 0 ,4} \leq   \frac{(p-1)p^{2\N-1}}{4}.
\end{align}
\eco
\begin{proof}
Since $\rho_i^\circledast (p^\N -\rho_i^\circledast) \leq \frac{1}{4} \cdot p^{2\N}$, we have
\begin{align}
g_{\N, \rho, 0, 4} &\stackrel{\eqref{EQ901}}{=} 1 +  \frac{-3p^\N +1 +   \sum_{i=1}^4 \rho_i^\circledast (p^\N -\rho_i^\circledast)}{6 p^\N} \cdot \sharp (C_{\N, \rho, 0, 4})-  \sum_{\lambda \in C_{\N,\rho, 0, 4}}\frac{\lambda^\circledast (p^\N -\lambda^\circledast)}{2p^\N} \\
& \leq  1 + \left( - \frac{1}{2} + \frac{1}{6p^\N}  + \frac{\sum_{i=1}^4 \rho_i^\circledast (p^\N -\rho_i^\circledast)}{6p^\N}\right) \cdot \sharp (C_{\N, \rho, 0, 4}) - \sum_{\lambda \in C_{\N, \rho, 0, 4}} \frac{1 \cdot (p^\N -1)}{2p^\N} \notag \\
& =  1 + \left( -1 + \frac{2}{3p^\N}  + \frac{\sum_{i=1}^4 \rho_i^\circledast (p^\N -\rho_i^\circledast)}{6p^\N} \right) \cdot \sharp (C_{\N, \rho, 0 ,4})
\\
& \leq 1 + \left( -1 + \frac{2}{3}\right) \cdot 3 +  \frac{4 \cdot \frac{p^{\N}}{2} \cdot \left(p^\N - \frac{p^\N}{2} \right)}{6p^\N} \cdot \frac{3(p-1)p^{\N -1}}{2} \notag \\
& = \frac{(p-1) p^{2\N -1}}{4},
\end{align}
where the last ``$\leq$" follows from 
Corollary \ref{Cor34}.
\end{proof}

On the other hand, we can apply Theorem \ref{Thm55} to compute the number of critical points of $\Pi_{\N, \rho, 0, 4}$.

\bco \label{Cor12}
Let $\rho$ be as in Theorem \ref{Thm55}, and
denote by $L_{\N, \rho, 0, 4}$ the number of the critical points of $\Pi_{\N, \rho, 0, 4}$.
Then, the value $L_{\N, \rho, 0, 4}$ 
 satisfies the following inequality:
 \begin{align} \label{Et543}
 L_{\N, \rho, 0, 4} \leq
 \frac{-p^\N +1 + \sum_{i=1}^4 \rho_i^\circledast (p^\N -\rho_i^\circledast)}{3p^\N} \cdot \sharp (C_{\N, \rho, 0, 4}) - \sum_{\lambda \in C_{\N, \rho, 0, 4}} \frac{\lambda^\circledast (p^\N -\lambda^\circledast)}{p^\N}.
 \end{align}
 If, moreover,   
  the inequality $2 \cdot \delta^{-1} (\rho_i) +1 \leq \frac{p^\N -3}{4}$ holds for every $i=1, \cdots, 4$, 
 then the inequality \eqref{Et543} reads 
\begin{align}
L_{\N, \rho, 0, 4} \leq
  \frac{-1 + \sum_{i=1}^4 \rho_i^\circledast}{3} \cdot \sharp (C_{\N, \rho, 0, 4})   -\sum_{\lambda \in C_{\N, \rho, 0, 4}} \lambda^\circledast.
\end{align}
\eco
\begin{proof}
The Riemann-Hurwitz formula implies the inequality
\begin{align}
L_{\N, \rho, 0, 4} &\leq 2 g_{\N, \rho, 0, 4} -2 + 2 \cdot  \mr{deg} (\Pi)  \\
&=  2 \cdot  \left( g_{\N, \rho, 0, 4} -1 + \frac{\sharp (C_{\N, \rho, 0, 4})}{3}\right) \notag \\
& \stackrel{\eqref{EQ901}}{=} \frac{-p^\N +1 + \sum_{i=1}^4 \rho_i^\circledast (p^\N -\rho_i^\circledast)}{3p^\N} \cdot \sharp (C_{\N, \rho, 0, 4}) - \sum_{\lambda \in C_{\N, \rho, 0, 4}} \frac{\lambda^\circledast (p^\N -\lambda^\circledast)}{p^\N},
\end{align}
which completes the proof of the first assertion.
Moreover, the second  assertion can be proved similarly by applying \eqref{EQ900}.
\end{proof}

\begin{exa} \label{Exa23}
Let us suppose that  $p=11$ and $\N =1$.
In this case, we have 
\begin{align}
\overline{1}^\circledast = 2, \hspace{5mm} \overline{2}^\circledast =4,  \hspace{5mm}\overline{3}^\circledast = 5,  \hspace{5mm}\overline{4}^\circledast = 3,  \hspace{5mm} \overline{5}^\circledast = 1
\end{align}
Now, let us take $\rho_0 := (\overline{3}, \overline{2}, \overline{4}, \overline{2}) \in ((\mbZ/11 \mbZ)^\times /\{  \pm 1 \})^4$.
By the definition of the set ``$C^q_{\N, \rho, 0, 4}$" ($q  =0, 1, \infty$), the following identities hold:
\begin{align}
C_{11, \rho_0, 0, 4}^\infty  = \{\overline{1}, \overline{2}, \overline{3}  \},
\hspace{5mm}
C_{11, \rho_0, 0, 4}^0 = \{\overline{1}, \overline{2}, \overline{3}  \},
\hspace{5mm}
C_{11, \rho_0, 0, 4}^1 = \{\overline{2}, \overline{3}, \overline{4}  \}.
\end{align}
In particular,  we have $\sharp (C_{1, \rho_0, 0, 4}) = 3 \cdot \sharp (C_{11, \rho_0, 0, 4}^0) = 9$.
Using these computations, 
one can apply Theorem \ref{Thm55} to obtain the following sequence of equalities:
\begin{align}
& \ \ \ \ g_{1, \rho_0, 0, 4}  \\
&= 1 +  \frac{-3p^\N +1 +   \sum_{i=1}^4 \rho_i^\circledast (p^\N -\rho_i^\circledast)}{6 p^\N} \cdot \sharp (C_{1, \rho_0, 0, 4})-  \sum_{\lambda \in C_{\N,\rho, 0, 4}}\frac{\lambda^\circledast (p^\N -\lambda^\circledast)}{2p^\N} \\
& = 1 + \frac{-3 \cdot 11 + 1 + \left( \overline{3}^\circledast  (11-\overline{3}^\circledast) + \overline{2}  (11-\overline{2}^\circledast) + \overline{4}^\circledast  (11 - \overline{4}^\circledast) + \overline{2}^\circledast  (11 -\overline{2}^\circledast)\right)}{6 \cdot 11} \cdot 9
 \notag \\
 & \ \  - \frac{\left( \overline{1}^\circledast (11- \overline{1}^\circledast) + \overline{2}^\circledast (11-\overline{2}^\circledast) + \overline{3}^\circledast (11 - \overline{3}^\circledast)\right) \cdot 2+ \left( \overline{2}^\circledast (11 -\overline{2}^\circledast) + \overline{3}^\circledast (11 -  \overline{3}^\circledast) + \overline{4}^\circledast (11 -\overline{4}^\circledast)\right)}{2 \cdot 11} \notag \\
& = 1 + \frac{-3 \cdot 11 + 1 + \left( 5\cdot 6 + 4 \cdot 7 + 3 \cdot 8 + 4 \cdot 7\right)}{6 \cdot 11} \cdot 9
  - \frac{\left(   2 \cdot 9 + 4 \cdot 7 + 5 \cdot 6\right) \cdot 2+ \left(4 \cdot 7 + 5 \cdot 6  + 3 \cdot 8 \right)}{2 \cdot 11} \notag \\
 & = 1 + \frac{117}{11} - \frac{117}{11} \notag \\
 & =1.
\end{align}
In particular, the dormant modular curve $\mcO p_{1, \rho_0, 0, 4}^{^\mr{Zzz...}}$ is not rational.

We remark that  the second genus formula \eqref{EQ900} cannot be applied in this case,
as 
 $2 \cdot \delta^{-1} (\overline{2}) +1 = 2 \cdot 3 +1 = 7   \not\leq \frac{11 -1}{2}$, which means that the condition described at the beginning of assertion (ii) in Theorem \ref{Thm55} does not hold.
In fact, the right-hand side of \eqref{EQ900} can be computed as follows:
\begin{align}
&  \  \ \ \ 
\left(1 + \frac{-3 + \sum_{i=1}^4 \rho_i^\circledast}{6}  \cdot \sharp (C_{\N, \rho, 0, 4}) - \frac{1}{2} \cdot \sum_{\lambda \in C_{\N, \rho, 0, 4}} \lambda^{\circledast} \right) \Biggl|_{(p, \N, \rho)= (11, 1, \rho_0)} \\
&   =1 +  \frac{-3 +  (\overline{3}^\circledast + \overline{2}^\circledast + \overline{4}^\circledast + \overline{2}^\circledast)}{6} \cdot 9
  \notag \\
& \ \ \ \ 
- \frac{1}{2} \cdot  \left( (\overline{1}^\circledast + \overline{2}^\circledast + \overline{3}^\circledast) 
+ (\overline{2}^\circledast + \overline{3}^\circledast + \overline{4}^\circledast) + (\overline{1}^\circledast + \overline{2}^\circledast + \overline{3}^\circledast)
 \right) \notag \\
 & = 1 +  \frac{-3 + (5 + 4 + 3 + 4)}{6} \cdot 9  - \frac{1}{2} \cdot \left( (2 + 4 + 5) + (4+ 5 + 3) + (2+4+ 5)\right) \notag \\
 & = 1 + \frac{13}{6} \cdot 9  - \frac{1}{2} \cdot 34\notag \\
 & = \frac{7}{2} \left( \neq 1 \right).
\end{align}
\end{exa}

\begin{exa} \label{Exa99}
Let $\M$ be a positive integer satisfying $m < \frac{p-1}{8}$.
Let $\rho := (\rho_i)_{i=1}^4$ be the  quadruple such 
that $\rho_1 = \rho_2 = \rho_3 = \rho_4 = \delta (\lambda)$, where
$\lambda := \frac{m(p^\N -1)}{p-1} \left(= \sum_{i=0}^{\N -1} m \cdot p^i \right)$.
Then, it is verified that 
\begin{align}
B_{\N, (\lambda, \lambda, \lambda, \lambda), 0, 4} = \left\{ \sum_{j=0}^{\N-1} a_j \cdot p^j \, \Bigg| \, 0\leq a_j \leq 2 m \ \text{for every $j$} \right\}.
\end{align}
In particular, we have $\sharp (C_{\N, \rho, 0, 4}^0) = (2 \M +1)^{\N}$.
The sum of the numbers in $B_{\N, (\lambda, \lambda, \lambda, \lambda), 0, 4}$ is  computed as follows:
\begin{align}
\sum_{\eta \in B_{\N, (\lambda, \lambda, \lambda, \lambda), 0, 4}} \eta &=
\sum_{(a_i)_{i=0}^{\N -1} \in \{0, 1, \cdots, 2 \M \}^\N} \sum_{j=0}^{\N -1} a_j \cdot p^j \\
& =  \sum_{j=0}^{\N -1} \sum_{(a_i)_{i=0}^{\N -1} \in \{0, 1, \cdots, 2 \M \}^\N} a_j \cdot p^j \notag \\
& =  \sum_{j=0}^{\N -1} (2m+1)^{\N -1} \cdot \frac{2\M \cdot (2\M +1)}{2} \cdot p^j \notag \\
& = \M \cdot (2\M +1)^{\N} \cdot \frac{p^\N -1}{p-1}.
\end{align}
Hence, by  \eqref{EQ900},  the genus $g_{\N, \rho, 0, 4}$ satisfies 
\begin{align}
g_{\N, \rho, 0, 4} & = 1 + \left(-1 + \sum_{i=1}^4 \delta^{-1}(\rho_i) \right) \cdot  \sharp (C_{\N, \rho, 0, 4}^0) - \sum_{\lambda \in C_{\N, \rho, 0, 4}} \delta^{-1}(\lambda) \\
& = 1 + \left(-1 + 4 m\cdot \frac{p^\N -1}{p-1} \right) \cdot (2\M +1)^{\N} - 3  m \cdot (2\M +1)^{\N} \cdot \frac{p^\N -1}{p-1} \notag \\
& = 1 + \left(m \cdot \frac{p^\N -1}{p-1} -1\right)(2\M +1)^\N.
\end{align}
\end{exa}

\begin{exa} \label{Exa30}
Suppose that $4 \mid p^\N -1$, and let $\rho := (\rho_i)_{i=1}^4$ be the quadruple defined by $\rho_1 = \rho_2 = \rho_3 = \rho_4 = \delta (\frac{p^\N -1}{4})$.
In particular, we have $\rho_i^\circledast = \frac{p^\N -1}{2}$.
It is verified that $B_{\N, (\lambda, \lambda, \lambda, \lambda), 0, 4} = B_\N$,  so the equality 
\begin{align} \label{Eq2098}
\left\{ \lambda^\circledast \, | \, \lambda \in C_{\N, \rho, 0, 4}^0 \right\} = \left\{ j \in \mbZ \, \Bigg| \, 1 \leq j \leq \frac{p^\N -1}{2}, p \nmid j\right\}
\end{align}
holds.
By using this equality, we have 
\begin{align}
\sum_{\lambda \in C^0_{\N, \rho, 0, 4}} \lambda^\circledast (p^\N -\lambda^\circledast) 
=\sum_{1 \leq j \leq \frac{p^\N -1}{2}, p\nmid j} j (p^\N -j) =  \frac{p^\N(p-1)(p^{2\N -1}+1)}{12}
\end{align}
Thus, the formula \eqref{EQ901} induces
\begin{align} \label{Eq891}
& \ \ \ \ g_{\N, \rho, 0, 4} \notag \\
&= 1 + \frac{-3p^\N +1 + 4 \cdot \frac{p^\N -1}{2} \cdot \left(p^\N -\frac{p^\N -1}{2}\right)}{6p^\N} \cdot \left(3 \cdot \frac{p^{\N -1} (p-1)}{2} \right) - \frac{3}{2p^\N} \cdot \frac{p^\N (p-1)(p^{2\N -1}+1)}{12} \notag \\
& = 1 + \frac{(p-1)(p^{2\N-1}-6p^{\N -1} -1)}{8}.
\end{align}
\end{exa}

\vspace{10mm}
\section{Asymptotic analysis for  dormant modular curves} \label{S5}

In this section, we investigate the asymptotic behavior of towers  of function fields associated to the projective system of dormant modular curves  of type $(0, 4)$ under level reduction. 
We introduce the notion of {\it asymptotic $\alpha$-goodness} for towers of function fields (cf. Definition \ref{Def57}), providing a quantitative measure of asymptotic richness comparable to that of the best classical example.
The consequence of our discussion is to estimate the geometric complexity of the tower under consideration by using this notion (cf. Theorems  \ref{Thm556}, \ref{Th45}).

Suppose that $k$ is a finite field of characteristic $p$.

\subsection{Asymptotically $\alpha$-good towers of function fields} \label{SS322}

For each function field $K$ over $k$, we denote by $P (K)$ the number of rational places and by $g (K)$ the genus of $K$ (i.e., the genus of the smooth projective curve associated to $K$).
Recall (cf. ~\cite[Definition 7.2.1]{Sti}) that a {\bf tower over $k$} (or a $k$-tower) is an infinite sequence $\mcK := (K_0, K_1, K_2, \cdots)$ of function fields $K_i/k$ satisfying the following conditions:
\begin{itemize}
\item[(a)]
$K_0 \subsetneq K_1 \subsetneq K_2 \subsetneq \cdots \subsetneq K_\ell \subsetneq \cdots$;
\item[(b)]
Each extension $K_{\ell +1}/K_\ell$ is finite and separable;
\item[(c)]
The genera satisfy $g (K_\ell) \to \infty$ as $\ell \to \infty$
\end{itemize}
(cf. ~\cite[Definition 7.2.1]{Sti}).

Now, let $\mcK := (K_0, K_1, K_2, \cdots)$ be a tower over $k$.
It is well-known that the limit $\lambda (\mcK) := \lim\limits_{\ell \to \infty} \frac{N (K_\ell)}{g (K_\ell)}$ does exist, and 
the resulting real number $\lambda (\mcK)$ is called the {\bf limit} of the $k$-tower $\mcK$.
The tower $\mcK$ is called {\bf asymptotically good} if 
it has a positive limit $\lambda (\mcK) > 0$ (cf. ~\cite[Definition 7.2.5]{Sti}).
The following generalizes this notion.

\bde \label{Def57}
Let $\alpha$ be a nonnegative  number.
We say that the tower $\mcK$ is {\bf asymptotically $\alpha$-good} 
if it  satisfies
\begin{align}
g (K_\N)^\alpha   = O (\sharp (P (K_\N))) \ (\N \to \infty).
\end{align}
That is to say,  there exist $C \in \mbR_{> 0}$ and $\N' \in \mbZ_{>0}$ such that $g (K_\N)^\alpha < C \cdot \sharp (P (K_\N))$ for every $\N \geq \N'$.
(The asymptotical $\alpha$-goodness can be defined even when $\mcK$ does not satisfy the third condition (c) described above.
Accordingly, in discussing this notion, we shall also consider   such sequences of function fields  $\mcK$.) 
\ede

If $\mcK$ is asymptotically $\alpha$-good, then it is also asymptotically $\alpha'$-good for any $\alpha' < \alpha$.
Hence, it makes sense to speak of the supremum 
\begin{align}
\Delta (\mcK) := \sup_{\alpha} \left\{\alpha \, | \, \text{$\mcK$ is asymptotically $\alpha$-good} \right\}.
\end{align}
It is verified that  $\mcK$ is asymptotically good if and only if it is asymptotically $1$-good, or equivalently, $\Delta (\mcK) \geq 1$.

\subsection{Towers of  dormant modular curves} \label{SS318}

Let $\N$ be a positive integer and $\rho$ an element of $((\mbZ/p^\N \mbZ)^\times /\{\pm 1\})^4$.
Denote by $K_{\N, \rho, 0, 4}$ the function field associated to the dormant modular curve $\mcO p_{\N, \rho, 0, 4}^{^\mr{Zzz...}}$ (hence $g (K_{\N, \rho, 0, 4}) = g_{\N, \rho, 0, 4}$).

\bpr \label{Prop31A}
The number $P (K_{\N, \rho, 0, 4})$
of rational places in $K_{\N, \rho, 0, 4}$ (i.e., the number of $k$-rational points of $\mcO p^{^\mr{Zzz...}}_{\N, \rho, 0, 4}$)
  satisfies the inequality 
\begin{align}
\sharp P (K_{\N, \rho, 0, 4}) \left(=\sharp (\mcO p^{^\mr{Zzz...}}_{\N, \rho, 0, 4} (k))\right) \geq \sharp (C_{\N, \rho, 0, 4}).
\end{align}
\epr
\begin{proof}
According to ~\cite[Theorem 8.27]{Wak20},
any dormant $\mr{PGL}_2^{(\N)}$-oper on a $3$-pointed projective line can be defined over $k$.
Since any dormant  $\mr{PGL}_2^{(\N)}$-oper $\msE^\spadesuit$ on a pointed curve in $\partial \overline{\mcM}_{0, 4}$ is obtained by gluing together such  dormant $\mr{PGL}_2^{(\N)}$-opers, 
$\msE^\spadesuit$ can be defined over $k$.
Moreover, recall that the projection $\Pi_{\N, \rho, 0, 4}$ is \'{e}tale over the points in $\partial \overline{\mcM}_{0, 4}$ (cf. ~\cite[Theorem C, (i)]{Wak20}), so we have 
\begin{align} \label{Eq367}
\sharp (\mcO p^{^\mr{Zzz...}}_{\N, \rho, 0, 4} (k)) &\geq \sharp (\partial \mcO p^{^\mr{Zzz...}}_{\N, \rho, 0, 4} (\overline{k})) \\
& = \sharp (\partial (\overline{\mcM}_{0, 4} (k))) \cdot \mr{deg}(\Pi_{\N, \rho, 0, 4}) \notag \\
&=  3 \cdot \left( \frac{1}{3} \cdot \sharp (C_{\N, \rho, 0, 4})\right) \notag \\
&= \sharp (C_{\N, \rho, 0, 4}). \notag
\end{align}
This completes the proof of the assertion.
\end{proof}

\bt \label{Prop31B}
Let us keep the above notation, and
let $s$ be an integer with $1 \leq s \leq \frac{p-3}{2}$.
Suppose that $\rho$ lies in $\delta (D_{s, \N})$.
Then, the following  inequality holds:
\begin{align} \label{EQ906}
P (K_{\N, \rho, 0, 4}) \geq 3 \cdot \left(\frac{4p}{p-1} \right)^{\frac{\log s}{2 \log p}} \cdot g_{\N, \rho, 0, 4}^{-\frac{\log s}{2 \log p}} .
\end{align}
\et
\begin{proof}
By Corollary \ref{Cor50}  and Lemma \ref{Lem22}, we have 
\begin{align}
P (K_{\N, \rho, 0, 4}) \cdot g_{\N, \rho, 0, 4}^{-\frac{\log s}{2 \log p}} \geq 3 \cdot s^\N \cdot  \left( \frac{(p-1)p^{2\N -1}}{4}\right)^{-\frac{\log s}{2\log p}} 
= 3 \cdot \left(\frac{4p}{p-1} \right)^{\frac{\log s}{2 \log p}},
\end{align}
which completes the proof  of the assertion.
\end{proof}

Next, let $\mbZ_p^\times / \{ \pm 1 \}$ be  the set of equivalence classes of elements $a \in \mbZ_p^\times \left(= \mbZ_p \setminus p\mbZ_p \right)$, in which $a$ and $-a$ are identified.
Note that $\mbZ_p^\times / \{ \pm 1 \} \cong \varprojlim_{\N} \left((\mbZ/p^\N \mbZ)^\times / \{ \pm 1 \}\right)$.

We fix an element $\rho$ of $(\mbZ_p^\times /\{ \pm 1 \})^4$.
For  a positive integer $\N$, we denote by $\rho_\N$ the element of $((\mbZ/p^\N \mbZ)^\times /\{ \pm 1 \})^4$ induced from $\rho$ via the natural quotient $\mbZ_p^\times \twoheadrightarrow (\mbZ/p^\N \mbZ)^\times$.
Then, we obtain a collection of function fields 
\begin{align} \label{EQ922}
\mcK_{\N, \rho, 0, 4} := (K_{1, \rho_1, 0, 4},  K_{2, \rho_2, 0, 4},  \cdots ,  K_{\N, \rho_\N, 0, 4},   \cdots)
\end{align}
equipped with  a chain of inclusions  $K_{1, \rho_1, 0, 4} \subseteq K_{2, \rho_2, 0, 4} \subseteq  \cdots \subseteq  K_{\N, \rho_\N, 0, 4} \subseteq \cdots$  associated to \eqref{Eq12}.
Since the projection $\Pi_{\N, \rho_\N, 0, 4}$ is generically \'{e}tale for every $\N$, each extension $K_{\N+1, \rho_{\N+1}, 0, 4}/K_{\N, \rho_\N, 0, 4}$  turns out to be  finite and separable.
Hence, under the assumption that    $\lim\limits_{\N \to \infty} g_{\N, \rho_\N, 0, 4} = \infty$,
the collection $\mcK_{\N, \rho, 0, 4}$ defines a $k$-tower.

\bt \label{Th45}
Let $s$ be an integer with $1 \leq s \leq \frac{p-3}{2}$.
(Note that the subsets $\delta (D_{s, \N}) \subseteq ((\mbZ/p^\N \mbZ)^\times /\{ \pm 1 \})^4$ forms a projective system, so we obtain its projective limit $\varprojlim\limits_{\N}\delta (D_{s, \N})$, considered as a subset of $(\mbZ_p^\times /\{ \pm 1 \})^4$).
Also, $\rho$ be an element of $\varprojlim\limits_{\N}\delta (D_{s, \N})$.
Then,  the $k$-tower $\mcK_{\N, \rho, 0, 4}$ is asymptotically $\frac{\log s}{2 \log p}$-good.
\et
\begin{proof}
The assertion follows from  Theorem \ref{Prop31B}.
\end{proof}

Next, let us briefly consider how large the value  $\Delta (\mcK_{\N, \rho, 0, 4})$ can take when we change the choice of $\rho$.

\bt \label{Thm556}
 Let us keep the above notation, and suppose that $4 \mid p-1$.
 Then, the following inequality holds:
\begin{align} \label{Eq3009}
\sup_{\rho \in (\mbZ_p^\times /\{ \pm 1 \})^4}\Delta (\mcK_{\N, \rho, 0, 4}) \geq \frac{1}{2}.
\end{align}
\et
\begin{proof}
Since $\frac{p^\N -1}{4} =p^0 \cdot \frac{p-1}{4}+ p^1 \cdot \frac{p-1}{4} + \cdots +  p^{\N-1} \cdot \frac{p-1}{4}$, our assumption implies $4 \mid p^\N -1$ for every $\N \in \mbZ_{>0}$.
Now, let us take $\rho := (\rho_i)_{i=1}^4$ to be the  element of $(\mbZ_p^\times / \{ \pm 1 \})^4$ such that $\rho_i$'s  are the class represented by $-\frac{1}{4}$ .
In particular, 
$\rho_\N = \left(\delta_\N  \left( \frac{p^\N-1}{4}\right), \delta_\N\left( \frac{p^\N-1}{4}\right), \delta_\N\left( \frac{p^\N-1}{4}\right), \delta_\N\left( \frac{p^\N-1}{4}\right)  \right)$.
According to \eqref{Eq2098}, we have $\sharp (C_{\N, \rho_\N, 0, 4}) = 3 \cdot \sharp (C_{\N, \rho, 0, 4}^0) = 3 \cdot \frac{p^{\N-1}(p-1)}{2}$. 
It follows that
\begin{align}
& \ \ \ \ P (K_{\N, \rho_\N, 0, 4}) \cdot g_{\N, \rho_\N, 0, 4}^{- \frac{1}{2}} \\
 &\geq \sharp (C_{\N, \rho_\N, 0, 4}) \cdot g_{\N, \rho_\N, 0, 4}^{- \frac{1}{2}}  \notag \\
& = 3 \cdot \sharp (C_{\N, \rho, 0, 4}^0)  \cdot  \left(1 + \frac{(p-1)(p^{2\N-1}-6p^{\N-1} -1)}{8} \right)^{-\frac{1}{2}} \notag \\
&  = 3 \cdot \frac{p^{\N-1}(p-1)}{2}  \cdot  \left(1 + \frac{(p-1)(p^{2\N-1}-6p^{\N-1} -1)}{8} \right)^{-\frac{1}{2}} \notag \\
& \geq 3 \cdot \frac{p^\N}{3} \cdot \left(\frac{p^{2\N}}{4}\right)^{-\frac{1}{2}} \notag \\
& =2,
\end{align}
where the first ``$\geq$"  follows from \eqref{Eq367} and the first   ``$=$" follows from \eqref{Eq891}.
This implies $\Delta (\mcK_{\N, \rho, 0, 4}) \geq  \frac{1}{2}$, which completes the proof of the assertion.
\end{proof}

\begin{rem}
We conjecture that the left-hand  side of \eqref{Eq3009} coincides with $\frac{1}{2}$ (for every odd prime $p$), but we have not yet been able to prove this claim. 
Note that, in addition to the case 
where $\rho = \left(\overline{\left(-\frac{1}{4}\right)}, \overline{\left(-\frac{1}{4}\right)}, \overline{\left(-\frac{1}{4}\right)}, \overline{\left(-\frac{1}{4}\right)} \right)$
 (cf.  the proof of Theorem \ref{Thm556}), 
 there are  many choices  of  $\rho$ such that  $\Delta (\mcK_{\N, \rho, 0, 4})$ is greater than or equal to  $\frac{1}{2}$.
 In fact, 
 $\mcK_{\N, \rho, 0, 4}$ is asymptotically $\frac{1}{2}$-good
 under the assumption that
 there exists a positive number $C$ satisfying $\sharp (C_{\N, \rho_\N, 0, 4}) \geq C \cdot p^\N$ for any sufficiently large $\N$.
 This claim is verified from the following sequence of inequalities defined for such $\N$'s:
 \begin{align}
 P (K_{\N, \rho_\N, 0, 4}) \cdot g_{\N, \rho_\N, 0, 4}^{-\frac{1}{2}} \geq 
 \sharp (C_{\N, \rho_\N, 0, 4}) \cdot \left(\frac{p^{2\N}}{4} \right)^{-\frac{1}{2}} \geq C \cdot p^\N \cdot\left(\frac{p^{2\N}}{4} \right)^{-\frac{1}{2}} =2 \cdot C,
 \end{align}
 where the first ``$\geq$" follows from \eqref{Eq367} and Corollary \ref{Cor50}.
  For example,   we see that the case where  $\rho = \left(\overline{\left(-\frac{1}{4}\right)}, \overline{\left(-\frac{1}{4}\right)}, \overline{\left(-\frac{1}{4}\right)}, \overline{\left(-\frac{5}{4}\right)} \right)$ satisfies this assumption, so 
  the  tower  $\mcK_{\N, \rho, 0, 4}$ is asymptotically $\frac{1}{2}$-good. 
\end{rem}

\vspace{10mm}
\section{Appendix: Heun's differential equations} \label{SA}

In this appendix, we discuss the relationship between dormant $\mr{PGL}_2$-oper on a $4$-pointed projective line and Heun's differential operators.

Let us fix an element $t \in k \setminus \{ 0, 1\}$.
 The  $4$-pointed projective line $\msP_t := (\mbP^1, \{0, 1, \infty, t \})$ over $k$ determines a log structure on $\mbP^1$; the resulting log scheme is denoted by $\mbP^{1 \mr{log}}$.

Recall that {\bf Heun's equation} is the differential equation 
 $\msD_{\vec{\alpha}} (f) = 0$ associated to 
a differential operator 
\begin{align}
 \msD_{\vec{\alpha}} := \frac{d^2}{dx^2} + \left(\frac{\gamma}{x}  + \frac{\delta}{x-1} + \frac{\epsilon}{x-t}\right) \frac{d}{dx} + \frac{\alpha \beta x  -q}{x (x-1) (x-t)}
\end{align}
on $\mbP^1 \setminus \{0, 1,  \infty, t \} \left(=\mr{Spec} (k [x, \frac{1}{x (x-1)(x-t)}])\right)$ defined for  a collection $\vec{\alpha} := (\alpha, \beta, \gamma, \delta, \epsilon, q) \in k^6$ with $\alpha + \beta +1 = \gamma + \delta + \epsilon$.
We shall call such a differential operator  {\bf Heun's differential operator}.
Note that
 $D_{\vec{\alpha}}^\clubsuit := dx^{\otimes 2} \otimes \msD_{\vec{\alpha}}$ defines a $2$nd differential operator $\mcO_{\mbP^1} \rightarrow \Omega^{\otimes 2}$ with unit principal symbol under the natural identification $\Omega^{\otimes 2} \otimes \mcT^{\otimes 2} = \mcO_{\mbP^1}$.
That is, $D^\clubsuit_{\vec{\alpha}}$ defines a $(2, 1, \mcO_{\mbP^1})$-projective connection on $\mbP^{1\mr{log}}/k$ in the sense of ~\cite[Definition 4.26]{Wak5}.
We shall write
\begin{align}
\msE^\spadesuit_{\vec{\alpha}}
\end{align}
for the corresponding $\mr{PGL}_2$-oper on $\msP_t$ via the correspondence in ~\cite[Theorem D and Equation (554)]{Wak5}.

If we write $y:= x-1$, $z := 1/x$, $w := x-t$, then the following equalities hold:
\begin{align} \label{EQQ288}
D_{\vec{\alpha}}^\clubsuit &= \left( \frac{dx}{x}\right)^{\otimes 2} \otimes \left( \partial_x^2 + \left(-1 + \gamma + \frac{\delta x}{x-1} + \frac{\epsilon x}{x-t} \right)  \partial_x + \frac{\alpha \beta x^2 - q x}{(x-1) (x-t)}\right) \\
& = 
\left( \frac{dy}{y}\right)^{\otimes 2} \otimes \left( \partial_y^2 + \left(1 - \delta + \frac{\gamma y}{y+1} - \frac{\epsilon y}{y + 1-t} \right) \partial_y + \frac{\alpha \beta y^2 + (\alpha \beta - q)y}{(y+1) (y+ 1-t)}\right) \notag \\
& = 
\left(\frac{dz}{z} \right)^{\otimes 2} \otimes \left( \partial_z^2 + \left(1-\gamma + \frac{\delta}{z-1} + \frac{\epsilon}{t (z-1/t)} \right)  \partial_z + \frac{-qz + \alpha \beta}{t(z-1) (z-1/t)} \right) \notag \\
& =
\left(\frac{dw}{w} \right)^{\otimes 2} \otimes \left( \partial_w^2 + \left(1 - \epsilon + \frac{\gamma w}{w+t}- \frac{\delta w}{w -1+t} \right)  \partial_w + \frac{\alpha \beta w^2 + (\alpha \beta t - q)w}{(w +t) (w-1+t)}\right), \notag
\end{align} 
where  $\partial_v := v  \frac{d}{dv}$ for $v \in \{ x, y , z, w \}$.
In particular,  the radius  of $\msE^\spadesuit_{\vec{\alpha}}$ at $0, 1, t$, and $\infty$ are, respectively, given by 
$\overline{\left(\frac{1- \gamma}{2}\right)}$, $\overline{\left(\frac{1-\delta}{2}\right)}$, $\overline{\left(\frac{\alpha - \beta}{2}\right)}$, and $\overline{\left(\frac{1-\epsilon}{2}\right)}$.

\bpr \label{Pro559}
The assignment $\msD_{\vec{\alpha}} \mapsto \msE_{\vec{\alpha}}^\spadesuit$ induces a surjective map
\begin{align} \label{Eq992}
\begin{pmatrix} \text{the set of Heun's differential operators} \\ \text{having a full set of root functions}\end{pmatrix}
\xrightarrow{\sim}
\begin{pmatrix} \text{the set of isomorphism classes} \\ \text{of dormant $\mr{PGL}_2$-opers on $\msP_t$}\end{pmatrix}.
\end{align}
Moreover, 
for two collections $\vec{\alpha}_i := (\alpha_i, \beta_i, \gamma_i, \delta_i, \epsilon_i, q_i) \in k^6$ ($i=1, 2$) with $\alpha_i +\beta_i +1 = \gamma_i + \delta_i + \epsilon_i$,
the associated Heun's differential operators  $\msD_{\vec{\alpha}_1}$,  $\msD_{\vec{\alpha}_2}$ determine the same  dormant $\mr{PGL}_2$-oper under \eqref{Eq992} if and only if the following equalities hold:
\begin{align}
(\alpha_1 - \beta_1)^2 &= (\alpha_2 - \beta_2)^2, \\
 (1-\gamma_1)^2 &= (1-\gamma_2)^2, \\
 (1 -\delta_1)^2 &= (1-\delta_2)^2, \\
  (1-\epsilon_1)^2 &= (1-\epsilon_2)^2, \\
   t\gamma_1 \delta_1 + \gamma_1 \epsilon_1 + 2q_1 &= t\gamma_2 \delta_2 + \gamma_2 \epsilon_2 + 2q_2.
\end{align}
\epr
\begin{proof}
First, we shall prove the surjectivity of  \eqref{Eq992}.
Let us take a dormant $\mr{PGL}_2$-oper $\msE^\spadesuit$ on $\msP_t$, and choose a dormant $2^{(1)}$-theta characteristic $\vartheta$ of $\msP_t$ whose underlying line bundle coincides with $\mcO_{\mbP^1}$.
According to  ~\cite[Theorem D]{Wak5},  $\msE^\spadesuit$ 
arises from a $(2, 1, \vartheta)$-projective connection $D^\clubsuit$ with a full set of  root functions (cf. ~\cite[Definition 4.64]{Wak5}).
After replacing $\vartheta$ with another, 
the local expression of  this differential operator around  $0 \in \mbP^1$ can be given by
\begin{align}  \label{EQgght}
\left(\frac{dx}{x} \right)^{\otimes 2} \otimes \left(\partial_x^2 + \left(r + \frac{\delta}{x-1} + \frac{\epsilon t}{x-t} \right) \partial_x + \frac{s x -q}{(x-1)(x-t)}\right)\end{align}
for some $\delta, \epsilon, q, r, s \in k$.
We set $\gamma := 1 + r -\delta -\epsilon$.
Also, let $\alpha, \beta$ be the roots of the characteristic polynomial $h(T)  \in k [T]$ of  $D^\clubsuit$ at $\infty \in \mbP^1$.
It follows from the expression \eqref{EQgght} that 
\begin{align}
(T -\alpha)(T-\beta) =h (T) = T^2 + (1 -\gamma - \delta - \epsilon)T + s.
\end{align}
This implies $\alpha + \beta +1 = \gamma + \delta + \epsilon$ and $s = \alpha \beta$.
Thus, the differential operator $D^\clubsuit$ coincides with $D^\clubsuit_{\vec{\alpha}}$, where $\vec{\alpha} := (\alpha, \beta, \gamma, \delta, \epsilon, q)$.
This completes the surjectivity of \eqref{Eq992}.

Next, we shall prove the second assertion.
Let $\vec{\alpha}_i$ ($i=1, 2$) be as in the statement of this proposition, and let $\vec{\alpha} := (\alpha, \beta, \gamma, \delta, \epsilon, q)$ be either $\vec{\alpha}_1$ or $\vec{\alpha}_2$.
By the first equality in \eqref{EQQ288},
the connection defining the $(\mr{GL}_2, 1, \mcO_{\mbP^1})$-oper (in the sense of ~\cite[Definition 4.27, (i)]{Wak5}) associated to  $D_{\vec{\alpha}}$ is of the form   $d + \frac{dx}{x} \otimes \begin{pmatrix} 0 & U  \\ 1 & V\end{pmatrix}$,
where $U := -\frac{\alpha \beta x^2 -qx}{(x-1)(x-t)}$ and $V = -\left( -1 + \gamma + \frac{\delta x}{x-1}+ \frac{\epsilon x}{x-t}\right)$.
After possibly tensoring with a flat line bundle and a suitable gauge transformation,
this can be transformed into the connection $d + \frac{dx}{x} \otimes \begin{pmatrix} 0 & U +  \frac{1}{4}\cdot V^2 - \frac{x}{2} \cdot  \frac{d}{dx}V \\ 0 & 0\end{pmatrix}$.
Let us set $A \left(\text{or} \ A_{\vec{\alpha}} \right):= \alpha - \beta$, $C  \left(\text{or} \ C_{\vec{\alpha}} \right):= 1-\gamma$, $D  \left(\text{or} \ D_{\vec{\alpha}} \right) := 1 -\delta$, and $E  \left(\text{or} \ E_{\vec{\alpha}} \right) := 1-\epsilon$.
Then,
we have 
\begin{align}
 U +  \frac{1}{4}\cdot V^2 - \frac{x}{2} \cdot  \frac{d}{dx}V = \frac{N_{\vec{\alpha}, 4} x^4 + N_{\vec{\alpha}, 3}x^3 + N_{\vec{\alpha}, 2}x^2 + N_{\vec{\alpha}, 1}x + N_{\vec{\alpha}, 0}}{4(x-1)^2(x-t)^2},
\end{align}
where 
\begin{align}
N_{\vec{\alpha}, 0} &= C^2 t^2,\\
N_{\vec{\alpha}, 1} &=2t \left(-C^2t-C^2+Ct + C -CDt  -CE + Dt +E  +2q -t-1 \right), \\
N_{\vec{\alpha}, 2} &= A^2t +C^2t^2 + 3C^2t + C^2 -2Ct^2 - 4Ct -2C+ 2CDt^2 + 2 CDt  + 2CEt + 2CE \\ & \ \ \ \ 
  + D^2t^2 -D^2t  -2Dt^2 -2Dt  -2Et-2E -E^2t + E^2  -4qt -4q +t^2 +6t+1, \\
N_{\vec{\alpha}, 3} &=-A^2(t+1) -C^2t -C^2+2Ct+2C -2CDt  -2CE -D^2t + D^2 + 2Dt + E^2t \\ & \ \ \ \ -E^2+2E +4q -2t -2, \\
N_{\vec{\alpha}, 4} &=A^2.
\end{align}
Observe that 
\begin{align}
\left( N'_{\vec{\alpha}, 2} := \right) &  \   N_{\vec{\alpha}, 2} - \frac{(t+1)N_{\vec{\alpha}, 1}}{t}  =tA^2  - (t^2+t+1)C^2+ t(t-1)D^2 + (1-t)E^2 -(t-1)^2, \\
\left( N'_{\vec{\alpha}, 3} := \right) & \  N_{\vec{\alpha}, 3} - \frac{N_{\vec{\alpha}, 1}}{t}  = (t+1)(C^2-A^2) + (t-1)(E^2-D^2). 
\end{align}
Hence, we have 
\begin{align}
\msE^\spadesuit_{\vec{\alpha}_1} \cong \msE^\spadesuit_{\vec{\alpha}_2}
\ &\Longleftrightarrow \ N_{\vec{\alpha}_1, j} = N_{\vec{\alpha}_2, j}  \ \text{for any} \ j=0, 1,2,3,4 \\
& \Longleftrightarrow \ N_{\vec{\alpha}_1, j} = N_{\vec{\alpha}_2, j}  \ \text{for} \ j=0, 1,4 \ \text{and} \ N'_{\vec{\alpha}_1, 2}= N'_{\vec{\alpha}_2, 2}
\ \text{and} \ N'_{\vec{\alpha}_1, 3}= N'_{\vec{\alpha}_2, 3} \\
&  \Longleftrightarrow \    A^2_{\vec{\alpha}_1} =   A^2_{\vec{\alpha}_2}, C^2_{\vec{\alpha}_1} = C^2_{\vec{\alpha}_2}, D^2_{\vec{\alpha}_1} = D^2_{\vec{\alpha}_2}, E^2_{\vec{\alpha}_1} = E^2_{\vec{\alpha}_2}, \ \text{and} \ N_{\vec{\alpha}_1, 1} = N_{\vec{\alpha}_2, 1}.
\end{align}
Since $\frac{N_{\vec{\alpha}, 1}}{2t} = -C^2t -C^2 + \left(t\gamma \delta + \gamma \epsilon +2q \right)$,
the equality $N_{\vec{\alpha}_1, 1} = N_{\vec{\alpha}_2, 1}$ is equivalent to the equality  $t\gamma_1 \delta_1 + \gamma_1 \epsilon_1 +2q_1 = t\gamma_2 \delta_2 + \gamma_2 \epsilon_2 +2q_2$ under the assumption that $F_{\vec{\alpha}_1}^2 =F^2_{\vec{\alpha}_2}$ for $F \in \{A, C, D, E \}$.
This completes the proof of this proposition.
\end{proof}

\end{document}